\documentclass[12pt,leqno]{amsart}
\usepackage{amssymb,mathrsfs,euscript}
\usepackage{mathrsfs,euscript,amssymb,amsmath,bbold,amsfonts,amsxtra,
placeins,graphicx,verbatim,stmaryrd}
\usepackage[all]{xy}
\usepackage[hypertex]{hyperref}
\textheight9in  \def\DATE{January 5, 2014}
\def\today{\DATE}
\catcode`\@=11
\def\@evenfoot{\rule{0pt}{20pt}[\today] \hfill}
\def\@oddfoot{\rule{0pt}{20pt}\hfill [\today]}

\swapnumbers
\newtheorem{theorem}{Theorem}[section]
\newtheorem{corollary}[theorem]{Corollary}
\newtheorem{principle}[theorem]{Principle}

\newtheorem{lemma}[theorem]{Lemma}
\newtheorem{proposition}[theorem]{Proposition}

\newtheorem*{claim}{Claim}
\newtheorem*{theoremA}{Theorem~A}
\newtheorem*{theoremB}{Theorem~B}
\newtheorem*{theoremC}{Theorem~C}
\newtheorem*{theoremAplus}{Theorem~A$_+$}
\newtheorem*{theoremBplus}{Theorem~B$_+$}
\newtheorem*{theoremCplus}{Theorem~C$_+$}
\newtheorem*{theoremD}{Theorem~D}

\theoremstyle{definition}
\newtheorem{example}[theorem]{Example}

\newtheorem{definition}[theorem]{Definition}

\newtheorem{prop}[theorem]{Proposition}
\newtheorem{defi}[theorem]{Definition}
\newtheorem{rem}[theorem]{Remark}

\numberwithin{equation}{section}
\DeclareMathOperator{\Hom}{Hom}

\DeclareMathOperator{\im}{Im}
\DeclareMathOperator{\Ker}{Ker}
\DeclareMathOperator*{\colim}{colim}

\def\CF{\hbox{\,$\tilde*$\,}}
\def\ground{\mathbb k}

\def\Sp{\mathfrak s}
\def\L{\hat{\mathscr L}}
\def\Lie{{\mathscr L}}
\def\a{\mathfrak a}
\def\ff{{\mathfrak f}}
\def\hatff{{\hat{\mathfrak f}}}
\def\g{\mathfrak g}
\def\h{\mathfrak h}
\def\J{J}
\def\LL{{\mathbb L}}\def\hoLdc{{\hbox{$\fQ$\rm-ho}}\L^{\it{dc}}}
\def\epi{\twoheadrightarrow}
\def\Cf{\hat{\mathcal C}}
\def\Cfun{{\mathcal C}}
\def\scup{\sqcup}\def\hoC{\hbox{ho$\catC$}}\def\hounder{\hbox{ho$(*/\catC)$}}
\def\algs{\mathscr A}\def\hoover{\hbox{ho$(\catC/*)$}}
\def\A{\hat{\mathscr A}}
\def\MCmod{\mathscr {MC}}
\def\fa{{\mathfrak a}}
\def\bfk{{\mathbf k}}

\def\id{\operatorname{id}}
\def\hatot{{\hbox{\hskip.2em $\widehat \otimes \hskip .02em$}}}
\def\Der{\operatorname{Der}}

\def\ot{{\otimes}}\def\U{{\mathcal U}}\def\uu{{\mathfrak u}}\def\vv{{\mathfrak v}}
\def\Prim{{\mathcal P}}\def\Ext{\operatorname{Ext}}\def\Tor{\operatorname{Tor}}
\def\bbQ{{\mathbb Q}}\def\oh{{\overline \h}}

\title{Disconnected rational homotopy theory}

\author{Andrey Lazarev and Martin Markl}
\thanks{The first author was partially supported by the EPSRC grant
  EP/J008451/1. The second author was supported by the Eduard \v Cech
  Institute P201/12/G028, the EPSRC grant EP/I014012/1 and RVO: 67985840.}
\address{Mathematical Institute of the Academy, {\v Z}itn{\'a} 25,
         115 67 Prague 1, The Czech Republic}
\address{MFF UK, 186 75 Sokolovsk\'a 83, Prague 8, The Czech Republic}
\email{markl@math.cas.cz}

\address{Department of Mathematics and Statistics
University of Lancaster
Lancaster LA1 4YF}
\email{a.lazarev@lancaster.ac.uk}

\def\hatX{\widehat X}\def\id{1\!\!1}\def\Lhat{{\hat\Lfun}}
\def\good{{D^I}}\def\cic{cohomologically invertible cocycle}
\def\vlra{{\hbox{$-\hskip-1mm-\hskip-2mm\longrightarrow$}}}
\def\cases#1#2#3#4{
                  \left\{
                         \begin{array}{ll}
                           #1,\ &\mbox{#2}
                           \\
                           #3,\ &\mbox{#4}
                          \end{array}
                   \right.
}
\def\coprod{\bigcup}\def\colim{\mathop{\rm colim}\displaylimits}
\def\calC{{\EuScript C}}
\def\norep{A_1 \timesred A_2}\def\Mult{{\rm coProd}}\def\Tlum{{\rm Prod}}
\def\catC{{\EuScript C}}\def\loc#1#2{#1[#2^{-1}]}
\def\nilpsimpcon{\hbox{$\fNQ$-$\SSet^c$}}
\def\honilpsimpcon{\hbox{$\fNQ$-ho$\SSet^c$}}
\def\honilpsimpdiscon{\hbox{$\fNQ$\rm-ho$\SSet^{\it dc}$}}
\def\ratcat{\hbox{$\fQ$-$\BGbez^c$}}
\def\horatcat{\hbox{$\fQ$-ho$\BGbez^c$}}
\def\horatcatdc{\hbox{$\fQ$\rm-ho$\BGbez^{\it dc}$}}
\def\ratcatdc{\hbox{$\fQ$-$\BGbez^{\it dc}$}}
\def\horathincat{\hbox{$\fQ$\rm-ho$\Hibez^{\it c}$}}
\def\horathincatdc{\hbox{$\fQ$\rm-ho$\Hibez^{\it dc}$}}
\def\rathincatdc{\hbox{$\fQ$-$\Hibez^{\it dc}$}}
\def\tf{{\widetilde f}}
\def\Min{{\mathscr{M}}}\def\Lfreecompl{\hat{\mathbb L}}
\def\tu{{\widetilde u}}
\def\timesred{{\hskip -.03em\times \hskip -.03em}}
\def\tf{{\widetilde f}}\def\epi{\twoheadrightarrow}
\def\hoHindcplus{{\rm ho}{\algs_+^{\it dc}}}
\def\Hindc{{\algs^{\it dc}}}\def\Hindcbar{\bar{{\algs}^{\it dc}}}
\def\hoHindc{{\rm ho}{\algs^{\it dc}}}\def\Hindcplus{{\algs_+^{\it dc}}}
\def\hoHindcbar{{\rm ho}\bar{{\algs}^{\it dc}}}
\def\hoHindcplus{{\rm ho}{\algs^{\it dc}_+}}
\def\Hinc{{\algs^{\it c}}}\def\Hincplus{{\algs_+^{\it c}}}
\def\hoHinc{{\rm ho}{\algs^{\it c}}}\def\hoHincplus{{\rm ho}{\algs_+^{\it c}}}
\def\bfk{{\mathbb k}}\def\bfQ{{\mathbb Q}}\def\Hin{{\tt \algs}}\def\Hi{\Hin}
\def\hoHin{{\rm ho}{\tt \algs}}\def\algs{\mathscr{A}}\def\ssets{\mathscr{S}}
\def\BG{{{\algs}_{\geq 0}}}
\def\BGdc{{{\algs}_{\geq 0}^{\it dc}}}
\def\hoBGdc{{{\rm ho}{\algs}_{\geq 0}^{\it dc}}}
\def\BGc{{{\algs}_{\geq 0}^{\it c}}}
\def\BGdcbar{{\bar{\algs}_{\geq 0}^{\it dc}}}
\def\hoBGc{{{\rm ho}{\algs}_{\geq 0}^{\it c}}}
\def\hoBGdcbar{{{\rm ho}\bar{\algs}_{\geq 0}^{\it dc}}}
\def\BGbez{{\algs}_{\geq 0}}\def\Hibez{{\algs}}
\def\SSet{{\tt \ssets}}\def\SSetc{{\ssets^c}}\def\SSetdc{{\ssets^{\it
      dc}}}\def\SSetcplus{{\ssets^c_+}}
\def\SSetdcplus{{\ssets_+^{\it dc}}}\def\hoSSetcplus{{\rm ho}{\ssets_+^c}}
\def\hoSSetc{{\rm ho}{\ssets^c}}\def\hoSSetdc{{\rm ho}{\ssets^{\it dc}}}
\def\fQ{{\rm f{\mathbb Q}}}\def\hoSSetdcplus{{\rm ho}{\ssets_+^{\it dc}}}
\def\fNQ{{\rm fN{\mathbb Q}}}
\def\hoBG{{\rm ho}\BG}

\def\honilpsimpdisconplus{\hbox{$\fNQ$\rm-ho$\SSet^{\it dc}_+$}}
\def\adjointext#1#2{{
\unitlength=0.05em
\thicklines
\begin{picture}(80.00,32)(-10.00,4.00)
\put(00.00,5.00){\vector(-3,1){0}}
\put(30.00,27){\makebox(0,0)[b]{\scriptsize $#1$}}
\put(30.00,-20){\makebox(0,0)[b]{\scriptsize $#2$}}
\put(60.00,15.00){\vector(3,-1){0}}
\qbezier(0.00,5.00)(30.00,-5.00)(60.00,5.00)
\qbezier(0.00,15.00)(30.00,25.00)(60.00,15.00)
\end{picture}}}
\def\L{\hat{\mathscr L}}\def\Neistrict{\hbox{$\it nDGLA$}}
\def\exthoLdc{{\hbox{ext$\fQ$\rm-ho}}\L^{\it{dc}}}
\def\Nei{\hbox{${\rm ho}(\it nDGLA)$}}\def\Lie{{\mathscr L}}
\def\Lfun{\mathcal L}\def\horathincatdcplus{\hbox{$\fQ$\rm-ho$\Hibez^{\it dc}_+$}}
\def\Cfun{{\mathcal C}}\def\Lc{{\hbox{$\fQ$\rm-}}\L_{\geq 0}}
\def\hoLc{{\hbox{$\fQ$\rm-ho}}\L_{\geq 0}}
\def\g{\mathfrak g}\def\honilpsimpconplus{\hbox{$\fNQ$\rm-ho$\SSet^{\it c}_+$}}
\def\h{\mathfrak h}

\def\a{\mathfrak a}
\def\Ass{\hbox{${\mathscr A} \hskip -.15em ss$}}
\def\Asshat{\hbox{${\widehat {\mathscr A}}\hskip -.15em ss$}}
\def\B{\mathcal B}
\def\Bhat{\hat{\mathcal B}}

\DeclareMathOperator{\MC}{MC}
\def\adjoint{{
\unitlength=0.05em
\thicklines
\begin{picture}(80.00,20.00)(-10.00,4.00)
\put(00.00,5.00){\vector(-3,1){0}}
\put(60.00,15.00){\vector(3,-1){0}}
\qbezier(0.00,5.00)(30.00,-5.00)(60.00,5.00)
\qbezier(0.00,15.00)(30.00,25.00)(60.00,15.00)
\end{picture}}}

\begin{document}

\begin{abstract}
We construct two algebraic versions of homotopy theory of
rational disconnected topological spaces, one based on differential graded commutative
associative algebras and the other one on complete differential graded Lie algebras.
As an application of the developed technology we obtain results on the structure of
Maurer-Cartan spaces of complete differential graded Lie algebras.
\end{abstract}

\bibliographystyle{plain}
\maketitle


\baselineskip15pt plus 2pt minus 1pt

\setcounter{tocdepth}{1}
\tableofcontents

\section*{Introduction}

The purpose of this paper is to construct an algebraic theory of
rational disconnected topological spaces (or simplicial sets), alluded
to in~\cite[p.~67]{markl12:_defor}.  The corresponding theory for \emph{connected} spaces was
constructed in the seminal papers \cite{quillen:Ann.ofMath.69,sullivan:Publ.IHES77}. The paper \cite{quillen:Ann.ofMath.69}
related rational connected spaces to differential graded Lie algebras
(dglas) whereas \cite{sullivan:Publ.IHES77} took the perspective of commutative
differential graded algebras (cdgas). In the present paper we pursue
both points of view and construct both the dgla and cdga algebraic
models for disconnected spaces.  It is interesting that the
differences between the two algebraic categories (which were somewhat
hidden in the Quillen and Sullivan approaches) become more pronounced
in our more general context; in particular we are naturally led to
consider dglas endowed with a linearly compact topology, while our
cdgas are still discrete.

The definitive
reference establishing a correspondence between rational connected spaces and cdgas
is \cite{bousfield-gugenheim}, and our treatment relies heavily on the results of that paper.
Recall that op.~cit.\ constructed a closed model category structure on \emph{non-negatively graded} cdgas and a Quillen adjunction between
this category and the category of simplicial sets. This adjunction
restricts to an equivalence between the homotopy categories of
connected cdgas of finite ${\mathbb Q}$-type
and connected rational nilpotent simplicial sets of finite type.

Our main innovation in establishing a cdga version of disconnected
rational homotopy theory is that we remove the restriction that our
cdgas be non-negatively graded and use the closed model structure on
the category of $\mathbb Z$-graded cdgas~\cite{hinich:CA97}.  This
seemingly innocent modification has quite dramatic consequences. For
example, any commutative algebra concentrated in degree zero is
cofibrant in the Bousfield-Gugenheim category, but \emph{not} in this
extended category, unless it is a retract of a polynomial algebra. The
closed model category of all cdgas appears more natural than the
Bousfield-Gugenheim category; for instance it allows one to define
Harrison-Andr\'e-Quillen cohomology, cf.~\cite{block05:_andre_quill}.

There is still a Quillen adjoint pair of functors between the categories of $\mathbb Z$-graded cdgas and simplicial sets giving rise to an adjunction on the level of homotopy categories. This adjunction restricts to an equivalence between the homotopy category of simplicial sets having finite number of connected components, each being rational, nilpotent and of finite type, and a certain subcategory of the homotopy category of $\mathbb Z$-graded cdgas. We give an explicit characterization of this subcategory.

To construct the second version of the disconnected rational homotopy
theory (based on dglas) we need to relate the homotopy theory of
commutative and Lie algebras. This relationship, which is sometimes
referred to as \emph{Koszul duality} was established in the work of
Quillen \cite{quillen:Ann.ofMath.69}; it was formulated as the duality between
differential graded coalgebras and dglas under certain (fairly severe)
restrictions on the grading of the objects under consideration. These
restrictions were subsequently removed in the seminal paper of Hinich
\cite{hinich01:_dg}. In our context we need a result that is Koszul
dual to Hinich's; it can be regarded as a Quillen equivalence between
the categories of cdgas and differential graded Lie coalgebras (which
can be dualized and viewed as complete dglas). We prove this result by
a suitable adaptation of Hinich's methods.  As a consequence we obtain
an algebraic model of disconnected rational homotopy theory based on
complete dglas. The condition of completeness is essential; it cannot
be removed even when restricted to connected spaces (or simplicial
sets). In the latter case our theory is close to, but still different
from, the one constructed in the papers by Neisendorfer and
Baues-Lemaire \cite{baues77:_minim,neisendorfer78:_lie}.
In particular, our complete dglas admit
minimal models even in the nonsimply-connected case, whereas
simply-connectedness is an essential requirement for the existence of
minimal models constructed by Neisendorfer and Baues-Lemaire.

One important application of the developed theory that we give in this
paper, concerns the structure of Maurer-Cartan spaces. Recall
(cf.~\cite{getzler:AM:09}) that associated to any dgla
$\g$ is a simplicial set $\MC_\bullet(\g)$; in the
case when $\g$ is non-negatively graded and nilpotent, this is a
simplicial set corresponding to $\g$ under the Quillen-Sullivan
correspondence. When $\g$ is \emph{not} non-negatively graded, the
simplicial set $\MC_\bullet(\g)$ is much more mysterious; this is a
fundamental object of study for deformation theory \cite{kontsevich:LMPh03, manetti:Lie,
markl12:_defor} and also comes up in modern approaches to quantum field theory
\cite{costello11:_renor}. It turns out that for differential graded Lie algebras
satisfying an appropriate completeness condition there exists an
operation of \emph{disjoint product} corresponding to the operation of
disjoint union of simplicial sets. The disjoint product of two dglas
is never non-negatively graded and we prove that the Maurer-Cartan
simplicial set of the disjoint product of complete dglas is weakly
equivalent to the disjoint union of the corresponding simplicial
sets. Furthermore, for an arbitrary complete dgla $\g$ the simplicial set $\MC_\bullet(\g)$ naturally
decomposes up to homotopy as a disjoint union  of Maurer-Cartan spaces of certain \emph{connected} dglas naturally associated  with $\g$.
Along the way we establish a general result of independent
interest, expressing the Chevalley-Eilenberg cohomology of free
products of dglas through the Chevalley-Eilenberg cohomology of the
individual pieces. It appears that this result is new even for
ordinary Lie algebras.

Similar results on the Maurer-Cartan spaces of dglas
were contained in the recent paper \cite{BM} by U.~Buijs and
A.~Murillo, who used completely different methods, see, in particular,
Theorem 5.5, Proposition 6.2 and Theorem 6.4 in op.~cit. Their statements do not
include any completeness assumptions. We have been unable to verify
the claims made in that paper.

M.~Golasi\'nski published a series of papers, see
e.g.~\cite{golasinski}, studying the equivariant 
rational homotopy type of spaces $X$ with an action of a finite
group $G$ with possibly disconnected fixed-point sets $X^H$. When $G$
is trivial, his theory describes spaces whose
connected components are rational nilpotent via families of
non-negatively graded complete nilpotent homologically connected cdgas indexed
by the components of $X$. We do not think that Golasi\'nski's theory implies
any of our results.

\subsection*{Notation and conventions}
All algebraic objects will be considered over a fixed field~$\bfk$ of
characteristic zero, at some places we require specifically $\bfk$ to
be the field $\bfQ$ of rational numbers. The abbreviation `dg' stands
for `differential graded'. Further, we will write `cdga' and `dgla'
for `commutative differential graded unital associative algebra' and
`differential graded Lie algebra' respectively.  We allow also the
terminal algebra ${\mathbf 0}$ in which $1=0$. For a (co)cycle $c$, the
symbol $[c]$ denotes its cohomology class. The degree of a homogeneous
element $a$ is denoted $|a|$.

All our cdgas will have cohomological grading with upper indices while
dglas will have homological grading with lower indices. There will,
however, be one important exception from this rule. At some places,
we will need the tensor product of a homologically graded Lie
algebra with the cohomologically graded cdga $\Omega$ of the Sullivan-de~Rham
forms. In this context, we consider $\Omega$ as homologically graded
by $\Omega_* := \Omega^{-*}$. The tensor product will then be a
homologically graded dgla as expected.

 The
\emph{suspension} $\Sigma V$ of a homologically graded vector space
$V$ is defined by the convention $\Sigma V_i=V_{i-1}$
resp.~$\Sigma V^i= V^{i-1}$ for $V$ cohomologically graded.
The functor of taking the linear dual takes
homologically graded vector spaces into cohomologically graded ones so
that $(V^*)^i=(V_{i})^*$; further we will write $\Sigma V^*$ for
$\Sigma(V^*)$; with this convention there is an isomorphism $(\Sigma
V)^*\cong\Sigma V^*$. Regarding spectral sequences, we will use the terminology
of~\cite{baordman:condition}. We will often refer to a duality between
the category of discrete vector spaces (colimits of finite-dimensional
ones) and linearly compact spaces (limits of finite-dimensional ones),
see e.g.~~\cite{lefschetz}.

\label{CMC}
Let $\Hin$ be the category of unital cdgas and $\BG$ its
subcategory consisting of non-negatively graded cdgas. By
\cite{hinich:CA97} and \cite{bousfield-gugenheim} respectively, both $\Hin$ and
$\BG$ are closed model categories whose weak equivalences are
morphisms inducing isomorphisms of cohomology, and fibrations are
surjective morphisms.

Note that $\BG$ has more cofibrations than $\Hin$. For instance,
every cdga concentrated in degree $0$ is cofibrant in
$\BG$. Proposition~\ref{sec:main-results} presents a wide class of
cdgas that are cofibrant in $\BG$ but not in $\Hin$.

As usual, a cdga $A = (A,d)$ is {\em connected\/} (resp.~{\em
homologically connected\/})\footnote{It would be more logical to say {\em
co\/}homologically connected, but we keep our terminology compatible
with~\cite{bousfield-gugenheim}.}  if $A^0 = \bfk$ (resp.~ $H^0(A,d) =
\bfk$) and $A^n =0$ (resp.~$H^n(A) = 0$) for $n \leq 0$. Each
connected algebra $A \in \BG$ admits a unique minimal
model~\cite[\S7]{bousfield-gugenheim}. Such an algebra $A$ is of {\em
finite $\mathbb Q$-type\/} if $A$ is defined over $\mathbb Q$ and its
minimal model has finitely many generators in each
degree. Equivalently, the cohomology of
$H(I/I^2)$, where $I \subset A$ is the ideal of
positive-degree elements of $A$, is finite-dimensional in each
degree, see \cite[\S9.2]{bousfield-gugenheim}

Let $\SSet$ denote the category of simplicial sets and $\nilpsimpcon$
its subcategory of connected nilpotent rational simplicial sets of
finite type. We denote
$\ratcat \subset \BG$ the subcategory of
homologically connected cdgas of finite ${\mathbb Q}$-type.
It is well-known \cite[Theorem~9.4]{bousfield-gugenheim} that there
exists a pair of adjoint functors
\begin{equation}
\label{houbicky}
F : \BG \adjoint \SSet : A
\end{equation}
that, for $\bfk = \bfQ$, induces an equivalence of the homotopy
categories $\horatcat$ and $\honilpsimpcon$. A similar adjunction and
induced equivalence hold also for augmented and pointed versions of
the above categories, see  again Theorem~9.4 in op.~cit.

We denote by $\Lie$ the category of dglas and by ${\L}$ the category
of \emph{complete} dglas, i.e.~inverse limits of finite-dimensional
nilpotent dglas; the morphisms in $\L$ are, naturally, continuous dgla
maps. We will show in Section~\ref{dual_Hinich} that $\L$ is a closed
model category whose weak equivalences are maps $f : \g' \to \g''$
such that $\Cfun(f) : \Cfun(\g'') \to \Cfun(\g')$ is a weak
equivalence; here $\Cfun(-)$ is the Chevalley-Eilenberg functor
recalled in Definition~\ref{CEHar}. Fibrations in $\L$ are surjective
morphisms.

The free product of two dglas $\g$ and $\h$ will be denoted by
$\g*\h$. If $\g$ and $\h$ are complete dglas then $\g*\h$ will stand
for the \emph{completed} free product of $\g$ and $\h$; it is thus a
categorical coproduct of $\g$ and $\h$.  Given a dgla
$\g$, a \emph{Maurer-Cartan} element in $\g$ is an element
$\xi\in \g_{-1}$ satisfying the Maurer-Cartan equation:
$d\xi+\frac{1}{2}[\xi,\xi]=0$. We will abbreviate the expression `Maurer-Cartan' as `MC'. The set of all MC elements in $\g$ will
be denoted by $\MC(\g)$. This definition can be extended to give a
simplicial set $\MC_\bullet(\g)$ whose vertices are just the MC
elements in $\g$; a precise definition is recalled in the main text.
Furthermore, given a dgla $\g$ and an MC element $\xi \in\MC(\g)$ we
can define a twisted differential $d^\xi$ in $\g$ by the formula
$d^\xi(?)=d(?)+[?,\xi]$; we will write $\g^\xi$ for the dgla $\g$
supplied with the twisted differential.

Given a dgla $\g$, we denote by $\g\langle x\rangle$ the dgla obtained
from $\g$ by freely adjoining the variable $x$ with $|x|=-1$ and
$d(x)=- \frac{1}{2}[x,x]$. Clearly, $x\in\MC(\g\langle x\rangle)$ and we
will write $\g\scup 0$ for the twisted dgla $(\g\langle
x\rangle)^x$.  One should view the construction $\g\scup 0$ as the
Lie analogue of adjoining an isolated base point to a topological
space. Furthermore, for two dglas $\g$ and $\h$ we set $\g\scup
\h:=(\g\scup 0)*\h$. The dgla $\g\scup \h$ will be called the
\emph{disjoint product} of $\g$ and $\h$; in the case when $\g$ and
$\h$ are complete we will write $\g\scup \h$ for the corresponding
completion. The operation of
disjoint product equips the category of complete dglas with a
non-unital monoidal structure. A non-complete version of the disjoint
product was considered in~\cite{BM}.

For convenience of the reader, we include a glossary of notation at the end
of Section~\ref{sec:main-results-2}.

\section{Main results}
\label{sec:main-results-2}

We call a cdga $A = (A,d) \in \Hin$ homologically {\em
disconnected\/} if $H^n(A) = 0$ for $n < 0$ and if $H^0(A)$ is isomorphic
to the direct product $\prod_{i \in \J} \bfk$ of copies of the ground
field indexed by some {\em finite\/} set $\J$. Let us denote by $\Hindc$
(resp.~$\BGdc$) the full subcategory of $\Hin$ (resp.~$\BG$)
consisting of homologically disconnected cdgas.  We also denote
by $\hoHindc$ (resp.~$\hoBGdc$) the full
subcategory of $\hoHin$ (resp.~$\hoBG$) whose objects are
homologically disconnected cdgas.\footnote{Model categories and their
  localizations are briefly recalled at the beginning of
  Section~\ref{sec:homot-prop-local}.}

\begin{theoremA}
The inclusion $\BG \subset \Hi$ induces an equivalence of the homotopy
categories $\hoBGdc$ and $\hoHindc$.
\end{theoremA}
\begin{rem}
A related question is whether $\hoBG$ is a full
subcategory of $\hoHin$. It is not difficult to show, for example, that for cdgas $A$ and $B$ with $A$ concentrated
in degree zero, the sets $\hoBG(A,B)$ and $\hoHin(A,B)$ are in natural bijective correspondence. In general however, we see no compelling reason
for the homotopy classes of maps in both categories to be the same.
\end{rem}

For a category $\catC$ denote by $\Mult(\catC)$ the category whose
objects are formal finite coproducts $A_1 \sqcup \cdots \sqcup A_s$, $s
\geq 1$, of objects of $\catC$ and the Hom-sets are
\[
\Mult(\catC)\big(A_1 \sqcup \cdots \sqcup A_s,B_1 \sqcup \cdots \sqcup
B_t\big) :=
\prod_{1 \leq i \leq s}\coprod_{1 \leq j \leq t}
\catC(A_i,B_j),
\]
with the obvious composition law.

\begin{rem}
\label{similarity}
The nature of our category $\Mult(\catC)$ bears some similarity to the
category $\hbox{\rm inj-}\catC$ of direct systems in $\catC$, see
e.g.~\cite[page~8]{has-edw}. This resemblance is however merely
superficial, as in the definitions of  $\hbox{\rm inj-}\catC$ as well
as in the dual definition  of the category $\hbox{\rm proj-}\catC$ of
inverse systems in $\catC$, one assumes the indexing small category to
be left filtering, while we use discrete finite categories.
\end{rem}

\begin{example}
\label{v_susarne}
Denote by $\SSetc$ the category of connected simplicial sets and by
$\SSetdc$ the category of simplicial sets with finitely many
components. Then clearly $\SSetdc \cong \Mult(\SSetc)$ and the same is obviously
true also for the homotopy categories, i.e.~$\hoSSetdc \cong \Mult(\hoSSetc)$.
\end{example}

Our next main theorem that states a similar result also for
homotopy categories of cdgas requires a contravariant
version of $\Mult(\catC)$.
Namely, denote by $\Tlum(\catC)$ the category\footnote{The similarity
  of this category to the category of inverse systems is addressed
  in Remark~\ref{similarity}.} whose
object are formal finite products $A_1 \timesred \cdots \timesred A_s$, $s
\geq 1$, of objects of $\catC$, and morphisms~are
\[
\Tlum(\catC)\big(A_1 \times \cdots \times A_s,B_1 \times \cdots \times
B_t\big) :=
\prod_{1 \leq j \leq t}\coprod_{1 \leq i \leq s}
\catC(A_i,B_j).
\]
In the following theorem, $\Hinc$
(resp.~$\BGc$) denotes the full subcategory of $\Hin$ (resp.~$\BG$)
consisting of homologically connected cdgas.

\begin{theoremB}
Each homologically disconnected non-negatively graded cdga $A \in
\BGdc$ is isomorphic to a finite product $\prod_{i \in \J}A_i$ of
homologically connected non-negatively graded cdgas $A_i \in \BGc$.
This isomorphism extends to a natural equivalence of categories
\[
\hoBGdc \sim \Tlum(\hoBGc).
\]
Each homologically disconnected cdga  $A \in
\Hindc$
is weakly equivalent to a finite product $\prod_{i \in \J}A_i$ of
homologically connected cdgas $A_i \in \Hinc$. As above,
one has an equivalence
\[
\hoHindc \sim \Tlum(\hoHinc).
\]
\end{theoremB}

Our proofs of the above theorems use a proposition describing
homotopy classes of maps whose domain is a localization of a cdga.
We believe that this statement is of independent~interest.

For a cocycle $u$ of a cdga $A$ we denote by $\loc Au$ the
localization of $A$ at $u$, i.e.~at the multiplicative subset
generated by $u$, see \cite[Section~3]{atiyah-macdonald}.  Since $u$
is a cocycle, $\loc Au$ bears the induced differential.  Notice that
each $\chi \in [A,D]_\Hin := \hoHin(A,D)$, for $A,D \in \Hin$, induces
a~map $\chi_* : H(A) \to H(D)$.

\begin{theorem}
\label{Andrey'}
Let $A,D \in \Hin$ be cdga and $u \in A$ a cocycle. Denote
\[
[A,D]^u_\Hin := \big\{ \chi \in [A,D]_\Hin\ |\  \chi_*([u]) \in H(D)
\mbox { is invertible}\big\}.
\]
There is a natural isomorphism $\big[A[u^{-1}],D\big]_\Hin \cong [A,D]^u_\Hin$.
\end{theorem}

Observe that each invertible odd-degree element $x$ of a graded commutative algebra,
i.e.~one for which there exists $y$ such that $xy=1$, equals zero
since $x^2=0$ by graded commutativity. So the only graded commutative algebra admitting invertible
elements in odd degrees is the terminal one.
Theorem~~\ref{Andrey'} is thus meaningful only when the degree $|u|$
is even.

The next two applications of our theory describe the homotopy
category of spaces with finitely many rational nilpotent components of
finite type. We need the following definition related to the algebraic
side.

\begin{definition}
\label{finite-type-def}
Let $A$ be a cdga. We say that a dg ideal $I$ in $A$ is an
\emph{augmentation ideal} if $A/I\cong \bfk$, the ground field. Further,
$A$ is said to have a finite type if for any augmentation ideal $I$ of
$A$ the space $H(I/I^2)$ is finite dimensional in every degree.
\end{definition}

Assume that the ground field $\bfk$ is the field
${\mathbb Q}$ of rational numbers and denote by $\ratcatdc$,
resp.~$\rathincatdc$, the subcategory of $\BGdc$, resp.~$\Hindc$,
consisting of algebras having a cofibrant replacement of finite type. In
Section~\ref{finite} we prove that this definition does not depend on
the choice of a cofibrant replacement and relate it to the definition of finite
$\bbQ$-type given in~\cite{bousfield-gugenheim}.

\begin{theoremC}
The following three categories are equivalent.
\begin{itemize}
\item[--]
The homotopy category
$\honilpsimpdiscon$ of simplicial sets with finitely many components that
are rational and of finite type,

\item[--]
the homotopy category
$\horatcatdc$ of homologically disconnected  non-negatively graded
cdgas of finite type over~${\mathbb Q}$, and

\item[--]
the homotopy category
$\horathincatdc$ of homologically disconnected ${\mathbb Z}$-graded
cdgas of finite type over~${\mathbb Q}$.
\end{itemize}
\end{theoremC}

Let us denote by $\hoLdc$ the full subcategory of the homotopy
category of $\L$ consisting of disjoint products of complete
non-negatively graded dglas with finite-dimensional homology in each
dimension. We call objects of $\hoLdc$ {\em disconnected\/} dglas. We
have

\begin{theoremD}
The following categories are equivalent:
\begin{itemize}
\item[--]
the homotopy category
$\honilpsimpdisconplus$ of pointed simplicial sets with finitely many components that
are rational and of finite type, and
\item[--]
the homotopy category $\hoLdc$ of disconnected complete dglas of
finite type.
\end{itemize}
\end{theoremD}

Neisendorfer in~\cite[Proposition~7.3]{neisendorfer78:_lie} proved that the subcategory
$\honilpsimpconplus$ of $\honilpsimpdisconplus$ consisting of {\em
  connected\/} spaces is equivalent to the homotopy category $\Nei$ of
non-negatively graded (discrete) dglas $L$ whose homology $H(L)$ is of
finite type and nilpotent. As a particular case of Theorem~D we get
another description of  $\honilpsimpconplus$.
Denote by $\hoLc$ the
full subcategory of $\hoLdc$ of complete non-negatively graded dglas
with finite dimensional homology in each degree.

\begin{corollary}
\label{sec:disc-spac-dglas}
The simplicial MC functor $\MC_\bullet(-)$  induces an
equivalence between the categories $\hoLc$ and $\honilpsimpconplus$.
\end{corollary}

A nice feature of the category $\L$ is that each $\g \in \L$ has a
{\em minimal model\/}, unique up to isomorphism, see
Definition~\ref{sec:dual-hinich-corr-5} and
Theorem~\ref{sec:dual-hinich-corr-6}.
To objects of $\hoLc$
there correspond non-negatively graded minimal dglas ${\Min}$
with homology of finite type. Corollary~\ref{sec:disc-spac-dglas}
implies a one-to-one correspondence between rational homotopy
types of connected nilpotent spaces of finite ${\mathbb Q}$-type and
isomorphism classes of minimal dglas  ${\Min}$ as above.

The description of $\honilpsimpconplus$ given in
Corollary~\ref{sec:disc-spac-dglas} substantially differs from
Neisendorfer's.  Notice, for instance, that the category $\Nei$ has
more objects than $\hoLc$. For example, the contractible free Lie
algebra ${\mathbb L}(x,\partial x)$, $|x|=1$, belongs to $\Nei$
but not to $\hoLc$; ${\mathbb L}(x,\partial x)$ is not complete.

Tracing the functors in~\cite{neisendorfer78:_lie}, one can associate to a dgla $\g \in
\Lc$
the corresponding $L \in \Neistrict$ as follows. The cdga $\Cfun(\g)$ is
connected and non-negatively graded, so it has the minimal model
$M_A$. By assumption, $M_A$ is a cdga of {\em finite type\/}, so we
may take $L := \Lfun(M_A)$, the {\em uncompleted\/} Quillen functor.

On the other hand, starting from $L \in \Neistrict$, we take
the cobar construction $\Cfun^c(L)$ on the dgla $A$, i.e.~the obvious coalgebra
version of the uncompleted functor $\Cfun(-)$, and its linear dual
$\Cfun^c(L)^*$.\footnote{Observe that $\Cfun^c(L)^*$ exists while
$\Cfun(L)$ may not.} Then $\g := \Lhat\big(\Cfun^c(L)^*\big)$ is the
corresponding dgla in $\Lc$.

Another application concerns the general structure of the MC simplicial
sets. The ground field $\bfk$ may again be an arbitrary field of
characteristic zero.

\begin{theorem}
\label{theoremF}
Let $\g_i$, $i\in \J$, be a collection of complete dglas indexed by a
finite set $\J$. Then the simplicial set $\MC_\bullet\big(\bigsqcup_{i\in
\J}\g_i\big)$ is weakly equivalent to the disjoint union $\bigcup_{i\in
\J}\MC_\bullet(\g_i)$.
\end{theorem}

Theorem~\ref{theoremF} yields the following elementary corollary on
the sets $\MCmod(-) := \pi_0 \MC_\bullet(-)$  of connected components of
the simplicial MC spaces. We do not know if it has a direct~proof.

\begin{corollary}
Let $\g_i, i\in \J$, be a collection of complete dglas indexed by a
finite set $\J$. Then there is a bijection
$
\MCmod\big(\bigsqcup_{i\in \J}\g_i\big)\cong \bigcup_{i\in \J}\MCmod(\g_i)
$
of the MC moduli sets.
\end{corollary}

The following theorem in a certain sense reverses
Theorem~\ref{theoremF}. It uses the twisting $\g^\xi$ of a complete
dgla $\g$ by an MC element $\xi \in \MC(\g)$ and its connected cover
$\overline{\g^\xi}$ defined in~(\ref{eq:10}).

\begin{theorem}
\label{sec:main-results-3}
For a complete dgla $\g$, one has a weak equivalence
\begin{equation}
\label{zas_bouchani}
\MC_\bullet(\g) \sim \bigcup_{[\xi] \in \MCmod(\g)} \MC_\bullet(\overline{\g^\xi})
\end{equation}
where the disjoint union in the right hand side runs over chosen
representatives of the isomorphism classes in $\MCmod(\g)$.
If $\MCmod(\g)$ is finite, one furthermore has a weak equivalence
\[
\MC_\bullet(\g) \sim
\MC_\bullet \Big(\bigsqcup_{[\xi] \in \MCmod(\g)} \overline{\g^\xi} \Big)
\]
of simplicial sets.
\end{theorem}

Theorems~A, B and C are proved in Section~\ref{proofs}, Theorem~D in
Section~\ref{disc-spaces}.
Theorem~\ref{Andrey'} is proved in
Section~\ref{sec:homot-prop-local} and
Theorem~\ref{sec:main-results-3} in
Section~\ref{sec:appl-struct-simpl-3}.
The proof of Theorem~\ref{theoremF} occupies
Section~\ref{disjointproof}; it is surprisingly involved in that it relies, essentially, on all of the
technology developed in the previous sections and the Appendix.
The pointed (or augmented) versions of Theorems~A, B and C are formulated
in Section~\ref{sec:augm-dg-comm}.
We finish this part by~a

\vskip .5em

\noindent
{\bf Glossary of notation.}
We use the following notation for various categories:
\[
\def\arraystretch{1.2}
\begin{array}{rl}
\Hin,&\mbox {the category of ($\mathbb Z$-graded) cdgas,}
\cr
\BG,&\mbox {the category of non-negatively graded cdgas,}
\cr
\Hinc,&\mbox {the category of homologically connected cdgas,}
\cr
\BGc,&\mbox {the category of non-negatively graded
homologically connected cdgas,}
\cr
\Hindc,&\mbox {the category of homologically disconnected cdgas,}
\cr
\BGdc,&\mbox {the category of non-negatively graded
homologically disconnected cdgas,}
\cr
\Lie,&\mbox {the category of dglas}
\cr
\L,&\mbox {the category of complete dglas,}
\cr
\SSet,&\mbox {the category of simplicial sets,}
\cr
\SSetc,&\mbox {the category of connected simplicial sets,}
\cr
\SSetdc,&\mbox {the category of simplicial sets with finitely many components.}
\end{array}
\]
The prefix `${\rm f{\mathbb Q}}$-' applied to a category of
algebras means `finite type over ${\mathbb Q}$' while the prefix `${\rm
  fN{\mathbb Q}}$-' applied to a category of simplicial sets
abbreviates `nilpotent, rational components of finite type.' The
subscript `$\hskip .1em_+$' means `pointed' for simplicial sets and `augmented'
\hbox{for~algebras}.

\part{The de~Rham-Sullivan approach.}
\label{part:point-view-dg}

In this part we describe our first version of disconnected rational
homotopy theory based on cdgas.

\section{Homotopy properties of the localization and proof of
  Theorem~\ref{Andrey'}}
\label{sec:homot-prop-local}

Recall~\cite[Sections~5,6]{dwyer-spalinski} that the homotopy category
$\hoHin$ of the model category $\Hin$ of ${\mathbb Z}$-graded unital
cdgas has the same objects as $\Hin$, and the morphism sets
$[X,Y]_\Hin$ defined~as
\[
[X,Y]_\Hin := \pi(QX,QY)_\Hin,\  X,Y \in \Hin,
\]
where $QX$ resp.~$QY$ is a cofibrant replacement of $X$
resp.~$Y$\footnote{All objects in $\Hin$ are fibrant. In a general
model category, the cofibrant replacement must be followed by the
fibrant replacement.} and $\pi(-,-)_\Hin$ denotes the set of homotopy
classes. By \cite[Proposition~5.11]{dwyer-spalinski}, if $A$ is
cofibrant and $Y$ fibrant, which in our situation means that $Y$ is arbitrary,
one has an isomorphism
\begin{equation}
\label{EQ:8}
[A,Y]_\Hin \cong \pi(A,Y)_\Hin.
\end{equation}
There is a functor $\gamma : \Hin \to \hoHin$ which is the identity on
objects and, for a morphism $f : X \to Y$ in $\Hin$, $\gamma(f)$ is
the homotopy class of a lift $\tilde f : QX\to QY$ of $f$. The
homotopy category $\hoBG$ of $\BG$ has an obvious similar description.

Let $A$ be a cdga and $S \subset A$ a multiplicative subset of
cocycles. Then the localization
\cite[Section~3]{atiyah-macdonald}\footnote{Notice that
  $S^{-1}A$ is in \cite{atiyah-macdonald} called `the ring of fractions
  with respect to the multiplicative subset $S$.'} $S^{-1}A$ of $A$ is a cdga
and the canonical map $A \to S^{-1}A$ is a morphism of cdgas. The
property crucial for us  is the exactness
\cite[Proposition~3.3]{atiyah-macdonald} of the functor $A \mapsto  S^{-1}A$.
If $S$ is multiplicatively generated by a cocycle $u \in A$
we will write $\loc Au$ for $S^{-1}A$.

\begin{rem}
\label{sec:homot-prop-local-1}
It is easy to see, using the exactness of the localization, that the homotopy type
of $\loc Au$ depends only on the cohomology class of $u$ in $H(A)$.
\end{rem}

Let us prove Theorem~\ref{Andrey'}
which we formulate in a slightly extended form as:

\begin{theorem}
\label{Andrey}
Let $A,D \in \Hin$, $u \in A$ a cocycle and $p : A \to A[u^{-1}]$ the
localization map. Denote
\[
[A,D]^u_\Hin := \big\{ \chi \in [A,D]_\Hin|\  \chi_*([u]) \in H(D)
\mbox { is invertible}\big\}.
\]
Then the map $p^\sharp : \big[A[u^{-1}],D\big]_\Hin \to [A,D]_\Hin$,
$p^\sharp(\chi) := \chi \circ \gamma(p)$, induces an isomorphism
\[
\big[A[u^{-1}],D\big]_\Hin \cong [A,D]^u_\Hin.
\]
\end{theorem}

\begin{example}
A curious particular case is when $u$ is cohomologous to zero $0$ in
$A$. Then $H\big(A[u^{-1}]\big)$ is the terminal algebra ${\bf 0}$ in which
$1=0$. By Remark~\ref{sec:homot-prop-local-1}, $A[u^{-1}]$ is
isomorphic, in the homotopy category, to ${\bf 0}$, so clearly
\[
\big[A[u^{-1}],D\big]_\Hin =
\cases \emptyset{if $1 \not= 0$ in
  $H(D)$, and}{\hbox{the one-point set}}{if $H(D) = {\bf 0}$.}
\]
It is immediate to see that the set $[A,D]^u_\Hin$
has the same description.
\end{example}

The rest of this section is devoted to the proof of Theorem~\ref{Andrey}.
We say that $\phi \in [X,Y]_\Hin$ is {\em represented\/} by $f \in
\Hin(X,Y)$ if $\phi = \gamma(f)$.
We call a cocycle $u \in X \in
\Hin$ {\em cohomologically invertible\/} if its cohomology class $[u]
\in H(X)$ is invertible.

\begin{proposition}
\label{sec:atestace}
Assume that $A \in \Hin$ is cofibrant and $D \in \Hin$ an cdga whose
each \cic\ is invertible.  Let $u \in A$ be a cocycle. Then each $\phi
\in \big[A[u^{-1}],D\big]_\Hin$ is represented by some $f: A[u^{-1}] \to D$.
\end{proposition}

\begin{proof}
Let us look at the left half of the diagram
\begin{equation}
\raisebox{-1.8cm}{}
\label{EQ:2.1}
{
\unitlength=1pt
\begin{picture}(200.00,60)(-60.00,128)
\thicklines
\put(-32,165){\makebox(0.00,0.00)[rt]{\scriptsize $q$}}
\put(15.00,145){\makebox(0.00,0.00)[l]{\scriptsize $c$}}
\put(100.00,105){\makebox(0.00,0.00)[t]{\scriptsize $f$}}
\put(80,135){\makebox(0.00,0.00)[b]{\scriptsize $\bar f$}}
\put(100.00,155){\makebox(0.00,0.00)[b]{\scriptsize $\tilde \phi$}}
\put(-40.00,110){\vector(1,-1){20.00}}
\put(-10.00,100){\vector(-1,1){20.00}}
\put(-40.00,100){\makebox(0.00,0.00){\scriptsize $r$}}
\put(-15.00,115){\makebox(0.00,0.00){\scriptsize $i$}}

\put(-15.00,170.00){\vector(-1,-1){30.00}}
\multiput(40,90)(11.3,3.8){10}{\qbezier(0,0)(2.4,.8)(4.8,1.6)}
\put(145.00,125){\vector(3,1){6}}
\put(40,170.00){\vector(3,-1){110}}
\put(10.00,170.00){\vector(0,-1){73}}
\multiput(-10,130)(11.5,0){13}{\line(1,0){6}}
\put(140,130.00){\vector(1,0){10}}
\put(155,130){\makebox(0.00,0.00)[l]{$D$.}}
\put(-55.00,130.00){\makebox(0.00,0.00){$\big(A[u^{-1}][s,ds],d'\big)$}}
\put(10.00,180.00){\makebox(0.00,0.00){$\big(A[y,z],d''\big)$}}
\put(10.00,85.00){\makebox(0.00,0.00){$\big(A[u^{-1}],\bar d\big)$}}
\end{picture}}
\end{equation}
The symbols $y,z,s,ds$ are new free generators with
$|z|=|s|:= -1$, \hbox{$|y| := -|u|$}.
The cdga $\big(A[u^{-1}],\bar d\big)$ is the localization of $A = (A,d)$ at $u$
with the induced differential. The differential $d'$ of
$A[u^{-1}][s,ds]$ is defined by
\[
d' (\bar a) := \bar d (\bar a)\ \mbox { for }\  \bar a \in A[u^{-1}], \  d'(s) :=
ds \ \mbox { and } \ d'(ds):= 0,
\]
and the differential $d''$
of $A[y,z]$ by
\[
d'' (a) :=  d (a)\ \mbox { for }\  a \in A, \ d''(z) :=
uy -1 \ \mbox { and } \ d''(y):= 0.
\]

Let $\bar a$ denote the image of $a \in A$ under the
localization map $A \to A[u^{-1}]$.
The map $c: A[y,z] \to A[u^{-1}]$ is then defined by
\[
c(a) := \bar
a,\ c(y) := u^{-1} \ \mbox { and } \  c(z) = 0.
\]
The map $q   : A[y,z] \to A[u^{-1}][s,ds]$ is given by
\[
q(a) := \bar a,\ q(y) := u^{-1}(ds+1) \ \mbox { and } \ q(z):= s.
\]
Finally, $i : A[u^{-1}] \to  A[u^{-1}][s,ds]$ is the inclusion
and $r :  A[u^{-1}][s,ds] \to  A[u^{-1}]$ the obvious retraction.
It is routine to verify that all the maps above commute with the
differentials, that $c = rq$ and that $ri$ is the identity.

Since $\big(A[u^{-1}][s,ds],d'\big)$ is the tensor product of
$\big(A[u^{-1}],\bar d\big)$ with the
`standard' acyclic cdga $\bfk[s,ds]$, one sees that
both $i$ and $r$ are weak equivalences. A simple spectral
sequence argument shows that also $c$ is a weak equivalence. The cdga
$\big(A[y,z],d''\big)$ was created from a cofibrant $A$ by a cell
attachment, it is therefore also cofibrant. As $A[u^{-1}]$ is
generated by the image of $A$ under the localization map and by $u^{-1}$,
$c$ is an epimorphism, i.e.~a fibration in $\Hin$.
The map $c$ thus can be taken as a cofibrant replacement of
$A[u^{-1}]$.

Let us inspect the localization of  $\big(A[y,z],d''\big)$ at $u$.
It is clear that $A[y,z][u^{-1}] \cong A[u^{-1}][y,z]$ with the
differential ${\bar d}'$ given by
\[
{\bar d}' (\bar a) := \bar d (\bar a)\ \mbox { for }\  \bar a \in A[u^{-1}], \
{\bar d}'(z) :=
uy -1 \ \mbox { and } \ {\bar d}'(y):= 0.
\]
It is  simple to check that the formulas
\[
\alpha(\bar a) := \bar a  \  \mbox { for }\  \bar a \in A[u^{-1}],
\ \alpha(y) := u^{-1}(ds+1) \ \mbox { and } \
\alpha(z) := s
\]
define an isomorphism
\[
\alpha : \big( A[u^{-1}][y,z],{\bar d}'\big)
\stackrel \cong \longrightarrow \big( A[u^{-1}][s,ds],d'\big)
\]
such that $\alpha p = q$, where $p :
A[y,z] \to  A[u^{-1}][y,z]$ is the localization map. We can therefore take $\big(
A[u^{-1}][y,z],d'\big)$ as the localization of
$\big(A[y,z],d''\big)$ at $u$, with $q$ the localization map.

Let $\tilde \phi : A[y,z] \to D$ as in~(\ref{EQ:2.1}) represents $\phi
\in \big[A[u^{-1}],D\big]_\Hin$, i.e.~$\phi = \gamma(\bar \phi)$. By
definition, $ \phi_*([u]) \in H(D)$ is invertible, therefore the
cocycle $\tilde \phi(u) \in D$ representing $\phi_*([u])$ is
invertible, so $\tilde \phi$ factorizes via the localization map $q :
A[y,z] \to A[u^{-1}][s,ds]$. We get, in~(\ref{EQ:2.1}), a unique map
$\bar f : A[u^{-1}][s,ds] \to D$ such that $\tilde \phi = \bar f
q$. Let finally $f := \bar f i : A[u^{-1}] \to D$.

We are going to prove that $f$ represents $\phi$.
Applying the functor $\gamma$ to the equation $fc =
\bar firq$ gives
\begin{equation}
\label{EQ:3}
\gamma(f)\gamma(c) = \gamma (\bar f) \gamma (i) \gamma (r) \gamma (q).
\end{equation}
Since $i$ and $r$ are weak equivalences and $ri = {\id}$, $\gamma
(i)$ and  $\gamma (r)$ are mutually inverse isomorphisms in
$\hoHin$. As $c$ is our
chosen cofibrant replacement of $A[u^{-1}]$, $\gamma(c)$ is the identity,
thus~(\ref{EQ:3}) reduces to
\[
\gamma(f) =  \gamma (\bar f)\gamma (q) = \gamma(\bar f q).
\]
The proof is finished by recalling that $\bar f q = \tilde \phi$,
hence $\phi = \gamma(\tilde \phi) = \gamma(f)$.
\end{proof}

\begin{lemma}
\label{sec:loc}
For each cdga $X \in \Hin$ there exists $\hatX \in \Hin$ and a weak
equivalence \hbox{$q : X \to \hatX$} such that
\begin{itemize}
\item[(i)]
each \cic\ $u \in \hatX$ is invertible, and
\item[(ii)]
each morphism $f: X \to Z$ whose target is an cdga $Z \in \Hin$ in
which all \cic{s} are invertible,
uniquely factorizes via $q : X \to \hatX$.
\end{itemize}
\end{lemma}

\begin{proof}
We start  by observing that if, in an cdga
$D = (D,d)$, all cocycles cohomologous to~$1$ are invertible, then all \cic{s}
are invertible.
Indeed, let $x \in D$ be cohomologically invertible, i.e.~$xy = 1+db$
for some $y,b \in D$. By assumption, $1+db$ is invertible, so $x^{-1}:=y
(1+db)^{-1}$ exists.

Denote by $S$ the multiplicative set of all \cic{s}
$u \in X$ and by $\hatX : =S^{-1}X$ the localization of $X = (X,d)$ at
$S$ with the induced differential $\bar d$. Let \hbox{$q : X \to \hatX$} be
the localization map.
To prove (i) it is, by the above observation, enough to show that
each cocycle $x \in \hatX$ cohomologous to $1$ is invertible. Let $x =
1+\bar db$, $b \in \hatX$. Clearly, $b = q(s)^{-1}q(a)$
for some $a \in X$ and $s \in S$, so $q(s)x
= q(s) + q(da) = q(s+da)$. Since $[s + da] = [s]$, $q(s + da)$
is invertible in $\hatX$
by the definition of $S$. We can therefore take $x^{-1} :=
q(s)q(s+da)^{-1}$.

Part (ii) follows from the standard universal property of the localization.
\end{proof}

\begin{lemma}
\label{sec:good}
Assume that all
\cic{s} of $D \in \Hin$
are invertible. Then there exists a good path object in the
sense of
\cite[\S 4.12]{dwyer-spalinski}
\begin{equation}
\label{EQ:1}
D \stackrel i\to \good \stackrel {(p_1,p_2)}\vlra
D \timesred D
\end{equation}
such that each \cic\ of $\good$ is invertible.
\end{lemma}

\begin{proof}
Take any good path object $D \stackrel {\bar i}\to P \stackrel {(\bar
  p_1,\bar p_2)}\vlra D \timesred D$. Lemma~\ref{sec:loc} applied to $P$
produces a cdga $D^I$ such that all its cohomologically
invertible cocycles are invertible, together with  a weak equivalence $q: P \to
D^I$. It is clear that, if all \cic{s} of $D$ are invertible, $D
  \times D$ has the same property therefore, by~(ii) of
Lemma~\ref{sec:loc}, $(\bar p_1,\bar p_2)$
factorizes~as
\[
(\bar p_1,\bar p_2):
P \stackrel q\to \good  \stackrel {(p_1,p_2)}\vlra
D \times D.
\]
We claim that~(\ref{EQ:1}) with $i :=  q \bar i$ is a good path
object for $D$.

Firstly, $i$, being the composition of two weak
equivalences, is a weak equivalence. Secondly, since $P$ is a good path
object, $(\bar p_1,\bar p_2)$ is a fibration in $\Hin$ i.e.~an
epimorphism, hence $(p_1,p_2)$ must be an epimorphism, i.e.~a
fibration, as well.
\end{proof}

\begin{proposition}
\label{sec:homot}
Let $A,D \in \Hin$ and assume that $A$ is cofibrant and
each \cic\ in $D$ is invertible.  Let $u \in A$ be a cocycle and $p: A
\to A[u^{-1}]$ be the localization map.
Assume that $f_i : A[u^{-1}] \to D$, $i=1,2$, are such that the compositions
$f_1p,f_2p  : A \to D$ are
homotopic. Then $f_1$ and $f_2$ are right homotopic.
\end{proposition}

\begin{proof}
Since $A$ is cofibrant, by \cite[Remark~4.23]{dwyer-spalinski}, we may
assume that $f_1p$ and $f_2p$ are right homotopic via a good path
object of Lemma~\ref{sec:good}. Let $h : A \to \good$ be such a right
homotopy,~i.e.%
\begin{equation}
\label{EQ:6}
f_1p =  p_1 h \ \mbox { and }\ f_2p =  p_2 h.
\end{equation}
It is clear that, for instance, $f_1p(u)$ is an invertible element in
$D$, with $f_1(u^{-1})$ as its inverse.  By
definition of a path object, $p_1 : \good \to D$ is a weak
equivalence, so the invertibility of $f_1p(u)= p_1 h(u) $ implies the
cohomological invertibility of $h(u) \in D$. By our choice of the
cylinder $\good$, $h(u) \in D$ is (strictly) invertible, thus the
homotopy $h$ factorizes as $h = \bar hp$ with some $\bar h : A[u^{-1}]
\to \good$.

It remains to prove that $\bar h$ is a right homotopy between $f_1$ and $f_2$,
that is
\begin{equation}
\label{EQ:4}
(f_1,f_2) = (p_1,p_2) \bar h.
\end{equation}
To this end, we invoke the obvious fact that two morphisms, say $u_1,u_2 :
A[u^{-1}] \to B$, agree if and only if their compositions $u_1p,u_2p : A \to
B$ with the localization map $p: A \to A[u^{-1}]$ agree.  To
prove~(\ref{EQ:4}), it therefore suffices to show that $(f_1p,f_2p) =
(p_1,p_2) \bar hp$, which follows from $h = \bar hp$
and~(\ref{EQ:6}).
\end{proof}

\begin{proof}[Proof of Theorem~\ref{Andrey}]
Let $c : \widetilde A \epi A$ be a cofibrant replacement of $A$ and $\tilde u
\in \widetilde A$
a cocycle such that $c(\tilde u) = u$. Let $p : A \to A[u^{-1}]$ be,
as in the theorem, the localization map for $A$ at $u$ and $\tilde p :
\widetilde A \to \widetilde A[\tilde u^{-1}]$
the localization map for $\widetilde A$ at $\tilde u$. One has the induced
morphism $\bar c : \widetilde A[\tilde u^{-1}] \to A[u^{-1}]$ that makes
the diagram
\[
{
\unitlength=.8pt
\begin{picture}(110.00,70.00)(0.00,0.00)
\thicklines
\put(50.00,5.00){\makebox(0.00,0.00)[b]{\scriptsize $p$}}
\put(50.00,65.00){\makebox(0.00,0.00)[b]{\scriptsize $\tilde p$}}
\put(110.00,30.00){\makebox(0.00,0.00)[l]{\scriptsize $\bar c$}}
\put(0.00,30.00){\makebox(0.00,0.00)[r]{\scriptsize $c$}}
\put(10.00,0.00){\makebox(0.00,0.00){$A$}}
\put(100.00,0.00){\makebox(0.00,0.00){$A[u^{-1}]$}}
\put(100.00,60.00){\makebox(0.00,0.00){$\widetilde A[\tilde u^{-1}]$}}
\put(10.00,60.00){\makebox(0.00,0.00){$\widetilde A$}}
\put(20.00,60.00){\vector(1,0){56}}
\put(100.00,48.00){\vector(0,-1){38.00}}
\put(20.00,0.00){\vector(1,0){56}}
\put(10.00,48.00){\vector(0,-1){38.00}}
\end{picture}}
\]
commutative.
 Being a cofibrant replacement, the map $c$ is as weak equivalence. By the
exactness of the localization, $\bar c$ is a weak equivalence as
well. Consider the induced diagram
\begin{equation}
\raisebox{-1cm}{}
\label{EQ:9}
{
\unitlength=.8pt
\begin{picture}(110.00,43.5)(0.00,30)
\thicklines
\put(50.00,5.00){\makebox(0.00,0.00)[b]{\scriptsize $p^\sharp$}}
\put(50.00,65.00){\makebox(0.00,0.00)[b]{\scriptsize $\tilde p^\sharp$}}
\put(125.00,30.00){\makebox(0.00,0.00)[l]{\scriptsize $\bar c^\sharp$}}
\put(-15.00,30.00){\makebox(0.00,0.00)[r]{\scriptsize $c^\sharp$}}
\put(-5,0.00){\makebox(0.00,0.00){$[A,D]_\Hin$}}
\put(115,0.00){\makebox(0.00,0.00){$\big[A[u^{-1}],D\big]_\Hin$}}
\put(115,60.00){\makebox(0.00,0.00)
{$\big[\widetilde A[\tilde u^{-1}],D\big]_\Hin$}}
\put(-5,60){\makebox(0.00,0.00){$[\widetilde A,D]_\Hin$}}
\put(70.00,60){\vector(-1,0){48}}
\put(115.00,13){\vector(0,1){36}}
\put(70,0.00){\vector(-1,0){48}}
\put(-5,13){\vector(0,1){36}}
\end{picture}}
\end{equation}
in which $p^\sharp$ (resp.~$\tilde p^\sharp$, resp.~$c^\sharp$,
resp.~$\bar c^\sharp$) are the pre-compositions with
$\gamma(p)$
(resp.~$\gamma(\tilde p)$, resp.~$\gamma(c)$,
resp.~$\gamma(\bar c)$).

Since $c$ and $\bar c$ are weak equivalences, the induced maps
$c^\sharp$ and $\bar c^\sharp$ are isomorphisms in $\hoHin$
($c^\sharp$ is in fact the identity). Clearly, $c^\sharp$ restricts to
an isomorphism $[A,D]^u_\Hin \cong [\widetilde A,D]^{\tilde u}_\Hin$
and also the inclusions ${\it Im}(p^\sharp) \subset [A,D]^u_\Hin$ and
${\it Im}(\tilde p^\sharp) \subset [\widetilde A,D]^u_\Hin$ are
obvious. Diagram~(\ref{EQ:9}) therefore restricts to
\[
\raisebox{-1.1cm}{}
{
\unitlength=.8pt
\begin{picture}(110.00,43.5)(0.00,30)
\thicklines
\put(50.00,5.00){\makebox(0.00,0.00)[b]{\scriptsize $p^\sharp$}}
\put(50.00,65.00){\makebox(0.00,0.00)[b]{\scriptsize $\tilde p^\sharp$}}
\put(125.00,30.00){\makebox(0.00,0.00)[l]{\scriptsize $\bar c^\sharp$}}
\put(105.00,30.00){\makebox(0.00,0.00)[r]{\scriptsize $\cong$}}
\put(-15.00,30.00){\makebox(0.00,0.00)[r]{\scriptsize $c^\sharp$}}
\put(5.00,30.00){\makebox(0.00,0.00)[l]{\scriptsize $\cong$}}
\put(-5,0.00){\makebox(0.00,0.00){$[A,D]^u_\Hin$}}
\put(115,0.00){\makebox(0.00,0.00){$\big[A[u^{-1}],D\big]_\Hin$}}
\put(115,60.00){\makebox(0.00,0.00)
{$\big[\widetilde A[\tilde u^{-1}],D\big]_\Hin$}}
\put(-5,60){\makebox(0.00,0.00){$[\widetilde A,D]^{\tilde u}_\Hin$}}
\put(70.00,60){\vector(-1,0){48}}
\put(115.00,13){\vector(0,1){36}}
\put(70,0.00){\vector(-1,0){48}}
\put(-5,13){\vector(0,1){36}}
\end{picture}}
\]
in which both vertical arrows are isomorphisms.

We conclude that the theorem will be proved for the
localization $p: A \to A[u^{-1}]$ if we prove it for  $\tilde p:
\widetilde A \to \widetilde A[\tilde u^{-1}]$. We may thus assume from
the beginning that {\em $A$ is cofibrant\/}.
An even simpler argument based on Lemma~\ref{sec:loc} shows that we
may also assume, without loss of generality, that {\em each \cic\ in $D$ is
invertible\/}.

To show that the image of $p^\sharp$ is $[A,D]^u_\Hin$ is now
easy. Since $A$ is cofibrant, each $\chi \in  [A,D]^u_\Hin$ is
represented by a map $w : A \to D$. As $\chi_*([u]) \in H(D)$ is,
by assumption, invertible, $w(u)$ is invertible in $D$. Thus $w$
factorizes via the localization map $p :A \to A[u^{-1}]$ as $w = fp$
for some $f :A[u^{-1}] \to D$. We then
have $\chi = \gamma(w) = \gamma(fp) = p^\sharp(\gamma(f))$, so $\chi
\in {\it Im}(p^\sharp)$.

Let us prove that $p^\sharp$ is injective. Assume that $\phi_i \in
\big[A[\tilde u^{-1}],D\big]_\Hin$, $i = 1,2$, are such that
\begin{equation}
\label{EQ:7}
p^\sharp \phi_1 = p^\sharp \phi_2.
\end{equation}
By
Proposition~\ref{sec:atestace}, there exist $f_i : A[\tilde u^{-1}]
\to D$ such that $\phi_i = \gamma(f_i)$. Equality~(\ref{EQ:7}) is then
equivalent to $\gamma(f_1 p) = \gamma(f_2 p)$. Since $A$ is cofibrant
this, by~(\ref{EQ:8}), means that $f_1 p$ and $f_2 p$ are
homotopic. By Proposition~\ref{sec:homot}, $f_1$ and $f_2$ are right
homotopic, which implies that $\gamma(f_1) = \gamma(f_2)$, i.e.~$\phi_1
= \phi_2$.
\end{proof}

\section{Maps of direct products of dg commutative associative algebras}

In this section we study maps, up to homotopy, whose source is a
finite direct product of cdgas. Somewhat unexpectedly, it turns out
that these maps can be completely understood in terms of (homotopy
classes of) maps out of individual components of these direct
products. Here is the first surprise.

\begin{proposition}
\label{sec:main-results}
Let $A_i$, $i \in \J$, be cofibrant algebras in $\BG$ indexed by a finite set
$\J$. Then the direct product  $\prod_{i \in \J}A_i$ is also cofibrant in $\BG$.
\end{proposition}

\begin{proof}
We prove the proposition for $\J = \{1,2\}$, the proof for an arbitrary
finite indexing set will be similar. Let thus
$A = A_1 \timesred A_2$. We need to prove
that, for any epimorphism $p: E \epi B$ in $\BG$ which is also a weak
equivalence, and for each $f : A \to B$, there exists a lift $\tf : A
\to E$ making the diagram
\begin{equation}
\label{eq:3}
\raisebox{-30pt}{}
{
\unitlength=.5pt
\begin{picture}(100.00,40.00)(0.00,40.00)
\thicklines
\put(49.00,43.00){\makebox(0.00,0.00)[br]{\scriptsize $\tf$}}
\put(106,40){\makebox(0.00,0.00)[l]{\scriptsize $p$}}
\put(50.00,6.00){\makebox(0.00,0.00)[b]{\scriptsize $f$}}
\put(100.00,80.00){\makebox(0.00,0.00){$E$}}
\put(100.00,0.00){\makebox(0.00,0.00){$B$}}
\put(2.00,0.00){\makebox(0.00,0.00)[r]{$A$}}
\put(100.00,65){\vector(0,-1){50}}
\put(13.00,13.00){\vector(4,3){70}}
\put(15.00,0.00){\vector(1,0){70}}
\end{picture}}
\end{equation}
commutative.

Let $e_1 := (1,0) \in A^0 = A_1^0 \timesred A^0_2$ and $e_2
:= (0,1) \in A^0  = A_1^0 \timesred A^0_2$. Clearly
\begin{equation}
\label{eq:2}
e_1 + e_2 =1 \ (\mbox{the unit of $A$}),\ e_1^2 = e_1,\
e_2^2 = e_2 \mbox { and } e_1e_2 = 0.
\end{equation}
Let $u_i := f(e_i)$, $i =1,2$. These elements  satisfy an
obvious analogue of~(\ref{eq:2}), moreover $de_1 = de_2 = 0$.
There are three possibilities:

\vskip 1em
\noindent
{\it Case 1.: $u_1 =1$, $u_2 = 0$.} Then $f$ restricted to $
0 \timesred A_2 \subset A$ is trivial. Indeed, for $(0,x_2)  \in 0 \times
A_2$ one has $(0,x_2) = (0,1)(0,x_2)$, therefore
\[
f(0,x_2) = f(0,1) f(0,x_2) = u_2f(0,x_2) = 0.
\]
In other words, $f$ factorizes via the projection $\pi_1 : A_1 \timesred
A_2 \to A_1$ as
\[
f : A_1 \timesred
A_2 \stackrel{\pi_1} \longrightarrow A_1\stackrel{f_1}
\longrightarrow B,
\]
where $f_1 (x_1) := f(x_1,0)$ for $x_1 \in A_1$. Since $A_1$ is cofibrant by
assumption, one has a lift $\tf_1 : A_1 \to E$ in the diagram
\[
\raisebox{-30pt}{}
{
\unitlength=.5pt
\begin{picture}(100.00,40.00)(0.00,40.00)
\thicklines
\put(48.00,44.00){\makebox(0.00,0.00)[br]{\scriptsize $\tf_1$}}
\put(106,40.00){\makebox(0.00,0.00)[l]{\scriptsize $p$}}
\put(50.00,6.00){\makebox(0.00,0.00)[b]{\scriptsize $f_1$}}
\put(100.00,80.00){\makebox(0.00,0.00){$E$}}
\put(100.00,0.00){\makebox(0.00,0.00){$B$}}
\put(2.00,0.00){\makebox(0.00,0.00)[r]{$A_1$}}
\put(100.00,65){\vector(0,-1){50}}
\put(13.00,13.00){\vector(4,3){70}}
\put(15.00,0.00){\vector(1,0){70}}
\end{picture}} \hskip 1em ,
\]
$\tf : = \tf_1 \pi$ then clearly solves the lifting problem~(\ref{eq:3}).

\vskip 1em
\noindent
{\it Case 2.: $u_1 =0$, $u_2 = 1$.} This `mirror image' of
Case~1 can be treated analogously.

\vskip 1em
\noindent
{\it Case 3.: $u_1,u_2 \not = 0$.} Since $p$ is a weak
equivalence and since there are no $0$-dimensional boundaries, $p$
induces an isomorphism of $0$-cocycles $Z^0(E) \cong
Z^0(B)$.\footnote{At this place we need $B$ and $E$ to be
  non-negatively graded.} In
particular, one has cocycles
$\tu_1,\tu_2 \in E^0$ such that $p(\tu_i) = u_i$, $i=1,2$,
satisfying conditions analogous to~(\ref{eq:2}).~Put
\[
B_i := u_iB\ \mbox { and } \  E_i := \tu_i E, \ i=1,2.
\]
It is clear that then
\[
B \cong B_1 \timesred B_2,\ E \cong E_1 \timesred E_2
\]
and that, under the above isomorphisms, also the maps $f$ and $p$ split,
\[
f = (f_1,f_2) : A_1 \timesred A_2 \to B_1 \timesred B_2,\
p = (p_1,p_2) : E_1 \timesred E_2 \to B_1 \timesred B_2,
\]
with $f_i := u_i f \iota_i$, where $\iota_i : A_i \hookrightarrow A_1
\timesred A_2$ are inclusions given by\footnote{These
inclusions are homomorphisms of non-unital cdgas, but the
composition $f_i$ preserves units.}
\[
\iota_1(x_1) := (x_1,0),\
\iota_2(x_2) := (0,x_2),\ x_i \in A_i, i=1,2.
\]
The maps
$p_1$ and $p_2$ are defined in the obvious similar way. Since both
$p_1$ and $p_2$ must clearly be weak equivalences and epimorphisms,
one has, for $i = 1,2$, the lifts $\tf_i$ in the diagrams
\[
\raisebox{-25pt}{}
{
\unitlength=.5pt
\begin{picture}(100.00,55.00)(0.00,40.00)
\thicklines
\put(48.00,44.00){\makebox(0.00,0.00)[br]{\scriptsize $\tf_i$}}
\put(106,40.00){\makebox(0.00,0.00)[l]{\scriptsize $p_i$}}
\put(50.00,6.00){\makebox(0.00,0.00)[b]{\scriptsize $f_i$}}
\put(100.00,80.00){\makebox(0.00,0.00){$E_i$}}
\put(100.00,0.00){\makebox(0.00,0.00){$B_i$}}
\put(2.00,0.00){\makebox(0.00,0.00)[r]{$A_i$}}
\put(100.00,65){\vector(0,-1){50}}
\put(13.00,13.00){\vector(4,3){70}}
\put(15.00,0.00){\vector(1,0){70}}
\end{picture}} \hskip 2em \mbox {.}
\]
The map $\tf := (\tf_1,\tf_2) : A_1 \timesred A_2 \to E_1 \timesred E_2$
solves the lifting problem~(\ref{eq:3}).
\end{proof}

Let us formulate the following simple principle
whose proof is straightforward.

\begin{principle}
\label{vedle_nekdo_tluce}
Assume $A_i$, $i \in \J$, are arbitrary
(unital) cdgas indexed by a finite set $\J$ and $D$ a cdga such that
$1$ is the only nontrivial idempotent in $Z^0(D)$. Then the projections
\begin{equation}
\label{proj}
\pi_i : \prod_{s \in \J} A_s \to A_i,\ i \in \J,
\end{equation}
from the cartesian product
induce monomorphisms of the
homomorphisms sets
\[
\pi_i^* : \mathscr A(A_i,D) \hookrightarrow
\mathscr A\big(\prod_{s \in \J} A_s,D\big),\ i \in \J,
\]
which in turn induce a decomposition
\begin{equation}
\label{prin}
\mathscr A\big(\prod_{s \in \J} A_s,D\big) \cong
\bigcup_{s\in \J} \mathscr A(A_s,D) \hskip 1em \hbox{(the
  disjoint union).}
\end{equation}
\end{principle}

Together with
Proposition~\ref{sec:main-results}, Principle \ref{vedle_nekdo_tluce} gives:

\begin{theorem}
\label{sec:main-results-1}
Let $A_i\in \BG$ be cdgas indexed by a finite set $\J$ and $D \in
\BG$ be such that $1 \in H^0(D)$ is the only nontrivial idempotent. Then
the projections~(\ref{proj}) induce a decomposition of the set of
homotopy classes
\begin{equation}
\big[\prod_{s \in \J} A_s,D\big]_\BG
\cong \bigcup_{s\in \J}\ [A_s,D]_\BG  \hskip 1em \hbox{(the
  disjoint union).}
\end{equation}
The same statement holds also with $\Hin$ in place of $\BG$.
\end{theorem}

\begin{proof}
Let us prove the first part.  Since $D$ is non-negatively graded,
$H^0(D) = Z^0(D)$, so $D$ fulfills the assumptions of Principle
\ref{vedle_nekdo_tluce}.  Moreover, any homotopy $h:\prod_{s \in \J}
A_s\to D[t,dt]$ factors through a unique homotopy $h_i:A_i\to D[t,dt]$
since the cdga $D[t,dt]$ also satisfies the assumptions of Principle
\ref{vedle_nekdo_tluce}. Consider the diagram
\[
\xymatrix
{
\coprod_{s\in \J}{\mathscr A}_{\geq 0}(A_s, D)
\ar^{\cong}[r]\ar[d]&\mathscr{A}_{\geq 0}\big(\prod_{s\in \J} A_s, D\big)\ar[d]\\
\coprod_{s\in \J}\big[A_s,D\big]_\BG\ar[r]&\big[\prod_{s\in \J} A_s,D\big]_\BG
}
\]
where the vertical arrows are natural quotient maps, associating to a
morphism its homotopy class. It follows that the lower horizontal
arrow, making the diagram commutative, exists, is unique and and
bijective; this finishes the proof for the category $\BG$.

The second part with $\Hi$ in place of $\BG$  must be proved
differently. The reason is that, firstly, no statement analogous to
Proposition~\ref{sec:main-results} holds in $\Hin$ and, secondly, even
if $1\in H^0(D)$ is the only nontrivial idempotent, there may be many
nontrivial idempotents in $Z^0(D)$, so Principle~\ref{vedle_nekdo_tluce}
does not apply.

To simplify the exposition, we assume again that $\J =
\{1,2\}$, the proof for a general finite $\J$ is similar.  Put
\[
u_1 := 1 \timesred 0 \in A_1 \timesred A_2\
\hbox{ and }\
u_1 := 0 \timesred 1 \in  A_1 \timesred A_2,
\]
and define
\[
[A_1 \timesred A_2,D]_\Hi^i \subset [A_1 \timesred A_2,D]_\Hi, \ i=1,2,
\]
as the subset of $\chi \in [A_1 \timesred A_2,D]$
such that $\chi_*([u_i]) = 1$. It is clear that
\[
[\norep,D]_\Hi = [\norep,D]^1_\Hi \cup  [\norep,D]^2_\Hi \hskip 1em \hbox{(the
  disjoint union).}
\]

So all we need to prove is that the projections induce an isomorphism
\begin{equation}
\label{eq:4}
[A_i,D]_\Hi \cong [\norep,D]^i_\Hi, \ i=1,2.
\end{equation}
Notice
that
\[
\loc{(A_1\timesred A_2)}{u_i} \cong A_i,\ i = 1,2.
\]
The isomorphism~(\ref{eq:4}) is thus a consequence of
Theorem~\ref{Andrey'} taken with $A = A_1\timesred A_2$ and $u = u_i$.
\end{proof}

For $A \in \BG$ consider its cofibrant replacement $c' :Q'A \epi A$ in
$\BG$ and take a cofibrant replacement $c: QA \epi Q'A$ of $Q'A$ in
$\Hin$. The composition $c' c: QA \epi A$ is clearly a~cofibrant
replacement of $A$ in $\Hin$. Notice also that, for $D \in \BG$, the
`standard' path object
\begin{equation}
\label{eq:5}
D[t,dt] := D \otimes \bfk[t,dt],\ |t|:=0,\ d(t) := dt,
\end{equation}
with the projections $p_1, p_2 : D[t,dt] \to D$ given by the
evaluation at $0$ resp $1$, is a good path object in the sense of \cite[\S
4.12]{dwyer-spalinski} for $D$ in both categories $\BG$ and $\Hi$. Therefore,
if $f_1,f_2 : Q'A \to D$ are right homotopic in $\BG$, then $f_1
c,f_2c$ are right homotopic in $\Hin$. By~(\ref{EQ:8}),
\[
[A,D]_\BG \cong \pi(Q'A,D)_\BG\ \mbox { and } \  [A,D]_\Hin \cong \pi(QA,D)_\Hin,
\]
so the pre-composition with $c$ defines a  natural map
\begin{equation}
\label{eq:a2}
K_{A,D} : [A,D]_\BG \to [A,D]_\Hi.
\end{equation}

\begin{lemma}
\label{sec:other-results-1}
Let $A \in \BGc$ and $D \in \BG$. Then $K_{A,D}$ is an isomorphism.
In particular, $\hoBGc$ is a full subcategory of $\hoHin$.
\end{lemma}

\begin{proof}
The cdga $A$ admits a minimal model $M$ by
\cite[Proposition~7.7]{bousfield-gugenheim}.
As $M$ is, by definition, weakly equivalent to $A$,
\begin{equation}
\label{eq:7}
[A,D]_\BG \cong [M,D]_\BG \ \ \mbox { and } \
[A,D]_\Hin \cong [M,D]_\Hin.
\end{equation}
The cdga $M$ is clearly cofibrant
in both categories $\BG$ and $\Hin$, therefore
\[
[M,D]_\Hin = \pi(M,D)_\BG \ \mbox { and } \ [M,D]_\BG = \pi(M,D)_\Hin
\]
by~(\ref{EQ:8}).  Since~(\ref{eq:5}) is a good path object for $D$ in both
categories $\BG$ and $\Hin$,
\[
 \pi(M,D)_\BG \cong  \pi(M,D)_\Hin.
\]
The lemma is an obvious combination of the above isomorphisms.
\end{proof}

\begin{proposition}
\label{sec:other-results}
The categories $\hoBGc$ and $\hoHinc$ are equivalent.
\end{proposition}

\begin{proof}
If $\calC'$ is a full subcategory of $\calC$ with the property that
each object of $\calC'$ is isomorphic to some object of $\calC$, then
both categories are equivalent. In light of
Lemma~\ref{sec:other-results-1},  it is enough
to prove that each $A \in \Hinc$ is weakly equivalent to an cdga
in $\BG$.

Proposition~7.7 of \cite{bousfield-gugenheim} states the existence of
a minimal model generated by elements of degrees $\geq 1$ of each
cdga $A \in \BGc$. One can easily verify that
the proof of this proposition leads to a minimal model of an {\em
arbitrary\/} homologically connected cdga $A$, i.e.~of an
arbitrary $A \in \Hinc$.
The minimal model $M$ of
$A \in \Hinc$ is then a~connected cdga belonging to $\BG$, weakly
equivalent to $A$. This finishes the proof.
\end{proof}

Theorem~\ref{sec:main-results-1} has the following important
consequence.

\begin{theorem}
Let $A_i$, $i \in \J$, be cdgas in $\BGc$ indexed by a finite set
and  $D \in \BG$ be such that $1 \in H^0(D)$ is the only nontrivial idempotent in
$H^0(D)$. Then the map in~(\ref{eq:a2}) with $A = \prod_{s\in \J} A_s$ is
an isomorphism
\[
\big[\prod_{s\in \J} A_s,D\big]_\BG
\cong \big[\prod_{s\in \J} A_s,D\big]_\Hi.
\]
\end{theorem}

\begin{proof}
By Lemma \ref{sec:other-results-1}, $[A_i,D]_\BG  \cong  [A_i,D]_\Hi$, $i \in \J$.
The rest follows from Theorem~\ref{sec:main-results-1}.
\end{proof}

\section{DG commutative associative algebras of finite type}
\label{finite}

In this section we investigate properties of cdgas having finite type
in the sense of Definition~\ref{finite-type-def}. The main results are
Propositions~\ref{vedle_delaji_lazne} and~\ref{p2}.
For a cdga $B$ and a map $f:A\to B$ we will denote by $\Der_f(A,B)$
the dg space of derivations of $A$ with values in $B$, where $B$ is
viewed as a dg $A$-module via $f$; if the map $f$ is clear from the
context we will write simply $\Der(A,B)$ for $\Der_f(A,B)$.

Note that having an augmentation ideal $I$ in $A$ is equivalent to
specifying a map $\epsilon :A\to \bfk$ (an augmentation). Furthermore,
the dual dg space $(I/I^2)^*$ is naturally identified with the dg
space $\Der_\epsilon(A,\bfk)$. Thus, $A$ is of finite type if and only
if for any augmentation $\epsilon :A\to \bfk$ the dg space
$\Der_\epsilon(A,\bfk)$ has finite-dimensional cohomology in each
positive degree.

We now discuss the homotopy invariance of the notion of finite type.
To this end, note that for a cofibrant cdga $A$ and a dg $A$-module
$M$ the dg space $\Der(A,M)$ is quasi-isomorphic to $C_{AQ}(A,M)$, the
Andr\`e-Quillen cohomology complex of $A$ with coefficients in $M$,
cf.~\cite[Theorem~2.4]{block05:_andre_quill}.

\begin{lemma}\label{sulho}
Let $A$ be a cofibrant cdga, $B$ is a cdga and $f,g:A\to B$ are two
Sullivan homotopic maps. Then the dg vector spaces $\Der_f(A,B)$ and
$\Der_g(A,B)$ are quasi-isomorphic.
\end{lemma}

\begin{proof}
Let $h:A\to B[t,dt]$ be a Sullivan homotopy from $f$ to $g$. The two evaluation
maps $|_{0,1} : B[t,dt]\to B$ determine maps of dg vector spaces
$\Der_h\big(A,B[t,dt]\big)\to \Der_f(A,B)$ and $\Der_h\big(A,B[t,dt]\big)\to
\Der_g(A,B)$. Comparing the corresponding spectral sequences
(\cite[Corollary 2.5]{block05:_andre_quill}) we conclude that both maps are
quasi-isomorphisms, giving the desired conclusion.
\end{proof}

\begin{proposition}
\label{vedle_delaji_lazne}
Let $A$ and $A^\prime$ be two quasi-isomorphic cofibrant cdgas. Then
$A$ is of finite type if and only if $A^\prime$ is of finite type.
\end{proposition}

\begin{proof}
Let $A$ be of finite type and $I^\prime$ be an augmentation ideal of
$A^\prime$. The given quasi-isomorphism $A\to A^\prime$ determines a
dg map $A\to A^\prime\to A^\prime/I^\prime\cong \bfk$.  Comparing the
corresponding spectral sequences in~\cite[Corollary 2.5]{block05:_andre_quill}
we obtain that $\Der(A^\prime,\bfk)$ is
quasi-isomorphic to
$\Der(A,\bfk)$ and thus, $\Der(A^\prime,\bfk)$ has finite-dimensional
cohomology, as desired. Since each quasi-isomorphism of cofibrant
algebras is homotopy invertible, the r\^oles of $A$ and $A'$ can be
exchanged. This finishes the proof.
\end{proof}

In the following statement $\Lhat A$ is the completed Harrison complex
of an augmented cdga $A$ recalled in Definition~\ref{CEHar}.

\begin{lemma}
\label{charact}
Suppose that $A$ is a cofibrant cdga and $I$ is an augmentation ideal
in $A$. Then $(I/I^2)^*$ is quasi-isomorphic to $\Sigma{\Lhat}A$.
\end{lemma}

\begin{proof}
Without loss of generality one can assume that the cdga $A$ is free as a graded
associative commutative algebra. It is easy to show that it is then
isomorphic, as a non-differential algebra, to the graded polynomial
ring generated by $I$.
Lemma~\ref{charact} therefore appears as a version of~\cite[Proposition~4.2(a)]{neisendorfer78:_lie}, so
we omit its proof.
\end{proof}

Observe that a~non-negatively graded
homologically connected cdga $A$  is of finite type if and only if it is of finite
type in the sense of \cite[\S9.2]{bousfield-gugenheim}. We can now formulate a
criterion for a homologically disconnected cdga to be of finite type.

\begin{proposition}
\label{p2}
Let $A$ be a homologically disconnected cdga, i.e.\ $A$ is
quasi-isomorphic, by Theorem~B, to $A_1\times\cdots\times A_n$ where $A_i$ are
connected cdgas, $i=1,2,\ldots, n$. Then $A$ is of finite type if and
only if each $A_i$ is of finite type.
\end{proposition}

\begin{proof}
We assume, without loss of generality, that $A$ is cofibrant and each
$A_i$ is also cofibrant. Suppose that every $A_i$ is of finite type
and let $I$ be an augmentation ideal in $A$. By Theorem~B, the map
$A\to A/I\cong \bfk$ must factor in the homotopy category through a
map $A_i\to \bfk$ for some $i=1,\ldots, n$. Since by Lemma \ref{sulho}
the homology of the space $\Der(A,\bfk)$ does not depend on the
homotopy class of the map $A\to \bfk$ we might as well assume that the map
$A\to \bfk$ factors through $A_i$ on the nose. Denote by $I_i$ the
kernel of the corresponding map~\hbox{$A_i \to \bfk$}.

Consider the dga $B:=\hat{T}\Sigma I^*$, the reduced cobar-construction of
$A$ with the augmentation given by $A \to A/I \cong \bfk$.
Note that $B$ computes $\Ext_{A}(\bfk,\bfk)$, the
differential Ext of the cdga~$A$. Similarly, the reduced
cobar-construction $B_i:=\hat{T}\Sigma I_i^*$ calculates
$\Ext_{A_i}(\bfk,\bfk)$.

One clearly has $\Ext_A(\bfk,\bfk) \cong \Ext_{A\otimes_A
A_i}(\bfk\otimes_A A_i,\bfk)$ while, by standard homological algebra,
$ \Ext_{A\otimes_A A_i}(\bfk\otimes_A A_i,\bfk)$ is isomorphic
to $\Ext_{A_i}(\bfk,\bfk)$. Looking at the primitive elements in the
spaces $\Ext_A(\bfk,\bfk)$ resp.\ $\Ext_{A_i}(\bfk,\bfk)$ that are
described as the homology of $B$ resp.\ $B_i$ we conclude that the complete
dglas $\Lhat(A)$ and $\Lhat(A_i)$ are quasi-isomorphic.  Therefore, by
Lemma~\ref{charact}, $I/I^2$ is quasi-isomorphic to $I_i/I^2_i$. This
proves that if each $A_i$ is of finite type, then so is $A$.

Conversely, suppose that $A$ is of finite type and let $I_i$ be an
augmentation ideal in $A_i$; then $\bfk$ becomes an $A_i$-module via
the augmentation $A_i\to A_i/I_i\cong \bfk$. Moreover, $\bfk$ is also
an $A$ module via the composition $A \stackrel{\pi_i}\to A_i \to
A_i/I_i\cong \bfk$ in which $\pi_i$ realizes the projection $p_i : A_1
\times \cdots \times A_n \to A_i$ in the homotopy category
$\hoHin$. The same argument as above shows that the dg spaces
$I_i/I^2_i$ and $I/I^2$ are quasi-isomorphic and thus $A_i$ is of
finite type.
\end{proof}

\begin{rem}
\label{bouraci-kladiva-opet-na-scene}
The quasi-isomorphism $\Lhat(A) \sim \Lhat(A_i)$ was crucial for the
proof of Proposition~\ref{p2}. It can be established differently.
Assume, for simplicity, that $n=i=2$. Since, by a standard spectral
sequence argument, $\Lhat(-)$ preserves quasi-isomorphisms, we only
need to prove that $\Lhat(A_1 \timesred A_2)$ is quasi-isomorphic to
$\Lhat(A_2)$, where $A_1 \timesred A_2$ is augmented via the composition
$A_1\timesred A_2 \to A_2 \to A_2/I_2 \cong \bfk$.  By
Proposition~\ref{monoidal} proved in the second part, $\Lhat(A_1
\timesred A_2) \cong \Lhat(A_1) \sqcup \Lhat(A_2)$. Moreover, by
Remark~\ref{zavolam}, $\Lhat(A_1) \sqcup \Lhat(A_2)$ is
quasi-isomorphic to $\Lhat(A_2)$.  This gives the requisite statement.
\end{rem}

\begin{rem}\def\bfK{{\mathbb K}}
It is \emph{not true} that for a cofibrant cdga $A$ of finite type and
any maximal ideal $I$ of $A$ the quotient $I/I^2$ has finite
dimensional cohomology. This is not even true for a homologically
connected $A$. Indeed, let $A$ be cofibrant cdga supplied with an
augmentation $A\to \bfk$. Let $\bfK \supset \bfk$ be an infinite dimensional
field extension of $\bfk$. The
factorization axiom in the closed model category $\algs$ expresses the composite map $A\to \bfK$ as
\[
\xymatrix{A\ar@{->>}[r] \ar@{^{(}->}[rd]_\sim & \bfk \ar@{^{(}->}[r] & \bfK
\\
&\tilde A \ar@{->>}[ru]_f&
}
\]
with a surjective map $f:\tilde{A}\to \bfK$,
where $\tilde{A}$ is quasi-isomorphic to $A$ and still cofibrant. The
kernel of $f$ is a maximal ideal $\tilde{I}$ in $\tilde{A}$ and
$\tilde{A}/\tilde{I}\cong \bfK$. Clearly
\[
\Der(A,\bfK) \cong \Der(A,\bfk)\ot_\bfk \bfK.
\]
On the
other hand, $\Der(A,\bfK)$ is quasi-isomorphic to $\Der(\tilde
A,\bfK) \cong (\tilde I/\tilde I^2)^*$.
We see that $H(\tilde I/\tilde I^2)$ may be infinite-dimensional
even when $\tilde A$ is of finite type.
\end{rem}

\section{Proofs of Theorems~A, B and C}
\label{proofs}

\begin{proof}[Proof of Theorem~A]
It is an obvious combination of Proposition~\ref{sec:other-results}
and Theorem~B which we prove below.
\end{proof}

\begin{proof}[Proof of Theorem~B]
We start by proving the more difficult second part of the theorem.  Assume that $A\in
\Hindc$ and let $\iota_i : \bfk \to \prod_{i \in \J}\bfk \cong H^0(A)$
be the canonical inclusion into the $i$th factor. Denote by $e_i
\in H^0(A)$, $i \in \J$, the idempotent $\iota_i(1)$ and by
$\loc {H(A)}{e_i}$ the localization of $H(A)$ at the multiplicative subset
generated by $e_i$.  By elementary algebra,
\begin{equation}
\label{eq:1}
H(A) \cong \prod_{i\in \J} \loc{H(A)}{e_i}.
\end{equation}
Choose a cochain $u_i \in A^0$ representing $e_i$.  The localization is
exact, therefore
$H\big(\loc A{u_i}\big) \cong \loc{H(A)}{e_i}$ and the natural cdga map
\begin{equation}
\label{eq:6}
A \to \prod_{i\in \J} \loc A{u_i}
\end{equation}
induces the isomorphism~(\ref{eq:1}) of cohomology.

As $H^0(\loc A{u_i}) \cong \loc{H^0(A)}{e_i} \cong \bfk$, each $\loc A{u_i}$ is
homologically connected; its cohomology in negative degrees clearly vanishes.
Therefore~(\ref{eq:6}) shows that each $A \in \Hindc$ is weakly
equivalent to a finite product of cdgas
from $\Hinc$. Denote, for the purposes of this proof, the full
subcategory of $\Hindc$ whose objects are
these products  by $\Hindcbar$. It follows from the above that the
corresponding homotopy categories $\hoHindc$ and $\hoHindcbar$ are equivalent.
Theorem~\ref{sec:main-results-1} then implies that $\hoHindcbar$ is
equivalent to $\Tlum(\hoHinc)$. This finishes the proof of the second~part.

The first part can be proved in exactly the same way, but the
situation admits a simplification. Assume the same notation as above.
The cdga $A$ is now non-negatively graded, so there are no
$0$-boundaries in $A^0$, thus $u_i$ is unique and it is an idempotent.
The multiplication with $u_i$ defines the projection $\pi_i : A \epi u_i
A$ which represents the localization map $A \to \loc A{u_i}$. Since
$\{u_i\}_{i\in \J}$ are orthogonal idempotents whose sum is
$1$, the system $\{\pi_i\}_{i\in \J}$ defines, instead of just a weak
equivalence~(\ref{eq:6}), a strict isomorphism
\[
A \cong \prod_{i \in \J} u_iA,
\]
in which $u_iA \in \BGc$. So $\BGdc$ is equivalent to its full
subcategory $\BGdcbar$ whose objects are finite products of cdgas
from $\BGc$, and the same is true also for the corresponding homotopy
categories, that is, $\hoBGdc \sim \hoBGdcbar$. The proof is finished
with the aid of Theorem~\ref{sec:main-results-1}.
\end{proof}

\begin{proof}[Proof of Theorem~C]
The equivalence of Example~\ref{v_susarne} clearly
restrict to the equivalence
\[
\honilpsimpdiscon \sim \Mult(\honilpsimpcon).
\]
By Proposition~\ref{p2}, the equivalences of Theorem~B restrict
to the equivalences
\begin{equation}
\label{zase_chripka}
\horathincatdc \sim \Tlum(\horathincat)
\ \mbox { and }\
\horatcatdc \sim \Tlum(\horatcat).
\end{equation}
The equivalence $\honilpsimpdiscon \sim \horatcatdc$ then
follows from the classical equivalence between $\horatcat$ and
$\honilpsimpcon$, cf.~\cite[Theorem~9.4]{bousfield-gugenheim}. It is easy
to show that it is in fact induced by adjunction~(\ref{houbicky}).
The equivalence $\horatcatdc \sim \horathincatdc$ is a
combination of~(\ref{zase_chripka}) with Proposition~\ref{sec:other-results}.
\end{proof}

\section{Augmented dg commutative associative algebras and pointed
  spaces.}
\label{sec:augm-dg-comm}

In this section we briefly outline the relationship between the
homotopy theory of augmented cdgas and pointed spaces. The results
formulated here will be used in Section~\ref{disc-spaces}. They are
more or less obvious analogues of the non-augmented theory developed
in the previous sections.

Recall (e.g.~\cite[Proposition 1.1.8]{hovey99:_model}) that for a given closed
model category $\catC$ and an object $O\in \catC$, the overcategory of
$O$ and the undercategory of $O$ are themselves closed model
categories with fibrations, cofibrations and weak equivalences created
in the category $\catC$.  Now consider the overcategory $\algs_+$ of
the initial object $\bfk$ in $\algs$; its objects are \emph{augmented} unital
cdgas; i.e. cdgas $A$ supplied with an augmentation $A\to \bfk$.  The
morphisms in $\algs_+$ are cdga maps respecting the augmentation.  The
category $\algs_+$ has a closed model structure inherited from
$\algs$.

Inside the category $\algs_+$ is the category $\algs_+^c$ consisting
of homologically connected cdgas and the category $\Hindcplus$ consisting of
augmented homologically disconnected cdgas. We also have the
corresponding subcategories of non-negatively graded augmented cdgas,
indicated by $\geq 0$ in the subscript.
Similarly we have the
category $\SSet_+$ of pointed simplicial sets, the
\emph{undercategory} of the terminal simplicial set; it has a closed
model structure inherited from $\SSet$. Let us formulate an augmented
version of Theorem~A:

\begin{theoremAplus}
The inclusion $\BG_+ \subset \Hi_+$ induces an equivalence of the homotopy
categories $\hoBGdc_+$ and $\hoHindcplus$.
\end{theoremAplus}

For a category $\catC$ with the terminal object $*$ denote by
$\Mult(\catC)_+$ the category whose objects are the formal finite
coproducts $A_1 \sqcup \cdots \sqcup A_s$, $s \geq 1$, such that
$A_1,\ldots, A_{s-1}$ are objects of $\catC$, $A_s$ is an object of
the undercategory $*/\catC$ of $\catC$, and the Hom-sets are
\begin{eqnarray*}
\lefteqn{
\Mult(\catC)_+\big(A_1 \sqcup \cdots \sqcup A_s,B_1 \sqcup \cdots \sqcup
B_t\big)}
\\
&:=&\Mult(\catC)(A_1\sqcup\cdots \sqcup A_{s-1},
B_1\sqcup\cdots\sqcup
B_{t})\times */\catC(A_s,B_t),
\end{eqnarray*}
with the obvious composition law.

Assume that $\catC$ is a closed model category and consider $*/\catC$
with the induced closed model category structure. Since
the categories $\hounder$  and $*/ \hoC$ are not equivalent in general, unlike
the un-pointed case, we cannot apply the above construction to the
homotopy category $\hounder$ directly. We need to modify the above definition
by taking $\Mult\big(\hounder\big)_+$ the category with the same objects as
$\Mult(*/\catC)$, but with the morphism sets
\begin{eqnarray*}
\lefteqn{
\Mult\big(\hounder\big)_+\big(A_1 \sqcup \cdots \sqcup A_s,B_1 \sqcup \cdots \sqcup
B_t\big)}
\\
&:=&\Mult(\hoC)(A_1\sqcup\cdots \sqcup A_{s-1},
B_1\sqcup\cdots\sqcup
\overline B_{t})\times \hounder(A_s,B_t),
\end{eqnarray*}
where $\overline B_{t}$ denotes the object $B_t \in */\catC$
considered as an object of $\catC$ by forgetting the coaugmentation $*
\to B_t$.
The above definitions are designed to model
the category of disconnected \emph{pointed} simplicial sets:

\begin{example}
Denote by $\SSetcplus$ the category of pointed connected simplicial sets and
by $\SSetdcplus$ the category of pointed simplicial sets with finitely many
components. One then has $\SSetdcplus \cong \Mult_+(\SSetc)$ and the
same is obviously
true for the homotopy categories, i.e.~$\hoSSetdcplus \cong \Mult_+(\hoSSetcplus)$.

Note that the weak equivalences in $\SSetdcplus$  are pointed maps inducing bijections
on the sets of connected components as well as a weak equivalence on each connected component.
\end{example}

Dually, let $\catC$ be a category with an initial object $*$.  Denote
by $\Tlum(\catC)_+$ the category whose object are formal finite
products $A_1 \timesred \cdots \timesred A_s$, $s\geq 1$, where $A_1,\ldots,
A_{s-1}$ are objects of $\catC$ and $A_s$ is an object of the
overcategory $\catC/*$ of $\catC$. The morphism sets are
\begin{eqnarray*}
\lefteqn{
\Tlum(\catC)_+\big(A_1 \times \cdots \times A_s,B_1 \times \cdots \times
B_t\big)}
\\
&:=&
\Tlum(\catC)(A_1\times\ldots\times A_s,B_1\times\ldots B_{t-1})\times
\hbox{$\catC/*$}(A_s,B_t).
\end{eqnarray*}
The category $\Tlum\big(\hoover\big)_+$ is
defined by an obvious dualization of the definition of
$\Mult\big(\hounder\big)_+$ given above.
The following theorems are proved analogously to the non-augmented case.

\begin{theoremBplus}
Each homologically disconnected non-negatively graded augmented cdga
$A \in \BGdc_+$ is isomorphic to a finite product $A_1 \times \cdots
\times A_s$ of homologically connected cdgas $A_1,\ldots,A_{s-1} \in
\BGc$ and a homologically connected augmented cdga $A_s \in
\BGc_+$.  This isomorphism extends to a natural equivalence of
categories
\[
\hoBGdc_+ \sim \Tlum(\hoBGc)_+.
\]
Each homologically disconnected augmented cdga $A \in \Hindcplus$ is
weakly equivalent to a finite product $A_1 \times \cdots \times A_s$
of homologically connected cdgas $A_1,\ldots,A_{s-1} \in \Hinc$ and
a homologically connected augmented cdga $A_s \in \Hincplus$. As above,
one has an equivalence
\[
\hoHindcplus \sim \Tlum(\hoHincplus)_+.
\]
\end{theoremBplus}

\begin{theoremCplus}
The following three categories are equivalent.
\begin{itemize}
\item[--]
The homotopy category
$\honilpsimpdisconplus$ of pointed simplicial sets with finitely many components that
are rational and of finite type,

\item[--]
the homotopy category
$\horatcatdc_+$ of homologically disconnected  non-negatively gra\-ded
augmented cdgas of finite type over~${\mathbb Q}$, and

\item[--]
the homotopy category
$\horathincatdcplus$ of homologically disconnected ${\mathbb Z}$-graded
augmented cdgas of finite type over~${\mathbb Q}$.
\end{itemize}
\end{theoremCplus}

\part{The Lie-Quillen approach}

In this part we give an application of the developed theory to the
structure of MC spaces and describe
the second version of disconnected
rational homotopy theory based on dglas.

\section{The simplicial Maurer-Cartan space}
\label{sec:appl-struct-simpl-3}

We write $\algs_+$ for the category of \emph{augmented} cdgas; it is,
thus, an overcategory of $\ground\in\algs$. As such, it inherits from
$\algs$ the structure of a closed model category. The weak
equivalences and fibrations are still quasi-isomorphisms and
surjective maps respectively. Note that the product of two augmented
cdgas $A$ and $B$ is their \emph{fiber product} $A\hskip
-.03em\times_{\ground}\hskip -.03em
B$. We denote the augmentation ideal of $A\in\algs_+$ by $A_+$; it is
a possibly non-unital cdga. Conversely, given a non-unital cdga $B$
one can form a unital algebra $B_e$ obtained by adjoining the unit;
$B_e\cong \ground\oplus B$. Thus, the category $\algs_+$ is equivalent
to the category of non-unital cdgas.

\begin{defi}
\label{sec:appl-struct-simpl-1}
A \emph{ complete} dgla is an inverse limit of finite-dimensional
nilpotent dglas. The category of complete dglas and their continuous
homomorphisms will be denoted by $\L$.
\end{defi}

\begin{rem}
\label{sec:appl-struct-simpl-2}
The functor of linear duality establishes an anti-equivalence between
the category $\L$ and that of conilpotent Lie coalgebras,
cf.~\cite{block05:_andre_quill} where complete Lie algebras were called
\emph{pronilpotent} Lie algebras; we feel that this terminology might
not be ideal since, e.g.~an~abelian Lie algebra on a countably
dimensional vector space is not pronilpotent under this convention.

Let us show, however, that complete dglas $\g \in \L$ are pronilpotent in the
classical sense,
\begin{equation}
\label{eq:9}
\g = \lim_k \g/\g_k,
\end{equation}
where $\g = \g_1 \supset \g_2 \supset \cdots$ is the lower central
series. Assume $\g =  \lim_n\g^n$ where $\g^n$, $n \geq 0$, are
finite-dimensional nilpotent.  Since the
filtered limit of finite-dimensional vector spaces is exact, we easily verify
that $(\g^n)_k \cong \lim_n (\g^n)_k$ and that $\g/\g_k \cong \lim_n
\g^n/(\g^n)_k$ for each $k \geq 1$. The nilpotence of
$\g^n$ implies $\lim_k \g^n/(\g^n)_k  \cong \g^n$, therefore
\[
\lim_k \g/\g_k \cong \lim_k \lim_n \g^n/(\g^n)_k \cong
\lim_n \lim_k \g^n/(\g^n)_k
 \cong  \lim_n \g^n = \g
\]
as desired.
\end{rem}

\begin{defi}\label{CEHar}\
Let $\Lhat:\algs_+\mapsto\L$ be the functor associating to an
augmented cdga $A$ the complete dgla $\L(A)$ whose underlying graded
Lie algebra is $\hat{\LL}\Sigma^{-1} A_+^*$, the completed free Lie algebra
on $\Sigma^{-1} A_+^*$. The differential $d$ in $\L(A)$ is defined as
$d=d_I+d_{\it II}$ where $d_I$ is induced by the internal differential in
$A_+$ and $d_{\it II}$ is determined by its restriction onto $\Sigma^{-1} A^*$,
which is, in turn, induced by the product map $A\otimes A\to A$.

Dually,
let $\Cfun:\L\mapsto\algs_+$ be the functor associating to a complete
dgla $\g$ the (discrete) augmented cdga $\Cfun(A)$ whose underlying
graded Lie algebra is $S\Sigma \g^*$, the symmetric algebra on
$\Sigma \g^*$. The differential $d$ in $\Cfun(\g)$ is defined as
$d=d_I+d_{\it II}$ where $d_I$ is induced by the internal differential in
$\g$ and $d_{\it II}$ is determined by its restriction onto $\Sigma \g^*$,
which is, in turn, induced by the bracket map $\g\otimes \g\to \g$.
\end{defi}

\begin{rem}
The dgla $\Lhat(A)$ is also known as the \emph{Harrison complex} of an
augmented cdga $A$ whereas $\Cfun(\g)$ is an analogue of the
Chevalley-Eilenberg complex of $\g$, except that usually $\g$ is not
assumed to be complete (and in that case $\Cfun(\g)$ has to be
completed).
\end{rem}
The following result follows directly from the definition.

\begin{prop}
\label{adj_bis}
The functors $\Cfun$ and $\Lhat$ are adjoint, so there is a natural
isomorphism for $A \in \algs_+$ and $\g\in\L$:
\[
\L\big(\Lhat(A), \g\big) \cong \algs_+\big(\Cfun (\g), A\big).
\]
\end{prop}

\begin{rem}
\label{sec:appl-struct-simpl}
In Section~\ref{dual_Hinich} we endow $\L$ with the structure of a closed model
category, in such a way that the functors $\Lhat$ and $\Cfun$ will
become inverse equivalences of the corresponding homotopy
categories. We will also see that each cdga of the form $\Cfun(\g)$ is
cofibrant.
\end{rem}

\subsection{Simplicial mapping space for cdgas and dglas}
To construct a simplicial Hom in the category of complete dglas we are
forced, as an intermediate step, to deal with dglas endowed with a
linear topology, but which are \emph{not complete}. An example of such
a dgla is a tensor product $A\otimes \g$ where $\g$ is a complete dgla
and $A$ a discrete cdga that is infinite dimensional. Our convention
is that the tensor product of a discrete vector space $W$ and a~complete vector space
$V=\lim_\alpha  V_{\alpha}$   is
the topological vector space $W\otimes V:=  \lim_\alpha
(W \otimes V_{\alpha})$. Whenever two dglas $\g$ and $\h$ are endowed
with a linear topology we will write $\Hom_{\text{dgla}}(\g,\h)$ for
the set of \emph{continuous} Lie homomorphisms from $\g$ into $\h$.

Let $\g$ and $\h$ be two complete dglas.
We are going to construct a simplicial set of maps
$\L(\g,\h)_\bullet$. As expected, the simplicial enrichment will be
obtained by tensoring the dgla $\h$ with the  Sullivan-de~Rham algebra
$\Omega(\Delta^\bullet)$ of polynomial forms on
the standard topological cosimplicial
simplex, cf.~\cite[\S2]{bousfield-gugenheim}.
According to our conventions, $\Omega(\Delta^\bullet)$ will be
considered as a homologically graded cdga, so
$\Omega(\Delta^\bullet)\otimes\h$ will be a homologically graded dgla. In
particular, $\Omega^p(\Delta^\bullet)\otimes\h_q$ will be placed in
homological degree $q-p$.

\begin{defi}
Let $\L(\g,\h)_\bullet:=\Hom_{\text
{dgla}}\big(\g,\Omega(\Delta^\bullet)\otimes\h\big)$ with the faces and degeneracy maps
coming from the corresponding geometric maps on the cosimplicial simplex
$\Delta^*$.
\end{defi}

\begin{rem}
Note that $\Omega(\Delta^n)$, $n \geq 0$, is a discrete
infinite-dimensional cdga whereas
$\h$ is a complete dgla; thus $\Omega(\Delta^n)\otimes\h$ is a
topological, but not a complete dgla.
\end{rem}

Recall that the category $\algs_+$ also has a simplicial Hom: for
two augmented cdgas $A$ and $B$ let $\algs_+(A,B)_\bullet$ be the
simplicial set for which\label{Dnes_s_Jarcou_mechaci}
\[
\algs_+(A,B)_n=\algs_+\big(A,(\Omega(\Delta^n)\otimes B_+)_e\big).
\]

The following result is an enriched version of Proposition~\ref{adj_bis};
its proof is a straightforward inspection.

\begin{prop}\label{adjen}
For any augmented cdga $A$ and a complete dgla $\g$
there is a natural isomorphism of simplicial sets:
\begin{equation}
\label{adjenriched}
\L(\Lhat(A),\g)_\bullet \cong \algs_+(\Cfun(\g),A)_\bullet\ .
\end{equation}
\end{prop}

\begin{rem}
We defined the simplicial mapping space between two augmented cdgas in
this way in order for the adjunction (\ref{adjenriched}) to hold. The
simplicial mapping space in the category of (unital) cdgas is
different; for two (unital) cdgas $A$ and $B$ we have
\begin{equation}
\label{zitra_musim_do_Prahy}
\algs(A,B)_\bullet := \algs\big(A,B \otimes \Omega(\Delta^\bullet)\big).
\end{equation}
It is easy to see that if $A$ is an augmented cdga and $B$ is a (unital)
cdga then there is an isomorphism between simplicial sets
\[
\algs_+(A,B_e)_\bullet \cong \algs(A,B)_\bullet\ .
\]
\end{rem}

Denote by $\Sp$ the dgla spanned by two vectors $x$ \and $[x,x]$ with
$|x|=-1$ and the differential $d(x)=-\frac{1}{2}[x,x]$. The dgla $\Sp$
is designed to model the topological space $S^0$, the zero-dimensional
sphere (which is a disjoint union of two points). Note that $\Sp$ is
isomorphic to $\Lhat(\ground\times\ground)$, where $\ground\times
\ground$ is the augmented cdga (with vanishing differential) obtained
from the ground field $\ground$ by adjoining a unit.

\begin{defi}
\label{MCdef}
Let $\g$ be a complete cdga. Its \emph{MC simplicial set} $\MC_\bullet(\g)$ is the simplicial mapping space:
\[
\MC_\bullet(\g) :=\L(\Sp,\g)_\bullet\ .
\]
In other words, $\MC_\bullet(\g)$ is the simplicial space
of MC elements in $\g \otimes \Omega(\Delta^\bullet)$.
The set $\pi_0\MC_\bullet(\g)$ is called the \emph{MC
  moduli set} of $\g$ and will be denoted by $\MCmod(\g)$.
\end{defi}

\begin{example}
\label{ex1}
Fix a finite set $S$ and assign, to each $s \in S$, a degree $-1$ generator
$x_s$. Consider the complete dgla $\g_S := \hat\LL(x_s;\ s \in S)$
freely generated by the set $\{x_s\}_{s \in S}$, with the differential
given by $dx_s := -\frac 12 [x_s,x_s]$.
Let us describe the MC elements in $\g_S \ot \Omega(\Delta^n)$.

The vector space $\big(\g_S \ot \Omega(\Delta^n)\big)_{-1}$ consists of
expressions
\[
\textstyle
x = \sum_{s \in S} x_s \ot \alpha_s,\ \alpha_s \in \Omega^0(\Delta^n).
\]
The MC condition for $x$ reads
\[
\sum_{s}\left(
-\frac12[x_s,x_s] \ot \alpha_s - x_s \ot d\alpha_s + \frac12 [x_s,x_s]
\ot \alpha_s^2
\right)
+
\sum_{s' \not= s''} [x_{s'},x_{s''}] \ot \alpha_{s'} \alpha_{s''}=0.
\]
By the freeness of $\g_S$, the above equation is satisfied if and only
if
\begin{subequations}
\begin{align}
\label{1a}
d \alpha_s &= 0 \mbox { for each } s \in S,
\\
\label{1b}
\alpha_s(1-\alpha_s) &= 0 \mbox  { for each } s \in S \mbox { and}
\\
\label{1c}
\alpha_{s'} \alpha_{s''}&=0  \mbox  { for each } s',s'' \in S,\ s'
\not= s''.
\end{align}
\end{subequations}
Equation~(\ref{1a}) implies that $\alpha_s$ is a constant,
by~(\ref{1b}) the only possible values of this constant are $0$ or
$1$. Equation~(\ref{1c}) then implies that $\alpha_s = 1$ for at most
one $s$.

We conclude that $\MC_\bullet(\g_S)$ is the constant simplicial set
representing the discrete space $S \cup \{*\}$. If $S$ is arbitrary, we
define $\g_S := \lim_F \g_F$, the obvious limit over finite subsets
$F \subset S$. It is easy to check that the above calculation remains
valid, so we see that each
non-empty discrete space can be realized as the MC space of
some complete dg-Lie algebra. An interpretation of this example is
given in Remark~\ref{Jaruska} below.
\end{example}

\begin{rem}
\label{Jaruska_ma_hodne_prace.}
In light of Proposition~\ref{adjen} we see that the MC simplicial set
of a complete dgla $\g$ is isomorphic to
\[
\algs_+\big(\Cfun(\g),\ground\times\ground\big)_\bullet \cong
\algs\big(\Cfun(\g),\ground)_\bullet.
\] It is easy to see that
if two complete $\g$ and $\h$ are related by a filtered
quasi-isomorphism, then $\MC_\bullet(\g)$ and $\MC_\bullet(\h)$ are
weakly equivalent simplicial sets. We will not need this result.
\end{rem}

\subsection{Disjoint products of dglas and their MC spaces}

\label{Pozitri_do_Cambridge}
Recall that an MC element $\xi\in\h_{-1}$ in a dgla $\h$ allows one to
twist the differential $d$ in $\h$ according to the formula
$d^\xi(?)=d(?)+[?,\xi]$. The graded Lie algebra $\h$ supplied with the
twisted differential $d^\xi$ will be denoted by $\h^\xi$. The same
construction applies when $\h$ is a complete dgla, in that case
$\h^\xi$ will likewise be complete.

Given an arbitrary complete dgla $\g$ consider the complete dgla $\g*\Sp$ where $*$ stands for the coproduct in the category $\L$. Abusing the notation, we will write $x$ for the image of $x\in \Sp$ inside $\g*\Sp$ under the inclusion of dglas $\Sp\hookrightarrow \g*\Sp$.  It is clear that $x$ is an MC element in $\g*\Sp$. The twisted dgla  $(\g*\Sp)^x$ is analogous to adjoining a base point to a topological space. Generalizing this construction, we define a \emph{disjoint product} of two complete dglas; note that it also makes sense for ordinary (non-complete) dglas.

\begin{defi}
Let  $\g$ and $\h$ be two complete dglas. Their disjoint product is the complete dgla $(\g*\Sp)^x*\h$; it will be denoted by $\g\scup\h$.
\end{defi}

In particular, for a dgla $\g$ the twisted free product $(\g*\Sp)^x$
is now denoted as $\g\scup 0$. Even if $\g$ is non-negatively
graded, $\g\scup 0$ is not; this obviously applies also to
more general disjoint products of dglas.

\begin{rem}
\label{zavolam}
The operation of a disjoint product of two dglas was considered
(without this name) in \cite{BM} and, in a more general, operadic
context in \cite{chuang13:maurer-cartan}. It is not hard to prove that for any dgla $\g$
(complete or not) the dgla $\g\scup0$ is acyclic, cf.\ \cite[Theorem 5.7]{chuang13:maurer-cartan}
or \cite[Lemma 6.1]{BM}. It follows that for any two dglas
$\g$ and $\h$ the dgla $\g\scup\h$ is quasi-isomorphic to $\h$.

In particular, the operation of a disjoint product of two (complete)
dglas is very far from being commutative. This phenomenon has the
following topological explanation, or analogue. Given two
\emph{pointed} topological spaces $X$ and $Y$ define their `disjoint
product' as the disjoint union $X\scup Y$ having its base point in
$Y$. Clearly, the operation of disjoint product of two topological
spaces is likewise noncommutative, owing to the asymmetric placement
of the base point.
\end{rem}

The following statement `measures' the non-commutativity of the
disjoint product.

\begin{proposition}
The disjoint products $\g\scup \h$ and $\h\scup \g$ are
 related by a twist; namely the element $-x$ is an MC element in
 $\g\scup \h\cong (\g*\Sp)^x*\h$ and
 there is an isomorphism of dglas
 \[
(\g\scup \h)^{-x}\cong \h\scup \g.
 \]
\end{proposition}

\begin{proof}
The isomorphism in the proposition is the identity map on $\g$ and
$\h$ and takes $x\in\g\scup \h$ to $-x\in\h\scup \g$. It is
straightforward to check that it has the desired property.
\end{proof}

However, the following result shows that the operation of the
disjoint product of dglas is associative, just as its topological
counterpart.

\begin{prop}
\label{Jarka_Ztratila_doklady}
For any three dglas $\g,\h,\a$ there is a natural isomorphism of dglas:
\[
(\g\scup \h)\scup\a\cong \g\scup(\h\scup\a).
\]
\end{prop}
\begin{proof}
We have an isomorphism of graded Lie algebras, disregarding the differential:
\begin{equation}\label{assoc}
\g\scup(\h\scup\a)\cong (\g*\Sp)*((\h*\Sp)*\a).
\end{equation}
Denote the generator of the first copy of $\Sp$ by $x$ and the
generator of the second copy of $\Sp$ by $y$ in (\ref{assoc}). Then we
have by definition an isomorphism of dglas:
\[
 \g\scup(\h\scup\a)\cong (\g*\Sp)^x*((\h*\Sp)^y*\a).
 \]
Similarly there is an isomorphism of dglas
\[
(\g\scup \h)\scup\a\cong (((\g*\Sp)^{x^\prime}*\h)*\Sp)^{y^\prime}*\a,
\]
where $x^\prime$ and $y^\prime$ have an obvious meaning.

Now consider the map $f:\g\scup(\h\scup\a)\cong (\g\scup
\h)\scup\a$ which is the identity on $\g,\h$ and $\a$,
$f(x):=x^\prime+y^\prime$ and $f(y):=y^\prime$.  A straightforward
calculation shows that $f$ is compatible with the differentials and
is, therefore, an isomorphism as claimed.
\end{proof}

\begin{rem}
The operation of disjoint product \emph{does not} make the category of dglas into a monoidal category, since there is no unit object. However, one can check that
the pentagon condition for the operation $\scup$ is satisfied and thus, in any expression involving multiple disjoint products the arrangement of brackets does not matter, up to a natural isomorphism.
\end{rem}

\begin{rem}
\label{Jaruska}
As an exercise we recommend to verify that for finite $S$ the algebra $\g_S$ from
Example~\ref{Jaruska} is isomorphic to the disjoint product of $0$
with itself, iterated ${\rm card\/}(S)$-times, that is
\[
\g_{\emptyset} \cong 0,\ \g_{\{*\}} \cong 0 \sqcup 0,\ \g_{\{*,*\}}
\cong (0 \sqcup 0) \sqcup 0 \cong 0 \sqcup (0 \sqcup 0), \ldots
\]
The fact that $\g_S$ represents the disjoint union $S \cup \{*\}$
established in Example~\ref{Jaruska} is thus corroborated by
Theorem~\ref{theoremF}.
\end{rem}

Define a {\em connected cover\/} $\oh$ of a dgla $\h =
(\h,[-,-],\partial)$ as the sub-dgla of $\h$ given by
\begin{equation}
\label{eq:10}
\oh_n := \left\{
\begin{array}{ll}
\h_n&\mbox {for }n > 0,
\\
\Ker(\partial : \h_0 \to \h_{-1})&\mbox {for }n=0, \mbox {and}
\\
0&\mbox {for }n <0.
\end{array}
\right.
\end{equation}
If $\h$ is complete then $\oh$ is clearly complete as well. The map
 $H_n(\oh) \to H_n(\h)$ induced by the inclusion $\iota: \oh
 \hookrightarrow \h$ is an isomorphism for $n \geq 0$.

For an MC element $\xi \in \h_{-1}$ denote by $\phi^\xi : \MC(\h^\xi) \to
\MC(\h)$ the isomorphism $\eta \mapsto \eta + \xi$, $\eta \in
\MC(\h^\xi)$, and by
$\phi_\bullet^\xi : \MC_\bullet(\h^\xi) \to
\MC_\bullet(\h)$ the obvious induced map of simplicial sets. Our proof
of Theorem~\ref{sec:main-results-3} will be based on

\begin{proposition}
\label{sec:introduction}
Let $\g$ be a complete dgla and $\xi \in \g_{-1}$ an MC element.  Then the composition
\[
\MC_\bullet(\overline{\g^\xi}) \stackrel{\MC_\bullet(\iota)}\vlra
\MC_\bullet(\g^\xi) \stackrel{\phi^\xi_\bullet}{\longrightarrow}
\MC_\bullet(\g)
\] induces a
weak equivalence between $\MC_\bullet(\overline{\g^\xi})$ and the
connected component of $\MC_\bullet(\g)$ containing $\xi \in
\MC(\g)$.
\end{proposition}

\begin{proof}
\label{p1}
Let $\h$ be a complete dgla.
By Remark~\ref{Jaruska_ma_hodne_prace.} we have an
isomorphism of simplicial spaces
$\MC_\bullet(\h) \cong \algs\big(\Cfun(\h),\bfk\big)_\bullet$,
so
\[
\pi_n\big(\MC_\bullet(\h)\big) \cong
\pi_n\big(\algs(\Cfun(\h),\bfk)_\bullet\big), \ n \geq 1.
\]
As shown in
Section~\ref{dual_Hinich}, the cdga $\Cfun(\h)$ is
cofibrant, so the homotopy groups
of the simplicial space
$\algs\big(\Cfun(\h),\bfk\big)_\bullet$ can be, for $n \geq 1$, calculated
using~\cite[Proposition~8.12]{bousfield-gugenheim} with $X:= \Cfun(\h)$. One gets
\[
\pi_n\big(\algs(\Cfun(\h),\bfk)_\bullet\big) \cong
\Hom\big(H^n(\Sigma \h^*),\bfk)\big)
\cong H_{n-1}(\h).
\]
Since all isomorphisms above are functorial, they combine into a
functorial isomorphism
\begin{equation}
\label{eq:3bis}
\pi_n\big(\MC_\bullet(\h)\big) \cong H_{n-1}(\h),\ n \geq 1,
\end{equation}
valid for an arbitrary complete dgla. This isomorphism is for $n \geq
2$ an isomorphism of abelian groups, but we will not need this
result.

Let us return to the proof of our theorem. It is clear that the
the simplicial isomorphism $\phi^\xi_\bullet : \MC_\bullet(\g^\xi)
\cong \MC_\bullet(\g)$ induces an isomorphism between the connected
component containing the trivial MC element $0 \in \MC(\g^\xi)$ and the connected
component containing $\xi \in \g_{-1}$.
Furthermore, it is clear that the
simplicial set $\MC_\bullet(\overline{\g^\xi})$ is connected, with the
trivial MC element $0 \in \overline{\g^\xi}_{-1}$ its only $0$-simplex. It is therefore
enough to show that the inclusion
$\iota : \overline{\g^\xi} \hookrightarrow \g^\xi$ induces and isomorphism
\[
\pi_n\big(\MC_\bullet(\overline{\g^\xi})\big)
\cong \pi_n\big(\MC_\bullet(\g^\xi)\big),
\]
for arbitrary $n \geq 1$. By~(\ref{eq:3bis}), this is the same as showing that
$\iota$ induces an isomorphism
\[
H_{n-1}(\overline{\g^\xi}) \cong  H_{n-1}(\g^\xi),
\]
for each $n \geq 1$. The last isomorphism easily follows from the
definition~(\ref{eq:10}) of the connected cover $\overline{\g^\xi}$ of  $\g^\xi$.
\end{proof}

\begin{proof}[Proof of Theorem~\ref{sec:main-results-3}]
The isomorphism~(\ref{zas_bouchani}) follows from Proposition~\ref{sec:introduction}, the
weak equivalence from~(\ref{zas_bouchani}) and Theorem~\ref{theoremF}.
\end{proof}

\section{Proof of Theorem~\ref{theoremF}}
\label{disjointproof}

We start this section by the following auxiliary statement.

\begin{lemma}\label{simpl}
Let $A$, $B$ and $C$ be unital cdgas with $A$ and $B$ cofibrant and
such that $H^0(C)$ has no idempotents different from $1$.  Denote by
$\overline{A\timesred B}$ a cofibrant approximation of $A\timesred B$. Then
there is a weak equivalence of simplicial sets
\[
\algs(\overline{A\timesred B},C)_\bullet
\ \simeq\ \algs(A,C)_\bullet \cup \algs(B,C)_\bullet\ .
\]
\end{lemma}

\begin{proof}
Let $S$ be a finite, connected simplicial set and let $\Omega(S)$ be
the cdga of Sullivan-de~Rham forms on $S$ (denoted by
$A(S)$ in \cite{bousfield-gugenheim}).
Then we have the following standard adjunction isomorphism,
cf.~\cite[Lemma~5.2]{bousfield-gugenheim}:
\begin{equation}
\label{Kooin}
\SSet\big(S,\algs(\overline{A\times B},C)_\bullet\big)
\cong \algs\big(\overline{A\times B}, \Omega(S)\otimes C\big).
\end{equation}
Observe that this adjunction automatically extends to simplicially
enriched Homs, that is
\begin{equation}
\label{v_Koline}
\SSet\big(S,\algs(\overline{A\times B},C)_\bullet\big)_\bullet
\cong \algs\big(\overline{A\times B}, \Omega(S)\otimes C\big)_\bullet.
\end{equation}
Indeed, by~(\ref{Kooin}) one has for each $n \geq 0$,
\[
\SSet\big(S \timesred \Delta^n,\algs(\overline{A\times B},C)_\bullet\big)
\cong \algs\big(\overline{A\times B}, \Omega(S \timesred
\Delta^n)\otimes C\big)
\cong \algs\big(\overline{A\times B}, \Omega(S) \!
\otimes\! C \!\otimes\! \Omega(\Delta^n)\big),
\]
consequently
\[
\SSet\big(S,\algs(\overline{A\times B},C)_\bullet\big)_n \cong
\algs\big(\overline{A\times B}, \Omega(S)
\otimes\! C\big)_n
\]
for each $n \geq 0$.
By naturality, the above individual isomorphisms assemble into an
isomorphism~(\ref{v_Koline}) of simplicial sets.  Since
$\overline{A\times B}$ is cofibrant, the simplicial set
$\algs(\overline{A\times B},C)_\bullet$ is
Kan. Adjunction~(\ref{v_Koline}) induces an isomorphism between the
sets of connected components of the corresponding simplicial sets,
therefore
\[
\big[S,\algs(\overline{A\times B},C)_\bullet\big]
\cong\big[\overline{A\times B}, \Omega(S)\otimes C\big]
\cong \big[{A\times B}, \Omega(S)\otimes C\big].
\]

The cdga $\Omega(S)\otimes C$ is bigraded and the first grading is
non-negative.  It is easy to show that in this situation any
idempotent in its cohomology must be contained in the $(0,0)$-graded
part which is isomorphic to $H^0(S)\otimes H^0(C)\cong H^0(C)$. But
$H^0(C)$ has no non-trivial idempotents other than $1$ by our assumption.  According
to Theorem 2.3
\[
\big[A\times B, \Omega(S)\otimes C\big]\cong \big[A,\Omega(S)\otimes
  C\big] \cup\, \big[B,\Omega(S)\otimes C\big]
\cong \big[S,\algs(A,C)_\bullet\big]\cup\, \big[S,\algs(A,C)_\bullet\big].
\]
Note that, because $S$ is connected, we have a natural
bijection
\[
\SSet \big(S,\algs(A,C)_\bullet \cup\algs(B,C)_\bullet\big)\cong
\SSet\big(S,\algs(A,C)_\bullet\big) \cup \SSet\big(S,\algs(A,C)_\bullet\big)
\]
and, therefore,
\[
\big[S,\algs(\overline{A\times B},C)_\bullet\big]\cong \big[A\times B, \Omega(S)\big].
\]
Thus, we have a natural bijection in the homotopy category of
simplicial sets, for any connected finite simplicial set $S$:
\[
\big[S,\algs(A,C)_\bullet\cup\algs(B,C)_\bullet\big]\cong
\big[S,\algs(\overline{A\times B},C)_\bullet\big].
\]
The desired conclusion follows.
\end{proof}

\begin{rem}
Note that the above lemma is an enriched analogue of Theorem
\ref{sec:main-results-1}. Since $A\timesred B$ is not cofibrant in
general, the result will not be true without taking a cofibrant
replacement.  Also note that the lemma can be easily generalized to a
finite direct product of cofibrant cdgas.
\end{rem}

We need one preliminary result on the MC twisting of dglas. Let $\g$ be a
complete dgla and consider the cdga $\Cfun(\g)$. It is clear that the
set of MC elements in $\g$ is in 1-1 correspondence with cdga maps
(augmentations) $\Cfun(\g)\to\ground$, in particular the canonical
augmentation corresponds to the zero MC element.

Let us explain this correspondence explicitly. The evaluation at $\xi \in
\g_{-1}$ determines a~degree $+1$ linear map
\hbox{$\alpha_\xi : \g^* \to \ground$} which in
turn defines a degree $0$ linear map (denoted by the same symbol)
\hbox{$\alpha_\xi : \Sigma \g^* \to \ground$} from the space of
generators of $\Cfun(\g)$ to the ground field.
The latter map extends to a unique
morphisms of graded algebras $\epsilon_\xi : \Cfun(\g) \to
\ground$ which commutes
with the differentials, i.e.~makes $\Cfun(\g)$ an augmented cdga, if
and only if $\xi$ is Maurer-Cartan.

Likewise, define a degree $0$ linear map
$\beta_\xi : \Sigma\g^* \to \ground
\oplus \Sigma\g^* \subset \Cfun(\g^\xi)$ as $\beta_\xi := \alpha_\xi  +
\id_{\Sigma\g^*}$. One can easily verify that $\beta_\xi$ extends to a cdga's
isomorphism $\phi_\xi: \Cfun(\g) \to \Cfun(\g^\xi)$ such that the
diagram
\begin{equation}
\raisebox{-3em}{}
\label{cockova_polevka}
{
\unitlength=.9pt
\begin{picture}(100.00,30)(0.00,30.00)
\thicklines
\put(80.00,27){\makebox(0.00,0.00)[tl]{\scriptsize $\epsilon_0$}}
\put(20.00,27){\makebox(0.00,0.00)[tr]{\scriptsize $\epsilon_\xi$}}
\put(50.00,53.00){\makebox(0.00,0.00)[b]{\scriptsize $\phi_\xi$}}
\put(50.00,0.00){\makebox(0.00,0.00){$\ground$}}
\put(100.00,50.00){\makebox(0.00,0.00){$\Cfun(\g^\xi)$}}
\put(0.00,50.00){\makebox(0.00,0.00){$\Cfun(\g)$}}
\put(20.00,50.00){\vector(1,0){60.00}}
\put(90.00,40.00){\vector(-1,-1){33}}
\put(10.00,40.00){\vector(1,-1){33}}
\end{picture}}
\end{equation}
commutes.

The maps $\epsilon_\xi$ and $\phi_\xi$ above have a nice geometric meaning
if we interpret elements $f \in \Cfun(\g)$ or~$f
\in \Cfun(\g^\xi)$ as polynomial functions on $\Sigma \g$. Then
$\epsilon_\xi$ is the evaluation at $\vartheta := \Sigma \xi$
while $\phi_\xi$ is given by
the shift by~$\vartheta$, that is
\[
\epsilon_\xi(f) = f(\vartheta)\ \mbox { and }
\phi_\xi(f)(u) = f(u+ \vartheta),\ u \in \Sigma \g.
\]
Checking the commutativity of~(\ref{cockova_polevka}) in this language
boils down to
\[
\epsilon_0\big(\phi_\xi(f)\big) = \phi_\xi(f)(0) = f(\vartheta) = \epsilon_\xi(f).
\]

Given $\xi\in\MC(\g)$ there is a bijective correspondence
$\MC(\g)\to\MC(\g^\xi)$ defined by $\eta\mapsto\eta-\xi$. In particular,
the element $\xi\in\MC(\g)$ corresponds to the zero element in
$\MC(\g^\xi)$.

\begin{lemma}\label{secondfac}
Let $\g$ be a complete dgla. Then the augmented cdga \hbox{$\Cfun(\g\scup
0)$} is quasi-isomorphic to $\Cfun(\g) \times \ground$, with the
augmentation given by the projection to the {\em second\/} factor.
\end{lemma}

\begin{proof}
Consider the cdga $\Cfun(\g*\Sp)$. According to Corollary \ref{freeprod}
the augmented cdga $\Cfun(\g*\Sp)$ is quasi-isomorphic to
$\Cfun(\g)\times_{\ground}(\ground\times\ground)\cong\Cfun(\g)\times\ground$;
moreover, the augmentation of $\Cfun(\g)\times\ground$ factors through
the projection onto the \emph{first} factor. On the other hand, there
is, as in~(\ref{cockova_polevka}), an isomorphism
of cdgas (not respecting the augmentation)
$\phi_x:\Cfun(\g*\Sp)\to\Cfun(\g*\Sp)^x$ such that the induced augmentation
$\Cfun(\g*\Sp)\to \Cfun(\g*\Sp)^x\to\ground$ corresponds to the MC
element $x\in\MC(\g*\Sp)$.

Note that the augmentation of $\Cfun(\g*\Sp)$ corresponding to the MC element $x\in\MC(\g*\Sp)$ translates via the above quasi-isomorphism into the augmentation of $\Cfun(\g)\times\ground$ given by the projection onto the second factor. This proves the desired statement.
\end{proof}

We now return to the proof of Theorem~\ref{theoremF}.
It suffices to consider the
case when $\J$ consists of two elements. Let $\g$ and $\h$ be two
complete dglas. Then by Theorem \ref{freeprod} and Lemma
\ref{secondfac} the augmented cdga
$\Cfun(\g\scup\h)=\Cfun\big((\g\scup 0)*\h\big)$ is quasi-isomorphic
to\label{Jarusku_boli_hlava}
\[
\Cfun(\g\scup 0)\times_{\ground}\Cfun(\h)\simeq(\Cfun(\g)\times\ground)
\times_{\ground}\Cfun(\h)\cong
\Cfun(\g)\times\Cfun(\h).
\]

We conclude that the simplicial set $\MC_{\bullet}(\g\scup\h)=
\algs_+(\Cfun(\g\scup\h),\ground\times\ground)_\bullet$ is weakly
equivalent to
$\algs_+(\overline{\Cfun(\g)\times\Cfun(\h)},\ground\times\ground)_\bullet$. The
latter simplicial Hom is taken in the category of augmented cdgas; it
is isomorphic to the simplicial mapping space
$\algs(\overline{\Cfun(\g)\times\Cfun(\h)},\ground)_\bullet$ taken in the
category of (unital) non-augmented cdgas.  Now we have the following
weak equivalence of simplicial sets, which follows from Lemma
\ref{simpl}:
\[
\algs\big(\overline{\Cfun(\g)\timesred\Cfun(\h)},\ground\big)_\bullet
\simeq
\algs(\Cfun(\g),\ground)_\bullet
\cup\algs(\Cfun(\h),\ground)_\bullet\ .
\]
This finishes the proof of Theorem~\ref{theoremF}.

We close this section by two versions of an example which is a dgla
analogue of a disjoint union of a circle $S^1$ and an isolated
point. This example is simple enough to be worked out by hand, although
the calculations are still nontrivial.  We will see that, in this
particular case, an analogue of Theorem~\ref{theoremF} actually holds
without a completion. The general claim, in the non-completed context,
was made in \cite[Theorem 6.4]{BM}, but we have not been able to parse
the proof in op.~cit.
\begin{example}

\label{sec:introduction-8}
Let $x$ (resp.~$a$) be a generator of degree $-1$ (resp.~$0$). Denote by
$\ff$ the free non-complete Lie algebra $\LL(x,a)$ generated by $x$ and $a$,
with the differential given by $dx := -\frac12 [x,x]$, $da := 0$. In
this example we describe
MC elements $t$ in $\ff \otimes \Omega(\Delta^n)$, for an
arbitrary $n \geq 0$.

Let us start by observing that $\ff_{-1}$ has a basis
$\{e_i\}_{i \geq 0}$, with
\[
e_0 := x,\ e_1 := [a,x],\ e_2 := [a,[a,x]],\
e_3 := [a,[a,[a,x]]],\ldots
\]
{}From degree reasons,
each element of $\big(\ff \otimes \Omega(\Delta^n)\big)_{-1}$ has the
form
\begin{equation}
\label{eq:2.2}
t = \textstyle\sum_{i \geq 0} e_i\ot  \alpha_i  + a\ot \omega,
\end{equation}
for some $\alpha_i \in \Omega^0(\Delta^n)$, $i \geq 0$, and $\omega
\in \Omega^1(\Delta^n)$, with the assumption that only {\em finitely many\/}
$\alpha_i$'s are $\not=0$. The MC equation for $t$ reads
\[
\textstyle\sum_{i \geq 0} (de_i \ot \alpha_i - e_i \ot d\alpha_i)
+ a \ot d\omega + \frac12 \sum_{i,j \geq 0} [e_i,e_j] \ot
\alpha_i\alpha_j
-\sum_{i \geq 0} [a,e_i]\ot \omega\alpha_i   = 0
\]
(the term $\frac12 [a,a]\ot \omega^2$ clearly vanishes). Singling out
the parts in $\ff_0 \ot \Omega^2(\Delta^n), \ff_{-1} \ot
\Omega^1(\Delta^n)$ and $\ff_{-2} \ot \Omega^0(\Delta^n)$, respectively,
one gets the equations
\begin{subequations}
\begin{eqnarray}
\label{2a}
a\ot d\omega &=& 0,
\\
\textstyle
\label{2b}
\sum_{i \geq 0} (e_i \ot d\alpha_i +  [a,e_i]\ot \omega\alpha_i)   &=& 0,
\\
\textstyle
\label{2c}
\sum_{i \geq 0} de_i \ot \alpha_i + \frac12 \sum_{i,j \geq 0} [e_i,e_j] \ot
\alpha_i\alpha_j &=& 0.
\end{eqnarray}
\end{subequations}
Since, by definition,
$e_{i+1} = [a,e_i]$,~(\ref{2b}) implies that $d\alpha_0 = 0$,
i.e.~{\em $\alpha_0 \in \bfk$ is a constant\/}, while
\begin{equation}
\label{eq3}
d\alpha_{i+1} =- \omega\alpha_i,\ \mbox { for } i \geq 0.
\end{equation}
Equation~(\ref{2c}) means that $\sum_{i \geq 0} e_i\ot  \alpha_i$ is an MC element in $\ff \ot \Omega^0(\Delta^n)$, with $\Omega^0(\Delta^n)$
taken with the trivial differential. It expands to the
system
\begin{subequations}
\begin{eqnarray}
\label{3a}\textstyle
de_0 \ot  \alpha_0  &=& -\textstyle\frac12[e_0,e_0] \ot \alpha_0^2,
\\
\label{3b}\textstyle
de_1 \ot \alpha_1 &=&   - [e_0,e_1] \ot \alpha_0 \alpha_1,
\\
\label{3c}\textstyle
de_2 \ot \alpha_2 &=&   - [e_0,e_2] \ot \alpha_0 \alpha_2 -
\textstyle\frac12[e_1,e_1] \ot \alpha_1^2,
\\
\label{3d}\textstyle
de_2 \ot \alpha_3 &=&   - [e_0,e_3] \ot \alpha_0 \alpha_3 -
\textstyle[e_1,e_2] \ot \alpha_1\alpha_2,
\\ \nonumber
&\vdots&
\end{eqnarray}
\end{subequations}
Since $e_0 = x$ and $dx = - \frac12[x,x]$,~(\ref{3a}) is
equivalent to $\alpha_0 =
\alpha_0^2$, therefore $\alpha_0 \in \Omega^0(\Delta^n)$ is a constant
that equals either $0$ or $1$.

\noindent
{\bf Case $\alpha_0 =0$}. In this case~(\ref{3b}) reduces to $de_1 \ot
\alpha_1 = 0$. Since $de_1 \not=0$, it implies $\alpha_1 =
0$. Then~(\ref{3c}) reduces to $de_2 \ot \alpha_2 = 0$ so, by the same
argument, $\alpha_2 = 0$ and~(\ref{3d})  reduces to $de_3 \ot \alpha_3
= 0$.  We prove inductively that $\alpha_n = 0$ for
each $n \geq 0$. While~(\ref{2b}) is automatically
satisfied,~(\ref{2a}) requires
$d\omega = 0$. We conclude that the MC
elements $t$ in $\ff \ot \Omega(\Delta^n)$ with $\alpha_0 = 0$ are precisely
closed forms in $\Omega^1(\Delta^n)$.

\noindent
{\bf Case $\alpha_0 =1$}. Observe first that $\xi := e_0 \ot 1 = x \ot 1$ is
an MC element of this form. We are going to prove that it is the
{\em only\/} MC-element of $\ff \ot \Omega(\Delta^n)$ with $\alpha_0=1$.

The Lie algebra $\ff_0$ is one-dimensional, spanned by $a$. Consider,
only for the purposes of this example,
the completion $\hatff$ of $\ff$ with respect to the grading
by the number of $a$'s. Then $\exp(\ff_0)$ acts on the set of MC
elements in $\hatff$, in particular,
\[
\exp (a) e_0 = e_0 + e_1 + \frac1{2!} e_2 +
\frac1{3!}e_3 + \frac1{4!}e_4 +    \cdots
\]
is an MC element in $\hatff$. This means that
\begin{eqnarray*}
de_0 &=& - \textstyle\frac12 [e_0,e_0],
\\
de_1 &=&   - [e_0,e_1]
\\
de_2  &=&   - [e_0,e_2] -[e_1,e_1] ,
\\
de_3 &=& -[e_0,e_3] - 3[e_1,e_2],
\\
&\vdots&
\end{eqnarray*}
Substituting this to~(\ref{3a})--(\ref{3d}), we get
\begin{subequations}
\begin{eqnarray}
[e_0,e_0] \ot  \alpha_0  &=& [e_0,e_0] \ot \alpha_0^2,
\\
{}[e_0,e_1] \ot \alpha_1 &=&  [e_0,e_1] \ot \alpha_0 \alpha_1,
\\
\label{4c}
([e_0,e_2] +[e_1,e_1])
 \ot \alpha_2 &=&   [e_0,e_2] \ot \alpha_0 \alpha_2 +
\textstyle\frac12[e_1,e_1] \ot \alpha_1^2,
\\
\label{4d}
([e_0,e_3] + 3[e_1,e_2]) \ot \alpha_3
&=&
[e_0,e_3] \ot \alpha_0 \alpha_3
+[e_1,e_2] \ot \alpha_1\alpha_2,
\\ \nonumber
&\vdots&
\end{eqnarray}
\end{subequations}

If $\alpha_1 = 0$,~(\ref{4c}) reduces to $[e_1,e_1] \ot \alpha_2 = 0$,
which implies that $\alpha_2 = 0$.  Then~(\ref{4d}) implies that
$[e_1,e_2] \ot \alpha_3 = 0$, so $\alpha_3 = 0$.  Continuing this
process we prove that actually $\alpha_n = 0$ for all $n \geq 1$.
Equation~(\ref{eq3}) with $i=0$ gives $\omega = 0$. So the only MC
element in $\hatff \otimes \Omega(\Delta^n)$ with $\alpha_0 = 1$ and
$\alpha_1 = 0$ is $\xi := x \ot 1$.

Let now $t$ be an MC element in $\ff \otimes \Omega(\Delta^n)$ with
$\alpha_0 = 1$ and $\alpha_1 \not= 0$.  Then $\exp(-a \ot \alpha_1) t$
is clearly an MC element  in $\hatff \otimes \Omega(\Delta^n)$ with
$\alpha_1 = 0$, so $\exp(-a \ot \alpha_1) t = x \ot 1$
or, equivalently,
\begin{equation}
\label{eq4}
t = \exp(a \ot \alpha_1) x =  a\ot d\alpha_1 +  e_0 \ot 1 +
e_1 \ot
\alpha_1 + \frac1{2!} e_2 \ot \alpha_1^2 +
\frac1{3!}e_3 \ot \alpha_1^3 + \cdots
\end{equation}
In particular, $t$ has {\em infinitely many\/} non-trivial elements,
so {\em it is not\/} an MC element in the uncompleted dgla $\ff \ot
\Omega(\Delta^n)$.  So $\alpha_1$ must be $0$ therefore the only MC
element of $\ff \ot \Omega(\Delta^n)$ with $\alpha_0 = 1$ is $x \ot
1$. We arrive at the following conclusion.

\begin{claim}
The only MC elements of $\ff\ot \Omega(\Delta^n)$
are either $x$ or closed $1$-forms in $\Omega^1(\Delta^n)$. In other
words, one has an isomorphism of simplicial sets
\[
\MC_\bullet(\ff) \cong  \MC_\bullet(\fa) \cup \MC_\bullet(0),
\]
where $0$ is the trivial Lie algebra and $\fa$ the one-dimensional
abelian Lie algebra generated by $a$. So $\ff$ serves as an algebraic model
for the disjoint union $S^1 \cup \{*\}$.
\end{claim}
\end{example}

\begin{example}
Let us analyze the completed version of the previous example, i.e.~the
MC elements in $\hatff  \hatot \Omega(\Delta^n)$, where $\hatff$
is the completion of the Lie algebra $\ff =\LL(x,a)$. Of course, we know the answer from Theorem~\ref{theoremF}
but the explicit calculation that could be performed in this simple case is still instructive. Most of the
work done in the previous example applies in the complete situation with only small
modifications. For instance,  elements of $(\hatff \hatot
\Omega(\Delta^n))_{-1}$ are of the form~(\ref{eq:2.2}), but the sum may now
have {\em infinitely many\/} nontrivial terms.

Equations~(\ref{2a})--(\ref{3d}) still take place,
as does the separation into Case $\alpha_0 = 0$ and
Case $\alpha_0 =1$. Precisely as in
Example~\ref{sec:introduction-8} we show that the only MC element with
$\alpha_0 = 0$ is the trivial one. In contrast with the uncomplete
case we, however, do not require the sum~(\ref{eq4}) to be {\em
  finite\/}, so $\alpha_1$ there may be an arbitrary function from
$\Omega^0(\Delta^\bullet)$.

We conclude that an MC element $t$ in $\hatff  \otimes
\Omega(\Delta^n)$ is either $1\ot \omega$ for a closed form
$\omega \in \Omega^1(\Delta^\bullet)$ if $\alpha_0 = 0$,  or
\[
t = \exp(a \otimes \alpha)x,\ \mbox { for some $\alpha \in
  \Omega^0(\Delta^n)$},
\]
if $\alpha_0 =1$. Since the isotropy subgroup of $x$ under the
action of $\exp\big(a \otimes \Omega^0(\Delta^n)\big)$ is trivial, we  conclude

\begin{claim}
One has an isomorphism of simplicial sets
\[
\MC_\bullet(\hatff) \cong  \MC_\bullet(\fa) \cup \Omega^0(\Delta^\bullet).
\]
\end{claim}

By \cite[Proposition~1.1.]{bousfield-gugenheim}, $\Omega^0(\Delta^\bullet)$ is
contractible, therefore $\MC_\bullet(\hatff)$ has the simplicial
homotopy type of  $\MC_\bullet(\fa) \cup \MC_\bullet(0)$. Thus
$\hatff$ also serves as an algebraic model
for the disjoint union $S^1 \cup \{*\}$. Note that the simplicial sets
$\MC_\bullet(\hatff)$ and $\MC_\bullet(\ff)$ are weakly equivalent, but \emph{not} isomorphic.
\end{example}

\section{Dual Hinich correspondence}
\label{dual_Hinich}

The purpose of this section is to establish an analogue to the main
result of Hinich's paper~\cite{hinich01:_dg}, giving a closed model
category structure on cocomplete cocommutative dg coalgebras. A~formal
dualization of Hinich's result states that the category of complete
cdgas has a closed model structure, cf.~\cite{lazarev13:_maurer} where this
formulation was explicitly spelled out. The result proved in this
section should be viewed as the Koszul dual to Hinich's. We prove this
Koszul dual version by suitably adapting Hinich's methods. The use of
the associative version of this result contained in \cite{Lefevre-hasegawa02, positselski}
allows us to shorten the proof in several places.

Recall that in Definition~\ref{CEHar} we introduced
an adjoint pair of functors $\Cfun$ and $\hat\Lfun$
between the categories $\algs_+$ of augmented cdgas and $\L$ of
complete dglas.  We will write $i_{\g}:\hat\Lfun\Cfun(\g)\to \g$ and
$i_A:\Cfun\hat\Lfun(A)\to A$ for the counits of this adjunction.  Our goal
is to establish a closed model category structure on $\L$ in such a
way that the functors $\Cfun$ and $\hat\Lfun$ would induce an equivalence
on the level of homotopy categories.

The adjoint functors $\hat\Lfun$ and $\Cfun$ can be `embedded' into an
adjunction between bigger categories of \emph{associative}
algebras. We will denote the category of augmented dg associative (discrete)
algebras (dgas) by $\Ass_+$; algebras in  $\Ass_+$ will be {\em cohomologically
graded\/}. The augmentation ideal in an augmented dga $A$ will be denoted
by $A_+$, as in the commutative case. Let us remind the details of the
corresponding adjunction following \cite{Lefevre-hasegawa02}. First, we need a
relevant definition.

\begin{definition}
\label{sec:dual-hinich-corr-1}
A \emph{complete augmented} dga is, by definition, an inverse limit of
finite-dimensional nilpotent augmented dgas.\footnote{We call an
augmented algebra nilpotent if its augmentation ideal is nilpotent in
the usual sense.} The category of complete dgas and their continuous
homomorphisms will be denoted by $\Asshat_+$.  Algebras in $\Asshat_+$
are assumed to be \emph{homologically~graded}.
\end{definition}

Algebras of Definition~\ref{sec:dual-hinich-corr-1} are complete in
the sense of~\cite[Appendix~A.1]{quillen:Ann.ofMath.69}. For instance, repeating the
arguments used in Remark~\ref{sec:appl-struct-simpl-1} to prove that
complete dglas are pronilpotent, one may easily show that $\lim_k
A/A_+^k \cong A$, which is Condition (c) on
page~265 of that paper.

\begin{rem}
The functor of linear duality establishes an anti-equivalence between
the category $\Asshat_+$ and that of conilpotent dg coalgebras,
cf.~\cite{hamilton-lazarev1} where complete dgas were called \emph{formal} dgas;
again, we opted to change this terminology since formality often has
a~different meaning in homological algebra.
\end{rem}

\begin{definition}
Let $\Bhat:\Ass_+\mapsto\Asshat_+$ be the functor associating to a dga $A$
the complete dga $\Bhat(A)$ whose underlying graded algebra is
$\hat{T}\Sigma^{-1} A_+^*$,
the completed tensor algebra on $\Sigma^{-1}
A^*_+$. The differential $d$ in $\hat{T}\Sigma^{-1} A^*_+$ is defined as
$d=d_I+d_{\it II}$ where $d_I$ is induced by the internal differential in
$A$ and $d_{\it II}$ is determined by its restriction onto $\Sigma^{-1} A^*_+$,
which is, in turn, induced by the product map $A_+\otimes A_+\to A_+$.

Likewise, let $\B:\Asshat_+\mapsto\Ass_+$ be the functor associating to a
complete dga $C$ the (discrete) dga $\B(C)$ whose underlying graded
algebra is $T\Sigma C_+^*$, the tensor algebra on $\Sigma
C^*_+$. The differential $d$ in $\B(C)$ is defined as
$d=d_I+d_{\it II}$ where $d_I$ is induced by the internal differential in
$C$ and $d_{\it II}$ is determined by its restriction onto $\Sigma C_+^*$,
which is, in turn, induced by the product $C_+\otimes C_+\to C_+$.
\end{definition}

\begin{rem}
We will refer to either functor $\B$ or $\Bhat$ as the
\emph{cobar-construction}. Note that for a dga $A$ its
cobar-construction $\Bhat(A)$ is a complete Hopf algebra, whose
space $\Sigma^{-1} A^*_+$ of algebra generators consists of primitive
elements. If $A$ is commutative, then the
differential of $\Bhat(A)$ takes $\Sigma^{-1} A^*_+$ to primitives and
thus makes $\Bhat(A)$ a dg Hopf algebra.
\end{rem}

\begin{proposition}
\label{bar}\
\begin{enumerate}
\item
Let $A$ be a cdga. Then there is a natural isomorphism of complete dglas:
\[
\Prim\Bhat(A)\cong\hat\Lfun(A).
\]
\item
Let $\g$ be a complete dgla and
$\hat U\g$ be its completed universal enveloping algebra.
Then there is a quasi-isomorphism of dgas
\[
\B(\hat U\g)\simeq \Cfun(\g).
\]
\end{enumerate}
\end{proposition}

\begin{proof}
The first statement is a consequence of the completed version of the
well-known fact that the primitive elements in the tensor algebra on a
graded vector space form the free Lie algebra on the same vector space
which easily follows from the Appendix to~\cite{quillen:Ann.ofMath.69}.

To prove the second statement, note that the inclusion of $\g$ into
$\hat U\g$ as the space of primitive elements induces a map of dgas
$\B(\hat U\g)\to\Cfun(\g)$. To see that the latter map is a
quasi-isomorphism it suffices to assume that $\g$ is
finite-dimensional nilpotent; the general result will be obtained
by passing to the limit. Denote by $\tilde{\g}$ the graded Lie algebra
with the same underlying space as $\g$ and the vanishing
differential. Then we have spectral sequences with $E_1$ terms
$H\big(\B(\hat U\tilde{\g})\big)$ and $H\big(\Cfun(\tilde{\g})\big)$,
which converge
to $\B(\hat U\g)$ and $\Cfun(\g)$ respectively. The map $\B(\hat
U\g)\to\Cfun(\g)$ gives a map between these spectral sequences.  It is
therefore enough to prove that the map $\B(\hat
U\tilde{\g})\to\Cfun(\tilde{\g})$ is a quasi-isomorphism.

As in the proof of Lemma~\ref{nilp} we use the filtration of
$\tilde{\g}^*$ induced by the shifted lower central series of
$\tilde{\g}$. This filtration induces, in the usual way, increasing
exhaustive and complete filtrations of $\B(\hat U\tilde{\g})$ and
$\Cfun(\tilde{\g})$ compatible with the map $\B({\hat
U}\tilde{\g})\to\Cfun(\tilde{\g})$. This brings us to the case when
$\tilde{\g}$ is abelian, the desired result then follows from the
calculation of $\Tor_{\Cfun(\tilde{\g})}(\ground,\ground)$ via the Koszul
complex.
\end{proof}

The category $\Ass_+$ has a closed model category structure, by a
general result of Hinich, cf.~\cite{hinich:CA97}. Namely, weak
equivalences in $\Ass_+$ are quasi-isomorphisms of augmented dgas and
fibrations are surjective homomorphisms. The category $\Asshat_+$ of complete dgas
also admits the structure of a closed model category, as follows.

\begin{definition}
A morphism $f:A\to B$ in $\Asshat_+$ is called
\begin{enumerate}
\item
a \emph{weak equivalence} if
${\B}(f):\B(B)\to\B(A)$
is a quasi-isomorphism in $\Ass_+$;
\item
a \emph{fibration} if $f$ is surjective; if, in addition, $f$ is a weak equivalence then $f$ is called an \emph{acyclic fibration};
\item
a \emph{cofibration} if $f$ has the left lifting property with respect to all acyclic fibrations. That means that in any commutative square
\[
\xymatrix{A\ar_f[d]\ar[r]&C\ar^g[d]\\
B\ar@{-->}[ur]\ar[r]&D}
\]
where $g$ is an acyclic fibration there exists a dotted arrow making
the whole diagram commutative.
\end{enumerate}
\end{definition}

\begin{theorem}
\label{closedass}
The category $\Asshat_+$ is a closed model category with fibrations,
cofibrations and weak equivalences defined as above. Moreover, it is
Quillen equivalent to the closed model category $\Ass_+$ via the
adjunctions $\B$ and $\Bhat$.
\end{theorem}
\begin{proof}
This is just a reformulation of Th\'eor\`eme 1.3.1.2 of \cite{Lefevre-hasegawa02}, taking into account the anti-equivalence between complete dgas and conilpotent dg coalgebras.
\end{proof}

\begin{rem}
The above result is an associative analogue of Hinich's theorem
\cite{hinich01:_dg} on the existence of a closed model category on dg
conilpotent coalgebras or, equivalently, on complete cdgas. All objects in
$\Asshat_+$ are fibrant and cofibrant objects correspond precisely
to $A_\infty$ algebras, cf.,~for example, \cite{hamilton-lazarev1} for a treatment of
$A_\infty$ algebras relevant to the present context.
\end{rem}

We will now construct a closed model category structure on $\L$ using
Theorem~\ref{closedass} as a shortcut; a more direct approach,
essentially repeating the original Hinich's argument in the dual context
is also possible.

\begin{definition}
\label{sec:dual-hinich-corr-2}
A morphism $f:\g \to \h$ in $\L$ is called
\begin{enumerate}
\item
a \emph{weak equivalence} if
${\Cfun}(f):\Cfun(\h)\to\Cfun(\g)$
is a quasi-isomorphism in $\algs_+$;
\item
a \emph{fibration} if $f$ is surjective; if, in addition, $f$ is a weak equivalence then $f$ is called an \emph{acyclic fibration};
\item
a \emph{cofibration} if $f$ has the left lifting property with respect to all acyclic fibrations.
\end{enumerate}
\end{definition}

We will prove that the above structures make $\L$ a closed model
category. Our proof will be based on the following results.

\begin{proposition}\
\label{adj}
\begin{enumerate}
\item
Let $A$ be an augmented cdga; then $i_A:\Cfun\hat\Lfun(A)\to A$ is a quasi-isomorphism.
\item
Let $\g$ be a complete dgla; then $i_{\g}:\hat\Lfun\Cfun(\g)\to \g$ is
a weak equivalence,
i.e.~$\Cfun(i_{\g}):\Cfun(\g) \to \Cfun\big(\hat\Lfun\Cfun(\g)\big)$ is a
quasi-isomorphism.
\end{enumerate}
\end{proposition}

\begin{proof}
There is a quasi-isomorphism $\Cfun\hat\Lfun(A)\simeq
\B\big({\hat U}\hat\Lfun(A)\big)$ in $\Ass_+$ by Proposition \ref{bar}(2).
Furthermore, the dga $\B\big({\hat U}\hat\Lfun(A)\big)=\B\Bhat(A)$ is
quasi-isomorphic to $A$ by Theorem \ref{closedass}. This proves (1).

Let us prove that the induced map
$\Cfun(i_\g):\Cfun\big(\hat\Lfun\Cfun(\g)\big)\to\Cfun(\g)$ is a
quasi-isomorphism of augmented cdgas. By Proposition \ref{bar}(2) we
have a quasi-isomorphism of dgas
\[
\Cfun\big(\hat\Lfun\Cfun(\g)\big)\simeq \B {\hat
  U}\big(\hat\Lfun\Cfun(\g)\big)=\B\Bhat\Cfun(\g)
\]
and, by Theorem \ref{closedass}, $\B\Bhat\Cfun(\g)$ is quasi-isomorphic
to $\Cfun(\g)$ as required.
\end{proof}

The following statement is an analogue of the `Key Lemma' of
\cite[p.~223]{hinich01:_dg}.

\begin{lemma}\label{keylem}
Let $A$ be a cdga, $\g$ be a complete dgla and $f:A\to \Cfun(\g)$ be a
surjective map. Consider the pushout diagram
\[
\xymatrix
{
\hat\Lfun\Cfun(\g)\ar^{i_\g}[d]\ar^{\hat\Lfun(f)}[r]&\hat\Lfun(A)\ar_j[d]\\
\g\ar[r]&\hskip .5em\a\ .
}
\]
Then $\Cfun(j):\Cfun(\a) \to \Cfun\hat\Lfun(A)$ is a quasi-isomorphism of cdgas.
\end{lemma}

In the proof of Lemma~\ref{keylem} we will use the following technical statement.

\newtheorem*{sublemma}{Sublemma}

\begin{lemma}
\label{sublemma}
Let $\phi : (A',d') \to (A'',d'')$ be a chain map. Assume that $A' =
\bigcup_{p,q}F_{p,q}A'$ is a~finite double filtration which is
descending in the sense that
\[
F_{p+1,q}A'\cup F_{p,q+1}A' \subset F_{p,q}A' \ \mbox { for each }\ p,q.
\]
Assume also that $d'$ is the sum $d'_1+d'_2$ of two
differentials such that
\[
d'_1(F_{p,q}A') \subset F_{p,q}A' \ \mbox { and } \
d'_2(F_{p,q}A')\subset F_{p-1,q+1}A' \ \mbox { for each }\ p,q.
\]

Suppose that $A''$ has a double filtration
with the similar properties and that $\phi$ is compatible with
these filtrations.
If the induced map
\[
\phi_{p,q} :\frac{F_{p,q}A'}{ F_{p+1,q}A'\cup F_{p,q+1}A'}
\longrightarrow \frac{F_{p,q}A''}{ F_{p+1,q}A''\cup F_{p,q+1}A''}
\]
of the quotients, with the differentials induced by the
`untwisted' parts $d'_1$ resp.~$d''_1$, is a~quasi-isomorphism for
each $p$ and $q$, then $\phi$ is a quasi-isomorphism, too.
\end{lemma}

\begin{proof}
Consider the spectral sequences in the $q$-direction. The $E_1$-sheets of
these spectral sequences are clearly the same as if the twisted parts
$d'_2$ resp.~$d''_2$ of the full differentials vanish. We may therefore
assume $d'_2 = d''_2 = 0$ from the beginning, in which case is the
claim obvious.
\end{proof}

We will also need to know that the functor $\Cfun(-)$ preserves
filtered quasi-isomorphisms. The proof of the following statement
is a harmless modification of~\cite[Proposition~4.4.4]{hinich01:_dg}.

\begin{lemma}
\label{sec:dual-hinich-corr}
Assume that
$\uu$ and $\vv$ are filtered complete dglas, with complete filtrations
\begin{equation}
\label{snezi}
\uu = F_1\uu \supset F_2\uu \supset  F_3\uu \supset
\cdots \ \mbox { resp. } \ \vv = F_1\vv \supset F_2\vv \supset F_3\vv
\supset \cdots  .
\end{equation}
Let $\phi : \uu \to \vv$ be
a morphism, compatible with the filtrations, such that the induced map
$\phi_n : \uu/F_n\uu \to  \vv/F_n\vv$ is a quasi-isomorphism for each
$n \geq 1$. Then $\Cfun(\phi) : \Cfun(\vv) \to \Cfun(\uu)$
is a~quasi-isomorphism, too.
\end{lemma}

\begin{rem}
The map $\phi_n : \uu/F_n\uu \to  \vv/F_n\vv$ is a chain map
in the category of linearly compact spaces. By saying it is
a quasi-isomorphism we mean that it is a quasi-isomorphism in the
underlying category of dg vector spaces. It is easy to show that
$\phi_n$ is a quasi-isomorphism if and only if its dual $\phi_n^*$
is a quasi-isomorphism in the usual sense, compare
Remark~\ref{sec:dual-hinich-corr-4}.
\end{rem}

\begin{proof}[Proof of Lemma~\ref{keylem}]
Our proof will be similar to that of \cite[pp.~224-225]{hinich01:_dg} or
\cite[pp.~42-43]{Lefevre-hasegawa02}. The kernel
$B$ of the map $f:A\to \Cfun(\g)$ is a non-unital cdga.
Since $\Cfun(\g)$ is free as a~non-differential algebra,
choosing an algebra splitting of $f$, we obtain an
isomorphism
\[
A\cong B\oplus\Cfun(\g)
\]
 of graded algebras,
in which $B$ is closed under the differential.  Therefore
$\hat\Lfun(A)$ is isomorphic, as a complete graded Lie algebra, to
$\hat\Lfun(B_e)*\hat\Lfun\Cfun(\g)$, the (completed) free product of
$\hat\Lfun(B_e)$ and $\hat\Lfun\Cfun(\g)$. The differential in $\hat\Lfun(A)$
consists of three parts: the differential in $\hat\Lfun(B_e)$, the
differential in $\hat\Lfun\Cfun(\g)$ and the `twisted' part
mapping the generators of $\hat\Lfun(B_e)$ into $\hat\Lfun\Cfun(\g)$.

We clearly have the following isomorphisms of complete graded
Lie algebras (again, disregarding the differentials):
\begin{align*}
\a\cong \hat\Lfun(A)*_{\hat\Lfun\Cfun(\g)}\g\cong
\big(\hat\Lfun(B_e)*\hat\Lfun\Cfun(\g)\big)*_{\hat\Lfun\Cfun(\g)}\g
\cong\hat\Lfun(B_e)*\g.
\end{align*}
The differential in $\hat\Lfun(B_e)*\g$ is, under this isomorphism, also the
sum of three parts: the differential in $\hat\Lfun(B_e)$, the differential
in $\g$ and the twisted part, which maps the generators of
$\hat\Lfun(B_e)$ into $\g$. It is
clear that the map $j:\hat\Lfun(A)\to\a$, under the above identifications,
equals
\begin{equation}
\label{zase_jsem_podlehl}
\id *i_\g:\hat\Lfun(B_e)*\hat\Lfun\Cfun(\g)\to\hat\Lfun(B_e)*\g,
\end{equation}
but the differentials differ by the twistings from those of the
free products.

Before we proceed, we need some notation. If
$\uu$ and $\vv$ are filtered complete dglas as in~(\ref{snezi}), then
their completed free product $\uu * \vv$ possesses an induced double filtration
with $F_{p,q}(\uu * \vv)$ the closure of the subspace generated by all
iterated brackets of $u_i \in F_{p_i}\uu$ and $v_j \in F_{q_j}\vv$, $1
\leq i \leq a$, $1 \leq j \leq b$, such that
\[
p_1+ \cdots+ p_a = p \ \mbox { and } \ q_1+ \cdots + p_b = q.
\]
One also has the associated total filtration with $F_n(\uu * \vv) :=
\bigcup_{p+q = n} F_{p,q}(\uu * \vv)$.

\def\Gr{{\rm Gr}}
For graded dglas one defines the bigrading  on their completed free product in
the obvious analogous manner. We then have the following formula
for the associated bigraded dgla:
\begin{equation}
\label{Vcera_na_Karafiatovi}
\begin{aligned}
\Gr_{p,q}(\uu * \vv) &:= \frac{F_{p,q}(\uu * \vv)}{F_{p+1,q}(\uu * \vv)
  \cup F_{p,q+1}(\uu * \vv)} \cong \big(\Gr(\uu) * \Gr(\vv)\big)_{p,q}.
\end{aligned}
\end{equation}

Let us return to our proof. The lower central series $\g_1 \supset
\g_2 \supset\g_3 \supset
\cdots$  of $\g$ induces, in the standard
manner, a filtration of $\hat\Lfun\Cfun(\g)$ while the complete free Lie algebra
$\hat\Lfun(B_e)$ is filtered by its lower central series. As  explained
above, one has the induced double filtrations of the completed free
products $\hat\Lfun(B_e)*\hat\Lfun\Cfun(\g)$ and $\hat\Lfun(B_e)*\g$.
The associated total filtrations are stable under
the differentials and the map~(\ref{zase_jsem_podlehl})
factorizes, for each $n$, into a map
\begin{equation}
\label{snih}
\phi = \phi_n:
\frac{\hat\Lfun(B_e)*\hat\Lfun\Cfun(\g)}{F_n(\hat\Lfun(B_e)*\hat\Lfun\Cfun(\g))}
\longrightarrow
\frac{\hat\Lfun(B_e)*\g}{F_n(\hat\Lfun(B_e)*\g)}
\end{equation}
of dglas. By Lemma~\ref{sec:dual-hinich-corr}, it suffices to prove
that $\phi$ is a quasi-isomorphism for each $n$.

Fix $n$ and consider the double filtrations of the quotients
in~(\ref{snih}) induced by the double filtrations
$F_{p,q}\big(\hat\Lfun(B_e)*\hat\Lfun\Cfun(\g)\big)$
resp.~$F_{p,q}\big(\hat\Lfun(B_e)*\g\big)$.
It is straightforward though technically involved to prove that these
filtrations and the map $\phi$ satisfy the assumptions of
Lemma~\ref{sublemma}. Taking the quotients was needed to make these double
filtrations finite.

By Lemma~\ref{sublemma}, it is enough to prove that the bigraded components
$\phi_{pq}$ of the map $\phi$ are
quasi-isomorphisms. Using~(\ref{Vcera_na_Karafiatovi}), one easily
identifies $\phi_{pq}$ with the map
\begin{equation}
\label{za_chvili_k_Pakousum}
(\id *i_{\Gr(\g)})_{p,q}:
\big(\Gr(\hat\Lfun(B_e))*\hat\Lfun\Cfun(\Gr(\g))\big)_{p,q}
\longrightarrow
\big(\Gr(\hat\Lfun(B_e))*\Gr(\g)\big)_{p,q}
\end{equation}
if $p+q < n$ while $\phi_{pq} = 0$ otherwise.

Since the gradings of both $\Gr(\hat\Lfun(B_e))$,
$\hat\Lfun\Cfun(\Gr(\g))$ and $\Gr(\g)$ are {\em
positive\/}, the $(p,q)$-components of the
free products in~(\ref{za_chvili_k_Pakousum}) are
spanned by iterated brackets of length $\leq p+q$. We may thus
in~(\ref{za_chvili_k_Pakousum}) disregard the topologies and consider
{\em uncompleted\/} free products.

\def\Lhin{{\mathscr L}}\def\Chin{{\mathscr C}}
In the proof of Theorem~\ref{free} we established that the uncompleted
free product is an exact functor. To finish the proof, it is therefore
enough to establish that the canonical map
$\hat\Lfun\Cfun\big(\Gr(\g))_p
\to \Gr(\g)_p$ is a quasi-isomorphism for each $p$. It is immediate to
see that $\hat\Lfun\Cfun(\Gr(\g))_p$ coincides with the degree $p$-part
$\Lhin\Chin(\Gr(\g))_p$ of the composition of the adjoint functors $\Chin$ and
$\Lhin$ (see~\cite[2.2]{hinich01:_dg} for their definitions)
applied to $\Gr(\g)$
considered as a discrete dgla. The proof is finished by
of~\cite[Proposition~3.3.2]{hinich01:_dg} by which the canonical
map $\Lhin\Chin(\Gr(\g)) \to \Gr(\g)$ is a quasi-isomorphism.
\end{proof}

Let us prove another auxiliary statement.

\begin{lemma}
\label{sec:dual-hinich-corr-3}
The functor $\hat\Lfun :\algs_+  \to \L$ preserves weak
equivalences. Moreover, it maps fibrations to cofibrations and
cofibrations to fibrations.
\end{lemma}

\begin{proof}
Let $f : A \to B$ be a weak equivalence, i.e.~a quasi-isomorphism, in
$\algs_+$. By definition, $\Lhat (f) : \Lhat(B) \to \Lhat(A)$ is a
weak equivalence if the induced map   $\Cfun\Lhat (f) : \Cfun\Lhat(A)
\to \Cfun\Lhat(B)$ is a quasi-isomorphism. This follows from the
diagram
\[
\xymatrix{
{\Cfun\Lhat(A)} \ar_{i_A}[d]\ar^{\Cfun\Lhat(f)}[r]&{\Cfun\Lhat(B)}\ar^{i_B}[d]
\\
A \ar^f[r]&B}
\]
in which the vertical maps are quasi-isomorphisms by
Proposition~\ref{adj}(1).

Assume that $f$ is a fibration in $\algs_+$, i.e.~an
epimorphism. To prove that $\Lhat(f)$ is a cofibration, we must
find a dotted arrow in each diagram
\[
\xymatrix{{\Lhat(B)}\ar_{\Lhat(f)}[d]\ar[r]&\g\ar^u[d]
\\
{\Lhat(A)}\ar@{-->}[ur]\ar[r]&\h}
\]
in which $u : \g \to \h$ is an acyclic fibration in $\L$. Using the
adjunction between $\Lhat$ and $\Cfun$, we see that we may equivalently
seek for a dotted arrow in the diagram
\[
\xymatrix{{\Cfun(\h)}\ar_{\Cfun(u)}[d]\ar[r]&A\ar^f[d]
\\
{\Cfun(\g)}\ar@{-->}[ur]\ar[r]&B}
\]
in which $f$ is a fibration in $\algs_+$ by assumption.  Since $\Cfun(u)$ is a
quasi-isomorphism by the definition of weak equivalences in $\L$, all
we need to show is that $\Cfun(u)$ is a cofibration in $\algs_+$.

Denote by $\g = \g_1 \subset \g_2 \subset \g_2 \subset \cdots$ the
lower central series of $\g$.
In Remark~\ref{sec:appl-struct-simpl-2} we verified that $\g = \lim
\g/\g_n$. Since $K := {\rm Ker}(f) \subset \g$ is a closed subspace, by
standard properties of filtered limits we verify that the limit of the
tower
\begin{equation}
\label{Za_chvili_sraz_s_Jarkou_a_Nausnickou}
\h \cong {\g}/{(\g_1 \cap K)} \stackrel{\pi_1}
\twoheadleftarrow  {\g}/{(\g_2 \cap K)} \stackrel{\pi_2}\twoheadleftarrow
{\g}/{(\g_3 \cap K)}\stackrel{\pi_3} \twoheadleftarrow  \cdots
\end{equation}
equals $\g$, therefore the colimit of the diagram
\[
\Cfun(\h) \hookrightarrow \Cfun( {\g}/{\g_2 \cap K}) \hookrightarrow \Cfun( {\g}/{\g_3
  \cap K})
\hookrightarrow \Cfun( {\g}/{\g_4 \cap K})\hookrightarrow  \cdots
\]
equals $\Cfun(\g)$. It would therefore suffice to prove that the maps
$\Cfun(\pi_n)$, where $\pi_n$ are as
in~(\ref{Za_chvili_sraz_s_Jarkou_a_Nausnickou}),
are cofibrations in $\algs_+$ for each $n \geq 1$.

To this end, observe that the kernel of $\pi_n$ is an abelian dgla
$(\g_n \cap K)/(\g_{n+1} \cap K)$. With this knowledge, it is easy to
see as in the proof of the lemma in~\cite[\S5.2.2]{hinich01:_dg} that
$\Cfun(\pi_n)$ is a~standard cofibration obtained by adding free
generators to $\Cfun( {\g}/{\g_{n}\cap K})$.

To prove the last part of the lemma, note the standard fact that
cofibrations in $\algs_+$ are monomorphisms, while $\Lhat(-)$ clearly
converts monomorphisms to epimorphisms, i.e.~fibrations in $\L$.
\end{proof}

In the proof of Lemma~\ref{sec:dual-hinich-corr-3}
we established that the functor
$\Cfun(-)$ converts fibrations to cofibrations. As each $\g \in
\Lhat$ is fibrant, this in particular implies that the dglas
$\Cfun(\g)$ are cofibrant in~$\algs_+$.

\begin{theorem}
\label{closedlie}
The category $\L$ is a closed model category with fibrations,
cofibrations and weak equivalences as in
Definition~\ref{sec:dual-hinich-corr-2}. Moreover, it is Quillen
equivalent to the closed model category $\algs_+$ via the adjunctions
$\hat\Lfun$ and $\Cfun$.
\end{theorem}

\begin{proof}
The arguments are precisely dual to those of Hinich
\cite[pp.~223--224]{hinich01:_dg}. 
The category $\hat{\mathscr L}$ admits arbitrary limits and
colimits. Indeed, the limits are created in the category of linearly
compact dg vector spaces. The coequalizer of two maps $f,g: \mathfrak
g\to \mathfrak h$ is the quotient of $\mathfrak h$ by the closed ideal
in $\mathfrak h$ generated by the elements $f(a)-g(a)$ where
$a\in\mathfrak g$. Finally, the coproduct of a family of complete
dglas is constructed by taking their free product, and then
completing.
This proves axiom CM1; the axioms CM2 and CM3 are
obvious, and one half of CM4 holds by definition.

Let us prove the factorization axiom CM5.  Given a map $f:\g\to\h$ in
$\L$, consider the corresponding map of cdgas
$\Cfun(f):\Cfun(\h)\to\Cfun(\g)$. Suppose we factorize  $\Cfun(f)$ as
\begin{equation}
\label{eq:31}
\Cfun(\h)\stackrel i\to A\stackrel p\to\Cfun(\g)
\end{equation}
where $p$ is a
surjective map, i.e.~a fibration, and $i$ a cofibration in
$\algs_+$.  Then one has an induced factorization of $f$:
\begin{equation}
\label{Dnes_docela_pekne}
\g\stackrel{\tilde{i}}\longrightarrow
\hat\Lfun(A)*_{\hat\Lfun\Cfun(\g)}\g\stackrel{\tilde{p}}\longrightarrow\h
\end{equation}
where $\tilde{p}$ is easily seen to be an epimorphism, i.e.~a
fibration in $\L$.  By construction, $\tilde{i}$ is obtained by a
cobase-change from $\Lhat(p)$ which is a cofibration by
Lemma~\ref{sec:dual-hinich-corr-3}. Thus $\tilde{i}$ is a~cofibration as
well.

Since $\algs_+$ is a closed model category,
factorization~(\ref{eq:31}) exists and can be chosen such that $p$ is
a quasi-isomorphism. Then the corresponding $\tilde{i}$
in~(\ref{Dnes_docela_pekne}) is acyclic; this follows (just as
in~\cite{hinich01:_dg}) from Lemma \ref{keylem} and Proposition
\ref{adj}(1). This proves the first factorization of~CM5.

The proof of the second factorization property of CM5 is similar, we
only choose this time $i$ in~(\ref{eq:31}) to be a quasi-isomorphism.
The map $\tilde{p}$ is then acyclic by Lemma~\ref{keylem}.

It remains to prove the second half of CM4, i.e.~that any acyclic
cofibration in $\L$ has the LLP (left lifting property) with respect
to all fibrations. Let $f:\g\to \h$ be such a cofibration. Factorize
it as in~(\ref{Dnes_docela_pekne}): $f=\tilde p \circ \tilde i$ where
$\tilde p$ is an acyclic fibration and $\tilde i$ is an acyclic
cofibration obtained by a cobase-change from a map $\hat\Lfun(p)$,
where $p$ is an acyclic fibration in $\algs_+$.

By Lemma~\ref{sec:dual-hinich-corr-3}, $\hat\Lfun(p)$ is an acyclic
cofibration in $\L$, so it has the LLP with respect to all
fibrations in $\L$, thus $\tilde i$ has the same property. Since
$\tilde p$
has the LLP with respect to $f$ it follows that $f$ is a retract of
$\tilde i$ and so, it has the LLP with respect to all fibrations in $\L$ as
required. This finishes our proof that $\L$ is a closed model
category.

The statement
that the adjoint functors $\Lhat$ and $\Cfun$ form a Quillen pair,
i.e.~satisfy assumptions of~\cite[Theorem~9.7]{dwyer-spalinski},
follows at once from Proposition~\ref{adj} and
Lemma~\ref{sec:dual-hinich-corr-3}.
\end{proof}

An interesting feature of the closed model structure of
Definition~\ref{sec:dual-hinich-corr-2} is the existence and uniqueness
of minimal models.

\begin{definition}
\label{sec:dual-hinich-corr-5}
A dgla $\Min \in \L$ of the form $\Min = \big(\Lfreecompl(M),\partial\big)$
is {\em minimal\/} if $\partial$ induces the trivial differential on $M$. A {\em
  minimal model\/} of a dgla $\g \in \L$ is a minimal dgla $\Min$
together with a weak equivalence $\Min \to \g$.
\end{definition}

\begin{theorem}
\label{sec:dual-hinich-corr-6}
Each $\g \in \L$ has a minimal model unique up to an isomorphism.
\end{theorem}

\begin{proof}
For $\g \in \L$ denote by $H$ the cohomology of $\Cfun(\g)$ and choose
a homotopy equivalence $f: \Cfun(\g) \to H$ in the category of dg
vector spaces; here $H$ is considered with the trivial differential.
By~\cite[Move~M1, p.~133]{markl:ha}, there exists a $C_\infty$-algebra
$C$ with the underlying dg-vector space $H$, and a $C_\infty$-morphism
\hbox{$F: \Cfun(\g) \to C$} extending $f$. Let $\Min =
\big(\Lfreecompl(\Sigma^{-1}H^*),\partial\big)$ be the complete dgla
corresponding to the $C_\infty$-algebra $C$. It is minimal and the
$C_\infty$-morphism $F$ translates into a weak equivalence $\varphi :
\Min \to \Lhat\Cfun(\g)$.  The canonical map $i_\g: \Lhat\Cfun(\g) \to
\g$ is, by Proposition~\ref{adj}(2), a weak equivalence, too. The
composition
\[
\rho:
\Min \stackrel \varphi\longrightarrow \Lhat\Cfun(\g)
\stackrel{i_\g}\longrightarrow  \g
\]
is then the desired minimal model of $\g$.

Let us prove uniqueness. Suppose that $\rho' : \Min' \to \g$ is
another minimal model. Since $\Cfun(\g)$ is fibrant as every object of
$\algs_+$, $\Lhat\Cfun(\g)$ is cofibrant by
Proposition~\ref{sec:dual-hinich-corr-3}.  Although $\rho'$ need not
be a fibration, it is still a map between cofibrant objects.  By
e.g.~\cite[Lemma~3.5]{vogt03:_cofib_e}, there exists the dotted arrow
$\alpha$ in the diagram
\[
\xymatrix{\Min  \ar[r]^\varphi &
\Lhat\Cfun(\g) \ar@{-->}[r]^\alpha \ar[rd]_{i_\g}
& \Min' \ar[d]^{\rho'}
\\
 && \g
}
\]
making the right triangle homotopy commutative. The composition
$\alpha \varphi$ is a weak equivalence of minimal dglas, so it induces
an isomorphism of the generators; it is therefore itself an
isomorphism.
\end{proof}

\subsection{Non-unital monoidal structures on $\L$ and $\algs_+$}
In Subsection~\ref{Pozitri_do_Cambridge} we constructed a non-unital
monoidal structure on the category $\L$ of complete dglas.  An
analogous, but simpler, structure on the category $\algs_+$ is given
by the product $A\times B$ of augmented cdgas $A, B \in \algs_+$
equipped with the augmentation through the projection onto the {\em
second\/} factor. The operation $A,B \mapsto A\times B$ is coherently
associative, but not commutative and not unital.  The two structures
correspond to each other under the adjoint functors $\hat\Lfun$ and
$\Cfun$:

\begin{proposition}
\label{monoidal}\
\begin{enumerate}
\item
For two augmented cdgas $A$ and $B$ there is a natural isomorphism in $\L$:
\[
\hat\Lfun(A\times B)\cong\hat\Lfun(A)\scup\hat\Lfun(B).
\]
\item
For two complete dglas $\g$ and $\h$ there is a natural quasi-isomorphism in $\algs_+$:
\[
\Cfun(\g\scup\h)\simeq\Cfun(\g)\times\Cfun(\h).
\]
\end{enumerate}

\begin{proof}
Starting from part (1), take $B$ to be the trivial augmented cdga:
$B=\mathbb k$. Disregarding the differentials, we have the following
isomorphisms of complete graded algebras:
\[
\hat\Lfun(A\times\mathbb k)
\cong \hat{\mathbb{L}}(\Sigma^{-1}A^*)
\cong
\hat{\mathbb{L}}(\Sigma^{-1}A^*_+)*\hat{\mathbb{L}}(\Sigma^{-1}\mathbb
k)\cong\hat\Lfun(A)*\hat\Lfun(\mathbb k)
\cong\hat\Lfun(A)*\mathfrak s.
\]
Direct inspection shows that, under the above isomorphism,
the differential in
$\hat\Lfun(A\times\mathbb k)$ corresponds to the twisted differential
$d^x$ in $\hat\Lfun(A)*\mathfrak s\cong
\hat\Lfun(A) \sqcup 0$, so the dglas $\hat\Lfun(A\times\mathbb k)$ and
$\hat\Lfun(A) \sqcup 0$ are isomorphic.
For a general cdga $B$ we have:
\[
\hat\Lfun(A\times B)
\cong \hat\Lfun\big((A\times{\mathbb k})\times_{\mathbb k} B\big)
\cong \hat\Lfun(A\times \mathbb k)*\hat\Lfun(B)
\cong \big(\hat\Lfun(A)\scup 0\big)*\hat\Lfun(B)
\cong \hat\Lfun(A)\scup \hat\Lfun(B).
\]
Here we used the fact that $\hat\Lfun$ transforms categorical products
in $\algs_+$ into categorical coproducts in $\L$, which follows
formally from its adjointness property.

Part (2) was established in the proof of Theorem~\ref{theoremF} on
page~\pageref{Jarusku_boli_hlava}.
\end{proof}
\end{proposition}

\section{Disconnected spaces and dg Lie algebras}
\label{disc-spaces}

In this section we prove Theorem~D.
The ground field  will always
be the field $\mathbb Q$ of rational numbers.

Let us define a functor $\mathcal Q:\SSet_+\mapsto\L$ as the
composition $\mathcal Q(S):=\Lhat \Omega(S)$, where $\Omega(S)$ is the polynomial
Sullivan-de~Rham algebra of a pointed simplicial set $S\in
\SSet_+$~\cite[\S2]{bousfield-gugenheim},
with the augmentation $A(S)
\to {\mathbb Q}$ induced by the inclusion of the base point of $S$.  The
simplicial MC space functor $\MC_\bullet:\L\mapsto \SSet_+$
was recalled in Definition~\ref{MCdef}, the base point of $\MC_\bullet(\g)$ is the
$0$-simplex corresponding to the trivial MC-element $0 \in \g$.
We then have the following result.

\begin{proposition}
The functors $\MC_\bullet:\L\mapsto \SSet_+$ and $\mathcal
Q:\SSet_+\mapsto \L$ form an adjoint pair, i.e.~there is a
natural isomorphism
\[
\SSet_+\big(S, \MC_\bullet(\g)\big) \cong \L\big(\mathcal Q(S),\g\big),
\]
for each pointed simplicial set  $S$ and a complete dgla $\g$.
\end{proposition}

\begin{proof}
Note that, by definition, there is an isomorphism of simplicial sets
$\MC(\g)\cong F\big(\Cfun(\g)\big)$, where $F$ is the Bousfield-Kan
functor associating to a cdga $B$ the simplicial set $\algs(B,\mathbb
Q)_\bullet$ recalled in~(\ref{zitra_musim_do_Prahy}).

Note also that if $B$ is an augmented cdga then $F(B)$ is naturally a
pointed simplicial set.  The proposition now follows from the fact
that the functors $\MC$ and $\mathcal Q$ are the compositions of the
upper resp.~lower functors in the diagram
\begin{equation}
\label{Predevcirem_jsem_slavil}
\raisebox{-.7em}{}
\L \adjointext \Cfun\Lhat \algs_+\adjointext {F}\Omega \SSet_+
\end{equation}
in which the functors $\Lhat$ and $\Cfun$ are adjoint by
Proposition~\ref{adj_bis}, while the functors $F$ and $\Omega$ form an
adjoint pair by~\cite[\S8.9]{bousfield-gugenheim} (the functor
$\Omega$ of Sullivan-De~Rham forms was denoted by $A$
in~\cite{bousfield-gugenheim}).
\end{proof}

\begin{definition}
The subcategory in the homotopy category of $\L$ formed by the disjoint
products of finitely many non-negatively graded complete dglas whose
homology are finite-dimensional in each degree will be denoted by $\hoLdc$.
\end{definition}

\begin{rem}
\label{sec:dual-hinich-corr-4}
Note that $\g = (\g,\partial) \in \L$ is a chain complex in the
category of linearly compact spaces. Since $\partial$ is continuous,
$\Ker(\partial)$, and also $\im(\partial)$ as a continuous image of a~linearly compact space, are closed subspaces of $\g$. So $H(\g) =
\Ker(\partial)/\im(\partial)$ is a linearly compact graded space. Clearly
$H(\g)^* \cong H(\g^*)$, where $ H(\g^*)$ is the ordinary cohomology
of the discrete dual $(\g^*,\partial^*)$.
\end{rem}

\begin{rem}
The functor
$\g,\h \mapsto \g\scup\h$ lifts to the homotopy category of $\L$.
Indeed, there is a quasi-isomorphism
$\Cfun(\g)\times \Cfun(\h)\simeq \Cfun(\g\scup\h)$, by
Proposition~\ref{monoidal}(2). If $\g$
and $\g^\prime$ are weakly equivalent complete dglas then, by
definition, the cdgas $\Cfun(\g)$ and $\Cfun(\g^\prime)$ are
quasi-isomorphic, thus the cdgas $\Cfun(\g)\times \Cfun(\h)$ and
$\Cfun(\g^\prime)\times \Cfun(\h)$ are likewise quasi-isomorphic. Therefore
the complete dglas $\g\scup\h$ and $\g^\prime\scup\h$ are
weakly equivalent.
\end{rem}

It follows from Theorem~\ref{closedlie}
resp.~\cite[Lemma~8.5]{bousfield-gugenheim} that the composite
functors $\MC_\bullet$ and $\mathcal Q$ induce an adjoint pair of
functors between the homotopy categories of $\L$ and $\SSet_+$. The
following result implies Theorem~D.

\begin{theorem}
\label{equidgla}
The functors $\MC_\bullet$ and $\mathcal Q$ determine mutually inverse
equivalences between the categories $\hoLdc$ and $\honilpsimpdisconplus$.
\end{theorem}

\begin{proof}
We use the pointed (resp.~augmented) versions of the results of
Part~\ref{part:point-view-dg} formulated in
Section~\ref{sec:augm-dg-comm}.  As the first step, observe that the
category $\hoLdc$ is equivalent to the auxiliary `extended' category
$\exthoLdc$ whose objects are dglas weakly equivalent to objects of
$\hoLdc$.  We claim that the left adjunction
of~(\ref{Predevcirem_jsem_slavil}) restricts to the adjunction
\begin{equation}
\label{Cajik_od_Jarky}
\raisebox{-.9em}{}
\exthoLdc \adjointext \Cfun\Lhat \horathincatdcplus,
\end{equation}
where $\horathincatdcplus$ is the full subcategory of ${\rm ho}\algs_+$
consisting of augmented homologically disconnected algebras of finite
type. To this end, we must check that
\begin{equation}
\label{Tucinek}
\Cfun(\g) \in
\horathincatdcplus\ \mbox { and }\ \Lhat(A) \in \exthoLdc
\end{equation}
whenever $\g \in \hoLdc$
and $A \in \horathincatdcplus$.

Let us look at the first condition. Assume that $\g = \g_1 \sqcup
\cdots \sqcup \g_k$, where $\g_i \in \hoLdc$ are non-negatively
graded. Since, by Proposition~\ref{monoidal}(2), $\Cfun(\g_1 \sqcup
\cdots \sqcup \g_k)$ is quasi-isomorphic to $\Cfun(\g_1) \times \cdots
\times \Cfun(\g_k)$, it is enough to show that $\Cfun(\g) \in
\horathincatdcplus$ whenever $\g \in \hoLdc$ is non-negatively graded.
By Remark~\ref{sec:appl-struct-simpl}, $\Cfun(\g)$ is cofibrant non-negatively graded
connected cdga whose only augmentation ideal $I$ is the ideal
generated by $\Sigma\g^*$, so $I/I^2 \cong  \Sigma \g^*$.
Therefore $\Cfun(\g)$ is connected of finite type since $H(\g^*)$ is such by
assumption.  This proves the first condition of~(\ref{Tucinek}).

To prove the second condition, observe that, by definition, the
algebra $A$ is weakly equivalent (i.e.~related by a ziz-zag of
quasi-isomorphisms) to a finite product of non-negatively graded
connected algebras of finite type, $A \simeq A_1 \times \cdots \times
A_k$, with the augmentation given by the projection to the last
factor.  Since, by Proposition~\ref{monoidal}(1),
\[
\Lhat(A_1 \times
\cdots \times A_k) \cong \Lhat(A_1)\sqcup \cdots \sqcup \Lhat(A_k),
\]
it is enough to verify that $\Lhat(A) \in \exthoLdc$ whenever $A$ is
connected, cofibrant non-negatively graded cdga of finite type. Since in
this case $\Lhat(A)$ is non-negatively graded, it suffices to
prove that the homology of $\Lhat(A)$ is finite-dimensional in each
degree. This, however, immediately follows from Lemma~\ref{charact}.

As the functors $\Lhat$ and $\Cfun$ form a Quillen pair by
Theorem~\ref{closedlie},
their restrictions in~(\ref{Cajik_od_Jarky}) are mutually inverse
equivalences of categories. The rest of the proof immediately follows
from Theorem~C$_+$ of Section~\ref{sec:augm-dg-comm} by which the category
$\horathincatdcplus$ is equivalent to the category $\honilpsimpdisconplus$.
\end{proof}

\begin{rem}
The reader may wonder why Theorem~\ref{equidgla} relates {\em
augmented\/} algebras to dg-Lie algebras with no additional
structure. The answer is that the dg Lie-analogue of an augmentation is
a choice of an MC-element. From this perspective, each dgla is
canonically augmented by the trivial MC-element $0$. Our definition
of the disjoint product $\g \sqcup \h$ is such that
its augmentation is induced by the augmentation of the last factor.
\end{rem}

\begin{rem}
Note that Theorem \ref{equidgla} allows one to say something new even
for connected spaces (or simplicial sets). Indeed, the category of
\emph{unpointed} connected spaces is a subcategory of
\emph{disconnected} pointed spaces. Namely, this subcategory consists
of spaces, consisting of two connected components, one of which is the
basepoint. We see, therefore, that the homotopy category of rational
connected unpointed spaces is equivalent to a certain subcategory of
$\hoLdc$, whose objects are of the form $\g\scup 0$ for a
non-negatively graded complete dgla $\g$. This subcategory is not
full; in fact it is easy to see that there is a natural bijection
\[
\big[\g\scup 0,\h\scup 0\big]
\cong \big[\MC_\bullet(\g),\MC_\bullet(\h)\big]\cup \{\ast\}
\]
where $*$ denotes an isolated basepoint corresponding to the zero map
$\g\scup 0\to\h\scup 0$.
\end{rem}

\appendix

\section{Cohomology of free products of dg Lie algebras}

The purpose of this appendix is to express the Chevalley-Eilenberg
cohomology of free products of dglas in terms of the
Chevalley-Eilenberg cohomology of the factors.
The main results are Theorem~\ref{free} for the non-complete case  and
Theorem~\ref{freeprod} for the compete one.
These results
seem completely standard but we have not found them in the
literature. There is an analogous result for group cohomology, but it
relies on the construction of the classifying space of a group; such a
construction has no analogue for an arbitrary (i.e.~not necessarily
nilpotent) dgla, so we needed to develop a certain algebraic machinery
instead. First, introduce some notation.

Recall that $\Lie$ is the category of (discrete) dglas. We denote
by $\A$ the category of complete cdgas; its objects are inverse
limits of finite dimensional nilpotent \emph{non-unital} cdgas.

\begin{defi}
Let $\Cf:\Lie \mapsto \A$ be the functor associating to a discrete dgla
$\g\in\Lie$ the complete cdga $\Cf(\g)$ whose underlying non-unital cdga
is $\hat{S}_+\Sigma\g^*$, the completed non-unital
symmetric algebra on $\Sigma \g^*$. The differential $d$ in
$\Cf(\g)$ is defined as $d=d_I+d_{\it II}$, where $d_I$ is induced by the
internal differential in $\g$ and $d_{\it \it II}$ is determined by its
restriction onto $\Sigma\g^*$, which is, in turn, induced by the
bracket map $\g\otimes \g\to \g$.
\end{defi}

\begin{rem}
\label{spectralseq}
The construction $\Cf(\g)$ is a double complex
with the horizontal differential $d_{\it II}$ and the vertical differential
$d_I$. As such, it has two spectral sequences associated with it; one
converging to the direct product totalization, the other
to the direct sum totalization. Since $\Cf(\g)$ is
constructed using direct products, only the first spectral sequence is
relevant. We denote this spectral sequence by $E^\prime(\g)$;
we have $E^\prime_2(\g)=H\big(\Cf(H(\g)\big)$.  Using this
spectral sequence, we easily prove that the functor $\Cf :\Lie \mapsto \A$ preserves
quasi-isomorphisms.
\end{rem}
The proof of Theorem~\ref{free} below will use the following lemma.

\begin{lemma}
\label{sec:append-cohom-free-1}
Let $f : \uu \to \vv$ be a morphism of dglas such that the
induced morphism \hbox{$\U(f)  : \U(\uu) \to \U(\vv)$} of their universal enveloping
algebras is a quasi-isomorphism. Then
$f$ is a quasi-isomorphism too.
\end{lemma}

\begin{proof}
By \cite[Theorem~4.5]{quillen:Ann.ofMath.69}, the functor $\U$
from dglas to cocommutative connected dg Hopf-algebras admits a
quasi-inverse, denoted by $\Prim$, which associates to a
dg Hopf algebra its dgla of primitive elements. Consider the following
commutative diagram of canonical maps:
\begin{equation}
\label{eq:8}
\raisebox{-.9cm}{}
{
\unitlength=1pt
\begin{picture}(100.00,31)(0.00,20.00)
\thicklines
\put(50.00,2){\makebox(0.00,0.00)[b]{\scriptsize $\Prim H\U(f)$}}
\put(50.00,42.00){\makebox(0.00,0.00)[b]{\scriptsize $H\Prim \U(f)$}}
\put(100.00,0.00){\makebox(0.00,0.00){\hphantom.$\Prim H\U (\vv)$.}}
\put(0.00,0.00){\makebox(0.00,0.00){$\Prim H\U (\uu)$}}
\put(100.00,40.00){\makebox(0.00,0.00){$H\Prim\U (\vv)$}}
\put(0.00,40.00){\makebox(0.00,0.00){$H\Prim\U (\uu)$}}
\put(25,0.00){\vector(1,0){50.00}}
\put(100.00,30.00){\vector(0,-1){20.00}}
\put(0.00,30.00){\vector(0,-1){20.00}}
\put(25,40.00){\vector(1,0){50}}
\end{picture}}
\end{equation}

Since $\Prim \U(\uu) \cong \uu$ one has  $H\Prim \U(\uu) \cong H(\uu)$
and, since $\U H \cong H\U$ by \cite[Theorem~2.1]{quillen:Ann.ofMath.69}, one has $\Prim
H\U (\uu) \cong \Prim \U H (\uu)  \cong H(\uu)$; similarly for
$\vv$ in place of $\uu$. We conclude that
both the vertical arrows in~(\ref{eq:8}) are isomorphisms, while the
bottom map $\Prim H\U(f)$ is an isomorphism by assumption. Thus $H(f) =
H\Prim \U(f)$ must be an isomorphism too.
\end{proof}

\begin{theorem}
\label{free}
Let $\g$ and $\h$ be two discrete dglas. Then there is a
quasi-isomorphism $\hbox{$\Cf(\g*\h)$} \cong \Cf(\g)\times\Cf(\h)$
where $\g*\h$ is the free (non-completed) product of $\g$ and $\h$.
\end{theorem}

\begin{proof}
As the first step we prove that the functor $\g\mapsto \hbox{$\g*\h$}$
preserves quasi-isomorphisms. Our proof of this fact will
use the isomorphism $\hbox{$\U(\g*\h)$}\cong \U(\g)*\U(\h)$ for universal
enveloping algebras, where the symbol $*$ in the right hand
side stands for the free product of dgas.

Let $\g^\prime$ be a dgla quasi-isomorphic to $\g$. Since $\U$
preserves quasi-isomorphisms by \cite[Theorem~2.1]{quillen:Ann.ofMath.69}, it suffices
to prove that the functor of taking the free product on the
category of dgas preserves quasi-isomorphisms. Indeed, if it is so, then
we have a chain of isomorphisms and quasi-isomorphisms:
\[
\U(\g*\h)\cong \U(\g)*\U(\h)\simeq  \U(\g^\prime)*\U(\h)\cong
\U(\g^\prime*\h).
\]
Lemma~\ref{sec:append-cohom-free-1} then implies that $\g*\h$ is
quasi-isomorphic to $\g^\prime*\h$.

To show the exactness of the free product functor for associative
algebras, observe that the product $A*B$ of dgas $A$ and $B$
decomposes as a direct sum of tensor products of $A$ and~$B$:
\[
A*B\cong A \oplus B \oplus (A\otimes B) \oplus (B\otimes A)
\oplus (B\otimes A\otimes B) \oplus (A\otimes B\otimes A) \oplus\cdots
\]
so the statement follows from the exactness of the tensor product
of dgas over a field of characteristic zero.

Recall that the functor $\g\mapsto \Cf(\g)$  preserves
quasi-isomorphisms
by  Remark~\ref{spectralseq}.
Since $\Lie$ is a closed model category, we
conclude that it suffices to prove the statement of our theorem in the
case when both $\g$ and $\h$ are standard cofibrant dglas, i.e.~when they
are obtained from the trivial dgla by a sequence of cell attachments. In
that case $\g*\h$ is likewise a standard cofibrant~dgla.

Further, we claim that if $\a$ is a cofibrant dgla then $\Cf(\a)$ is
quasi-isomorphic to the space $\Der(\a,\ground)$ of derivations of
$\a$ with coefficients in the trivial one-dimensional $\a$-module
$\ground$ -- note that $\Der(\a,\ground)$ is isomorphic to
$(\a/[\a,\a])^*$, the dual space of indecomposables of $\a$. This is
a~standard fact which can be proved e.g.~by $\ground$-linear dualization
of \cite[Proposition~9.1.1]{hinich01:_dg}.
Finally, it is clear that
$\Der(\g*\h,\ground)\cong \Der(\g,\ground)\times\Der(\h,\ground)$.
The desired result is proved.
\end{proof}

Recall that we are really interested in \emph{complete} dglas whereas
the last result concerns non-complete ones. Somewhat surprisingly, it
also holds for complete dglas, as a \emph{consequence} of Theorem
\ref{free}. In the formulation of the complete case we have to take
the fiber product of the corresponding CE complexes $\Cfun(\g)$ and
$\Cfun(\h)$ since the latter are \emph{augmented} cdgas.  Let us prove
some preliminary statements. From now on we will write ${\Cfun}_+(\g)$
for the augmentation ideal of the cdga $\Cfun(\g)$.

\begin{lemma}
\label{nilp}
Let $\g$ and $\h$ be finite dimensional nilpotent graded Lie algebras,
viewed as \emph{complete} dglas.  Then there is a quasi-isomorphism
${\Cfun}_+(\g*\h)\simeq{\Cfun}_+(\g)\times\Cfun_+(\h)$.
\end{lemma}

\begin{proof}
The statement of the lemma is similar to the quasi-isomorphism of
Theorem \ref{free}, however the essential difference is that $\g$ and
$\h$ are regarded as \emph{complete} dglas (with vanishing
differentials) and so, their free product $\g*\h$ is likewise
completed. However, it turns out not to influence the result.  The
natural maps $\g \to \g * \h$ and $\h \to \g*\h$ induce a map
\begin{equation}
\label{chprod}
\Cfun_+(\g*\h)\to\Cfun_+(\g)\times\Cfun_+(\h);
\end{equation} we want to show that this map is a quasi-isomorphism.

We start the proof by reducing the statement to the case when $\g$ and $\h$
are abelian.
First observe that $\g$, being finite dimensional
nilpotent, has finite filtration, whose associated graded is
abelian (graded) Lie algebra which will be denoted by $\g^\natural$.

Indeed, consider the lower central series $\g = \g_1 \supset \g_2
\supset \cdots \supset \g_{s+1} = 0$ of $\g$. Although the associated graded of
this descending filtration need not be abelian, the shifted filtration
\begin{equation}
\label{malo_snehu}
\g = F_0\g \supset  F_1\g \supset \cdots \supset F_{s}\g = 0
\end{equation}
with $F_n \g := \g_{n+1}$, $0 \leq n \leq s$,
does have this property. Let $G^n\g^* := (F_{n+1}\g)^\perp = \g_{n+2}^\perp$ be the
annihilator of $F_{n+1}\g$ in $\g^*$.
Then
\begin{equation}
\label{na_Honzove_prednasce}
0 \subset G^0\g^* \subset  G^1\g^* \subset \cdots \subset G^{s}\g^* = \g^*
\end{equation}
is a finite ascending filtration of $\g^*$.
The subspaces
\[
G^n\Cfun_+(\g) := \bigoplus_{k \geq 1}\bigoplus_{n = n_1+\cdots+ n_k}
\Sigma G^{n_1}\g^* \cdots \Sigma G^{n_k}\g^* \subset \Cfun_+(\g)
\]
form an ascending filtration
\[
0 \subset G^0\Cfun_+(\g) \subset G^1\Cfun_+(\g) \subset
G^2\Cfun_+(\g) \subset \cdots \subset \Cfun_+(\g)
\]
which is exhaustive, Hausdorff and complete. The $E_1$-term of the
induced spectral sequence clearly equals
$\Cfun_+(\g^\natural)$ and this spectral
sequence strongly converges to $\Cfun_+(\g)$.

Let us turn our attention to the completed free product $\g *
\h$. One has a natural
epimorphism $\pi: \LL(\g,\h) \epi \g \CF \h$, where
$\LL(\g,\h)$ is the free graded Lie algebra generated by $\g \oplus
\h$, and $\g\CF\h$ denotes, only in this proof, the {\em uncompleted\/}
free product. Let $\LL^{\geq m}(\g,h)$ be the ideal
spanned by products of at least $m$ elements, and  $(\g \CF \h)^{\geq
  m} := \pi \big(\LL^{\geq m}(\g,h)\big)$. Since $\pi$ is an
epimorphism, $(\g \CF \h)^{\geq  m}$ is an ideal, and
\[
\g * \h = \lim_m \ (\g \CF \h) /(\g \CF \h)^{\geq
  m}.
\]

Filtration~(\ref{malo_snehu}) induces, in the standard manner, a
filtration $F_n \LL(\g,\h)$, $n \geq 0$, of $\LL(\g,\h)$. Denote
finally by $F_n(\g \CF \h): = \pi\big(F_n \LL(\g,\h)\big)$ the induced
filtration of the uncompleted free product and by $F_n(\g * \h)$ its
closure in $\g * \h$. It follows from the continuity of the
bracket in $\g * \h$ that
\[
\g*\h =
F_0 (\g *\h) \supset F_1 (\g *\h) \supset F_2 (\g *\h) \supset \cdots
\]
is a descending filtration by ideals. The formula $G^n(\g*\h)^* :=
\big(F_{n+1}(\g*\h)\big)^\perp$ defines an ascending filtration
\begin{equation}
\label{napadna}
0 \subset G^0(\g* \h)^* \subset G^1(\g* \h)^* \subset G^2(\g* \h)^* \subset
\cdots (\g * \h)^*.
\end{equation}

Recall that $(\g * \h)^*$, by definition, consists of {\em continuous\/}
linear functionals. By continuity, every such functional $\alpha$ factors
through the canonical epimorphism
\[
(\g * \h) \epi (\g \CF \h)/(\g \CF \h)^{\geq m}
\]
for $m >\!> 0$.
By the finiteness of~(\ref{malo_snehu}), the induced filtration of
$(\g \CF \h)/(\g \CF \h)^{\geq m}$ is finite, so $\alpha$ annihilates
$F_n (\g *\h)$ for $n >\!>0$.
This implies that the filtration~(\ref{napadna}) is exhaustive.

Now we proceed as in the case of
$\Cfun_+(\g)$. The filtration~(\ref{napadna}) induces an exhaustive,
Hausdorff and complete filtration of $\Cfun_+(\g * \h)$. The
induced spectral sequence strongly converges to $\Cfun_+(\g * \h)$ and
its $E_1$-term is $\Cfun_+(\g^\natural *
\h)$.

Since the canonical map~(\ref{chprod}) is compatible with the
filtrations, by the comparison theorem for spectral sequences, it is
enough to prove that~(\ref{chprod}) is a quasi-isomorphism
with $\g$ replaced by $\g^\natural$.
Repeating the same steps with $\h$ in place of $\g$, we prove that
the desired statement would follow if we can
prove that the map
$\Cfun_+(\g^\natural*\h^\natural)\to \Cfun_+(\g^\natural)\times
\Cfun_+(\h^\natural)$  is a quasi-isomorphism. In other words, we
reduced the statement of the lemma to the case when both $\g$ and $\h$
are {\em abelian\/}.

The dg space $\Cfun(\g*\h)=S\Sigma(\g*\h)^*$ consists of symmetric
tensors in $\Sigma(\g*\h)^*$; we can introduce a grading on
$\Cfun(\g*\h)$ as follows. Let $x\in\g*\h$. Then $x$ has weight $n+1$
if $x$ is a sum of Lie monomials of bracket length $n$; this grading
lifts to $\Cfun(\g*\h)$.

It is clear that the differential in $\Cfun(\g*\h)$ preserves the
weight grading and so the dg space $\Cfun(\g*\h)$ decomposes as an
infinite direct sum of subcomplexes consisting of elements of fixed
weight. It follows that the map (\ref{chprod}) is a quasi-isomorphism
if and only if it is a quasi-isomorphism for each weight
component. This, in turn, holds if and only if and only if a similar
statement holds when the infinite direct sum over all positive weights
is replaced by the corresponding infinite direct product. Observe
that the obtained completed dg space has the form
$\hat{\Cfun}(\g*\h)^*=\hat{S}\Sigma(\g*\h)^*$ where in the last
formula $\g*\h$ stands for the \emph{uncompleted} free product. It is,
therefore, nothing but the standard complex computing the
Chevalley-Eilenberg cohomology of the uncompleted Lie algebra $\g*\h$.
Thus, the statement is reduced to computing the usual
Chevalley-Eilenberg cohomology of $\g*\h$ and, therefore, follows from
Theorem \ref{free}.
\end{proof}

\begin{lemma}
\label{sec:append-cohom-free}
Let $\a$ be a complete dgla and let $\tilde{\a}$ be the complete dgla
with vanishing differential and the same graded Lie bracket as $\a$. Then
there exists a spectral sequence $E^{\prime\prime}(\a)$ converging strongly
to $\Cfun_+(\a)$ such that
$E_1^{\prime\prime}(\a)=H\big(\Cfun_+(\tilde{\a})\big)$.
\end{lemma}\label{sscom}

\begin{proof}
As graded vector spaces, $\Cfun_+(\a) = \bigoplus_{p \geq 1}
S^p(\Sigma^{-1}\a^*)$, where $S^p(-)$ denotes the subspace of the
symmetric algebra consisting of elements of homogeneity $p$. It is
obvious that
\begin{equation}
\label{dnes}
F_n\Cfun_+(\a) := \bigoplus_{q \geq n,\ p \geq 1} \big[S^p(\Sigma^{-1}\a^*)\big]^{p+q}
\end{equation}
forms a decreasing exhaustive filtration of the cochain complex
$\Cfun_+(\a)$. Its degree $m$ component equals
\[
F_n\Cfun_+(\a)^m :
= \bigoplus_{q \geq n,\ m > q} \big[S^{m-q}(\Sigma^{-1}\a^*)\big]^m.
\]
Clearly $F_n\Cfun_+(\a)^m = 0$ if $n > m$, the
filtration~(\ref{dnes}) is thus Hausdorff and complete. The induced
spectral sequence therefore converges strongly and obviously has the
properties stated in the lemma.
\end{proof}

\begin{theorem}
\label{freeprod}
Let $\g$ and $\h$ be two complete dglas. Then there is a quasi-isomorphism
$\hbox{$\Cfun_+(\g*\h)$} \simeq\Cfun_+(\g)\times \Cfun_+(\h)$.
\end{theorem}

\begin{proof}
Assume first that $\g$ and $\h$ are finite-dimensional nilpotent.
Arguing as in the proof of Lemma \ref{nilp} we construct a map
$\hbox{$\Cfun_+(\g*\h)$}\longrightarrow\Cfun_+(\g)\times\Cfun_+(\h)$;
the desired statement is equivalent to proving that this is a
quasi-isomorphism. Consider a map of spectral sequences
\[
E_1^{\prime\prime}(\g*\h)=
H(\Cfun_+(\widetilde{\g*\h}))\to E_1^{\prime\prime}(\widetilde\g)
\times E_1^{\prime\prime}(\widetilde\h)=
H\big(\Cfun_+(\tilde{\g})\big)\times H\big(\Cfun_+(\tilde{\h})\big)
\]
of Lemma~\ref{sec:append-cohom-free}.
The map in the middle is an isomorphism by Lemma \ref{nilp}, so the desired
result for finite dimensional nilpotent $\g$ and $\h$ follows.

Assume now that $\g = \lim_\alpha \g_\alpha$ and $\h = \lim_\beta
\h_\beta$ are limits of finite-dimensional nilpotent algebras.
Then the continuous duals equal $\g^* = \colim_\alpha \g_\alpha^*$,
$\h^* = \colim_\beta \h_\beta^*$
which readily implies that
\[
\Cfun_+(\g) = \colim_\alpha \Cfun_+(\g_\alpha)\  \mbox { and }\
\Cfun_+(\h) = \colim_\beta \Cfun_+(\h_\beta).
\]
By the standard properties of the filtered limits, one has
\begin{equation}
\label{zitra_na_Borovou_Ladu}
\g * \h \cong \lim_{\alpha,\beta,m} (\g_\alpha \CF \h_\beta)/
(\g_\alpha \CF \h_\beta)^{\geq m},
\end{equation}
where $\CF$ denotes, as in the proof of Lemma~\ref{nilp}, the
uncompleted free product. The isomorphism~(\ref{zitra_na_Borovou_Ladu})
represents $\g * \h$ as a limit of finite dimensional spaces, so
\begin{eqnarray*}
(\g * \h)^* &\cong& \colim_{\alpha,\beta,m}\big[ (\g_\alpha \CF \h_\beta)/
(\g_\alpha \CF \h_\beta)^{\geq m}\big]^*
\cong \colim_{\alpha,\beta} \colim_m \big[ (\g_\alpha \CF \h_\beta)/
(\g_\alpha \CF \h_\beta)^{\geq m}\big]^*
\\
&\cong& \colim_{\alpha,\beta}
(\g_\alpha * \h_\beta)^*
\end{eqnarray*}
which easily implies that
\[
\Cfun_+(\g * \h) = \colim_{\alpha,\beta} \Cfun_+(\g_\alpha * \h_\beta).
\]
Since the canonical map~(\ref{chprod}) is compatible with these
colimits and the colimits are exact functors, the general case follows
from the finite-dimensional nilpotent one.
\end{proof}

\begin{rem}
Note that the quasi-isomorphism of Theorem \ref{freeprod} can also be formulated as $\Cfun(\g*\h)\simeq\Cfun(\g)\times_{\ground}\Cfun(\h)$.
\end{rem}


\end{document}